\documentclass{amsart}

\usepackage{nomencl}
\makenomenclature

 \RequirePackage{ifthen}
 \renewcommand{\nomgroup}[1]{%
 \ifthenelse{\equal{#1}{L}}{\item[\textbf{List of notations}]}{%
 \ifthenelse{\equal{#1}{G}}{\item[\textbf{General conventions}]}{}}}

\usepackage{graphicx,stmaryrd}
\usepackage{dsfont}
\usepackage{mathrsfs}
\usepackage{mathtools}
\usepackage{algorithm}
\usepackage{algorithmic}

\newcommand{\abs}[1]{\vert#1\vert}
\newcommand{\av}[1]{\langle#1\rangle}

\newcommand{\Abs}[1]{\Vert#1\Vert}

\newcommand{\jump}[1]{\llbracket#1\rrbracket}

\newcommand{\dsp}{\displaystyle}
\newcommand{\bU}{{\mathbf U}}

\newcommand{\ovV}{\overline{V}}

\newcommand{\ovu}{\overline{u}}
\newcommand{\uU}{\underline{U}}

\newcommand{\uV}{{\underline V}}
\newcommand{\uw}{{\underline w}}
\newcommand{\uP}{{\underline P}}
\newcommand{\R}{{\mathbb R}}
\newcommand{\N}{{\mathbb N}}
\newcommand{\dt}{\partial_t}
\newcommand{\dz}{\partial_z}
\newcommand{\dx}{\partial_x}
\newcommand{\dy}{\partial_y}

\newcommand{\cC}{{\mathcal C}}
\newcommand{\cI}{{\mathcal I}}

\newcommand{\curl}{\mbox{\textnormal{curl} }}

\newcommand{\dive}{\mbox{\textnormal{div} }}

\newcommand{\eps}{\varepsilon}

\newcommand{\bom}{\boldsymbol{\omega}}
\newcommand{\bug}{{\bf U}_G}
\newcommand{\dbug}{\dot{\bf U}_G}

\newtheorem{proposition}{Proposition}
\newtheorem{corollary}{Corollary}
\newtheorem{lemma}{Lemma}
\theoremstyle{remark}
\newtheorem{remark}{Remark}

\newtheorem{definition}{Definition}
\newtheorem{notation}{Notation}

\title{On the dynamics of floating structures}

\author{David Lannes}
\address{Institut de Math\'ematiques de Bordeaux\\ Universit\'e de Bordeaux et CNRS UMR 5251\\ 351 Cours de la Lib\'eration \\33405 Talence Cedex, France}
\thanks{The author has been partially funded by the ANR-13-BS01-0003-01 DYFICOLTI and  the ANR-
13-BS01-0009-01 BOND }

\begin{document}

\maketitle

\begin{abstract}
This paper addresses the floating body problem which consists in studying the interaction of surface water waves with a floating body.  We propose a new formulation of the water waves problem that can easily be generalized in order to take into account the presence of a floating body. The resulting equations have a compressible-incompressible structure in which the interior pressure exerted by the fluid on the floating body is a Lagrange multiplier that can be determined through the resolution of a $d$-dimensional elliptic equation, where $d$ is the horizontal dimension.
In the case where the object is freely floating, we decompose the hydrodynamic force and torque exerted by the fluid on the solid in order to exhibit an added mass effect; in the one dimensional case $d=1$, the computations can be carried out explicitly.\\
We also show that this approach in which the interior pressure appears as a Lagrange multiplier can be implemented on reduced asymptotic models such as the nonlinear shallow water equations and the Boussinesq equations; we also show that it can  be transposed to the discrete version of these reduced models  and propose simple numerical schemes in the one dimensional case. We finally present several numerical computations based on these numerical schemes; in order to validate these computations we exhibit explicit solutions in some particular configurations such as the return to equilibrium problem in which an object is dropped from a non-equilibrium position in a fluid which is initially at rest.

\end{abstract}

\section{Introduction}

\subsection{General setting}

Krylov published in 1898 a method to compute the hydrodynamic loads for ship motions in waves, assuming that the presence of the ship did not perturb the waves, but the {\it floating body problem} was probably formulated by Fritz John in two celebrated papers \cite{John1,John2}. It consists in studying the motion of the mechanical system formed by a fluid  and a partially immersed solid ${\mathcal C}(t)$. The fluid is delimited above  by a free surface, and is assumed to be incompressible and in irrotational motion, while the solid ${\mathcal C}(t)$ can have a prescribed motion or can be freely floating. In the latter case, the motion of the solid is governed by Newton's laws in which the gravity force (and possibly other external forces) is complemented by the force and torque exerted by the liquid on the solid.

This is a complex problem in which two free boundary problems are involved. The first one is the standard water waves problem consisting in describing the evolution of the surface of the fluid when it is in contact with the air. The second free boundary problem comes from the fact that the {\it wetted surface} $\partial_{\rm w}{\mathcal C}(t)$, i.e. the portion of the boundary of the solid in contact with the fluid, depends on time. For these reasons, Fritz John considered a much simplified problem. Expressing the velocity $\bU$ in the fluid domain $\Omega$ in terms of a velocity potential $\Phi$,
$$
\bU=\nabla_{X,z}\Phi \quad \mbox{ and }\quad \Delta_{X,z}\Phi=0\quad  \mbox{ in } \quad\Omega, \quad \mbox{ and }\quad \partial_n \Phi=0\quad \mbox{ at the bottom,}
$$
he made the following assumptions
\begin{itemize}
\item A linear model for the evolution of the free surface waves is considered in the {\it exterior domain} (i.e. where the surface of the fluid is not in contact with the solid), namely
$$
\begin{cases}
\dt \zeta- (\dz \Phi)_{\vert_{z=0}}=0,\\
\dt \Phi_{\vert_{z=0}}+g\zeta=0,
\end{cases}
$$
where $g$ is the gravity and $\zeta$ the parametrization of the free surface above the rest state $z=0$.
\item The motion of the solid is assumed to be of small amplitude.
\item The variations of the wetted surface $\partial_{\rm w}{\mathcal C}(t)$ with time are neglected.
\end{itemize}
On the {\it interior domain} (i.e. under the structure),  the continuity of the normal velocity across $\partial_{\rm w}{\mathcal C}$, yields the additional condition
$$
\partial_n \Phi=\underline{U}_{{\rm w}}\cdot n
$$
where $\underline{U}_{{\rm w}}$  is the velocity of the solid on the wetted surface and $n$ the upward unit normal vector; when the solid is in forced motion, this is a known function of time, and when the solid is freely floating it must be deduced from the Newton's laws that govern the motion of the solid. In the latter case, it is necessary to know the pressure  exerted by the fluid on the bottom of the boat (called the {\it interior pressure} $\uP_{\rm i}$); this is done in \cite{John1} using the linearized Bernoulli equation,
$$
-\frac{\uP_{\rm i}-P_{\rm atm}}{\rho}=(\dt \Phi)_{\vert_{z=\zeta_{\rm w}}}+g\zeta_{\rm w},
$$
where $P_{\rm atm}$ is the atmospheric pressure, and  $\zeta_{\rm w}$ the parametrization of the bottom of the solid.\\
Finally, some transition conditions are needed at the {\it contact line} that separates the interior and exterior domains.  In \cite{John1}, these conditions are not stated clearly and not completely correct; as we shall see, this is mainly because the velocity potential $\Phi$ is not the appropriate quantity to express such transition conditions. In \cite{John2}, it is further assumed that the motion in time is harmonic at some given frequency, so that the full problem reduces to a spectral problem, in which the main difficulty becomes the analysis and/or numerical computation of the associated Green functions.

Fritz John's approach of the floating body problem, though oversimplified in many aspects (it misses in particular the nonlinear effects, the evolution of the wetted surface, etc.), is still used and studied a lot, both theoretically and numerically. It has been slightly generalized to include second order effects \cite{Ogilvie} (though still neglecting the time variations of the wetted surface) and is still the principal method used in the extensive  literature devoted to floating structures such as wave power devices for instance  \cite{MMc,LN}; it is also the basis of softwares like WAMIT, widely used to compute the motion of offshore structures in waves.

More recently, the nonlinear effects in the floating body problem have been taken into account in various numerical studies, mostly based on boundary element methods for the resolution of the potential equation (see for instance  the review \cite{Folley}). The nonlinear aspect of the underlying hydrodynamics is taken into account by a nonlinear boundary element method (see \cite{Grilli_etc,HagueSwan} for instance), and the hydrodynamic forces on the wetted surface can be computed at each time (see for instance \cite{Kashiwagi}), which allows the description of the  evolution of the contact line. These methods require the resolution of boundary integral equations and have a big computational cost. This is also the case of the CFD approach based on the numerical resolution of the full Navier-Stokes equations (see \cite{ParoliniQuarteroni} and references therein). 

All these methods have in common that they require the resolution of a $(d+1)$-dimensional elliptic problem in the fluid domain ($d$ is the horizontal dimension), or a boundary integral equation, in order to compute the interior pressure $\uP_{\rm i}$ through Bernoulli's equation as explained above; moreover, the presence of the time derivative of the velocity potential in this expression yields considerable stability issues in the numerical simulations \cite{Kashiwagi}. 

In this paper, we propose a different approach than the one initiated by F. John and in particular, we no longer seek to recover the interior pressure $\uP_{\rm i}$ through Bernoulli's equation. More precisely, we propose a new formulation of the full (nonlinear) floating body problem in which
\begin{itemize}
\item The transition conditions at the contact line can be expressed in a simple way and the evolution law for the contact line can be derived.
\item The problem is stated as a $d$-dimensional compressible-incompressible model in which the interior pressure $\uP_{\rm i}$ is found as the Lagrange multiplier associated to the constraint that the surface of the fluid coincides with the boundary of the solid under the floating body (i.e. $\zeta=\zeta_{\rm w}$).
\end{itemize}
The interest of this formulation, itself based on a new formulation of the standard water waves equations in terms of $(\zeta,Q)$, where $Q$ is the horizontal discharge, is that the dimensionality of the elliptic equation one has to solve to find $\uP_{\rm i}$ is reduced: it is now a simple $d$-dimensional elliptic equation (as opposed to the $d+1$ elliptic equation on the potential one has to solve in the approach described above). 
Replacing $\uP_{\rm i}$ by the solution of this elliptic equation, one can moreover eliminate the constraint $\zeta=\zeta_{\rm w}$ in the interior region, exactly in the same way as the incompressible Euler equations can be transformed into an unconstrained quasilinear evolution equation on the velocity.
Note also that the compressible-incompressible structure mentioned above is typical of congested flows that appear in several contexts such as two-phase flows \cite{BBCR,BPZ,PZ}, traffic jams \cite{BDMR},  formation of crowds \cite{DH}, granular flows \cite{LLM,Perrin}, compressible-low Mach coupling in gaz dynamics \cite{PDD}, etc. 

\medbreak

We also want in this paper  to take advantage, with this new formulation of the floating body problem, of the progresses that have been made in recent years in the mathematical study of the motion of a rigid body ${\mathcal C}(t)$ totally immersed in an incompressible perfect fluid confined to a domain $\Omega$. This is also a problem that has attracted a lot of attention, starting with the works of d'Alembert, Kelvin and Kirchoff. The equations governing the motion are provided by the Euler equations for the dynamics of the fluid in the region $\Omega\backslash {\mathcal C}(t)$ outside the solid, often (but not necessarily) complemented with an irrotationality assumption. The existence and uniqueness of classical solutions to this problem has been proved in \cite{ORT,RR,HMT}; in \cite{GST} the authors used the added mass effect to prove that the regularity of the motion of the solid is limited only by the regularity of the boundary of the solid. Roughly speaking, the added mass effect consists in the fact that some components of the hydrodynamics force and torque applied on the solid act as if the mass-inertia matrix in Newton's law were modified by the addition of a positive matrix. This is because a rigid body has to accelerate not only itself but also the fluid around it. Exploiting this effect is necessary for a sharp mathematical analysis of the equations \cite{GST,GMS} and plays also a crucial role for the stability of numerical simulations in many fluid-structure interaction problems \cite{CausinGerbeauNobile}. This added-mass effect can be quite complex however, since it   depends strongly on the location of the solid with respect to the boundaries of the fluid domain \cite{GMS}; in the case a floating body considered here, the analysis is complicated by the fact that the boundary of the fluid domain is a free surface, which moreover intersects the surface of the body. The second goal of the paper is therefore to
\begin{itemize}
\item Exhibit the added-mass effect in our compressible-incompressible formulation of the floating body problem
\item Take advantage of the simplicity of the elliptic equation on the interior pressure $\uP_{\rm i}$ to get a simple expression of the mass-inertia matrix (which becomes explicit in $1+1$ dimension).
\end{itemize}

Of course, the resulting formulation of the floating body problem remains quite complex. From the mathematical viewpoint, proving a local well-posedness result is a very challenging issue since, not speaking of the coupling with the solid motion, it requires several results on the water waves equations that are important open problems. For instance, the regularity of the surface on the whole domain is not expected to be better than Lipschitz because there is an angle/wedge at the contact line\footnote{At the day, the best result in terms of low regularity for the surface elevation in the water waves equations is $H^{3/2+d/2-\eps}(\R^d)$, for some $\eps>0$ explicit \cite{ABZ}.}, there is no result on the mixed initial-boundary value problem for the water waves equations, etc. The numerical simulation of the full water waves equations is also quite demanding. For these reasons, and with the goal of being able to study numerically real wave-structure interactions, and in particular nonlinear effects (efforts on offshore platforms in extreme events, wave energy converters, etc.), one is led to derive simplified asymptotic models. We shall consider here the case of shallow water configurations for which the asymptotics of the water-waves equations (without floating body) is now well understood \cite{AL1,Iguchi,L_book}. There are however only a few references that extend the resulting asymptotic models in the presence of a floating body. In \cite{KKK} the authors used a Boussinesq model to describe the flow under the free surface, while solving the potential equation for $\Phi$ under the floating body (from which the interior pressure $\uP_{\rm i}$ is recovered along the lines described above). Closer to our approach, \cite{Jiang} and \cite{Mario} propose a system of two Boussinesq systems (one under the structure, and the other one under the free surface), and the interior pressure is numerically solved so that these two sets of equations are compatible; the formulation used in these references does not however allow to write a simple explicit elliptic equation on the interior pressure as in the approach we propose here. The third goal of this paper is therefore
\begin{itemize}
\item To use the strategy explained above (in the case of the full water waves equations) in order to allow for the presence of a floating structure in various shallow water models --- we consider here the nonlinear shallow water equations and a Boussinesq system. More precisely, we show that in the presence of a floating body, these models can be written under the form of a compressible-incompressible system. To every model corresponds a particular Lagrange multiplier and therefore a particular  interior pressure $\uP_{\rm i}$.
 \item To generalize this approach to numerical schemes; we show in particular how to find a discretization of the interior pressure in such a way that it plays the role of a discrete Lagrange multiplier. 
 \item Show the efficiency of this method with some numerical computations for the one dimensional nonlinear shallow water and Boussinesq models.
\end{itemize}
In these numerical computations, the fact that the interior pressure  is the discrete Lagrange multiplier associated to the constraint that the surface of the fluid under the floating structure coincides with the boundary of this latter allows us  to solve the equations in the full computational domain (without having to handle the coupling between the interior and exterior regions); the surface elevation computed in this way coincides at machine precision with the bottom of the solid in the wetted region. Under the assumptions described above, the floating body problem can therefore be solved numerically very efficiently.

\subsection{Organization of the paper}

We first describe in Section \ref{sectWWfloat} how the waves are affected by the presence of a floating structure, without considering the motion of the solid itself. The formulation of the equations is first given in \S \ref{sectFSEF}; it follows from this formulation that the horizontal discharge (or the vertically averaged vertical velocity $\ovV$) is a natural quantity to express the transition conditions at the contact line. We therefore seek in \S \ref{sectFSEzV} a formulation of the water waves equations in terms of this variable (and of the surface elevation). After proving that such a formulation exists and is closed (i.e. that all the physical quantities involved can be reconstructed in terms of the horizontal discharge and of the surface elevation), we generalize this formulation in \S \ref{sectFSEzVfloat}  in the presence of a floating structure. This formulation has a compressible-incompressible structure: it is compressible in the exterior region, and incompressible under the floating structure. The interior pressure $\uP_{\rm i}$ naturally appears as the Lagrange multiplier associated to the "incompressibility" condition $\zeta=\zeta_{\rm w}$, and it can be found by solving a simple $d$-dimensional elliptic equation.\\
In Section \ref{sectsolid}, the motion of the floating structure is considered. We first consider in \S \ref{sectpresc} the case of a solid with a forced motion, while the case of a freely floating solid is studied in \S \ref{sectfloat}. In the latter case, the motion of the solid is found through Newton's laws where the force corresponding to the interior pressure $\uP_{\rm i}$ is the buoyancy force; this force is decomposed into several components, one of which corresponding to an added mass effect. Let us mention that in both cases (forced motion and freely floating body), specific attention is paid to the one-dimensional case: the elliptic equation for the interior pressure $\uP_{\rm i}$ is then one-dimensional and can be solved explicitly.\\
The evolution of the contact line is then studied in Section \ref{sectcontact}, in the one dimensional case in \S \ref{sectcontact1d}, and in the two-dimensional case in \S \ref{sectcontact2d}. We also explain in \S \ref{sectvertical} the modifications one has to carry out when the boundaries of the floating structure are vertical at the contact line.\\
In Section \ref{sectAsfloat}, we replace the water waves equations for the free surface by simpler asymptotic models. The case of the nonlinear shallow water equations is considered in \S \ref{sectSW}, while the Boussinesq equations are treated in \S \ref{sectBouss}.\\
We then show in Section \ref{sectdiscrete} how to implement our approach at the level of the numerical scheme. To this end, we consider a simple one dimensional configuration in which the solid is only allowed to move vertically and has vertical lateral walls (the contact points are then independent of time). We show how to discretize the interior pressure in such a way that it plays the role of a discrete Lagrange multiplier for the numerical scheme. The equations are presented in \S \ref{sect_modSW} when the hydrodynamic model is the nonlinear shallow water equations. Particular attention is paid to the ordinary differential equation resulting from Newton's law. We are in particular able to find a simple nonlinear second order ODE governing the motion of the solid in the return to equilibrium problem (the solid is dropped from an out of equilibrium position in a fluid initially at rest).
The numerical scheme is then presented and studied in \S \ref{sectnumSW}, and this approach is extended in \S \ref{sect_modBouss} when the underlying hydrodynamic model is the Boussinesq system. \\
The numerical computations based on these schemes are then presented in Section \ref{sectcomput}. For the nonlinear shallow water equations, several configurations are considered in \S \ref{sectnumsimSW}: a fixed solid, a solid in prescribed motion, and a freely floating solid. In the last two cases we can derive formulas for explicit solutions for some configurations and use them to validate our numerical simulations. Numerical simulations when the hydrodynamical model is the Boussinesq system are then presented in \S \ref{sectnumBouss}.

Finally, several results are postponed to Appendices. In Appendix \ref{Appinteriorpressure} we derive an alternative equation for the interior pressure, while the equations of motion for the solid structure are given in Appendix \ref{appbodyframe} in a body frame instead of the Eulerian frame.

\subsection{Notations}

We just introduce here some basic notations; {\it a  full table of notations is provided at the end of the paper.}\\
- We denote by $d=1,2$ the horizontal dimension.\\
- The gradient operator with respect to the horizontal variables $X\in \R^d$ is denoted by $\nabla$; the full $(d+1)$-dimensional gradient operator is denoted $\nabla_{X,z}$, where $z$ is the vertical variable.\\
- If ${\bf A}\in \R^{d+1}$, we denote by ${\bf A}_{\rm h}\in \R^d$ its horizontal components, and by ${A}_{\rm v}$ its vertical (last) component.\\
- We denote by ${\bf e}_z$ the unit upward vertical vector, and by ${\bf e}_x$ and ${\bf e}_y$ the unit vectors in the horizontal directions $x$ and $y$.\\
- When $d=2$, we write $X=(x,y)$ and sometimes use the notation $\partial_1=\dx$, $\partial_2=\dy$.\\
- We denote with single vertical bars $\vert \cdot \vert$ norms over the horizontal plane $\R^2$, and with a double bar $\Vert \cdot \Vert$ norms over the fluid domain $\Omega$. For instance,
$$
\abs{f}_2=\Big(\int_{\R^d}\abs{f}^2\Big)^{1/2}\quad\mbox{ and }\quad
\Abs{F}_2=\Big(\int_{\R^d}\Abs{f}^2\Big)^{1/2}.
$$

\section{Water waves and floating structures}\label{sectWWfloat}

\subsection{The free surface Euler equations with a floating structure}\label{sectFSEF}

Let us consider here the dynamics of the waves in the presence of a
partially immersed device (typically a ship or a floating wave energy converter). Denoting by $\cC(t)$ the
volume occupied by the (solid) device at time $t$, we write $\partial\cC(t)$ its
boundary and $\partial_{\rm w}\cC(t)$ the {\it wetted surface}, that is, the
portion of $\partial\cC(t)$ in contact with the water, and by ${\mathcal
  I}(t)\subset \R^d$ ($d$ being the horizontal dimension) its projection on the horizontal plane, which we shall
refer to as the {\it interior domain}. The {\it exterior domain} ${\mathcal E}(t)$ is then naturally defined as
$$
{\mathcal E}(t)=\R^d\backslash \overline{{\mathcal I}(t)}.
$$
We consider in
this paper the  case where overhanging waves do not occur and where the wetted surface can be
parametrized by a graph of some function $\zeta_{\rm w}(t,X)$, for all
$X\in {\mathcal I}(t)$. The surface of the water is therefore
determined by the graph of a function $X\in \R^d\mapsto \zeta(t,X)$ satisfying the constraint  $\zeta(t,X)=\zeta_{\rm w}(t,X)$ on  ${\mathcal I}(t)$. Denoting by $h_0$ the typical depth at rest and by $-h_0+b(X)$ a parametrization of the bottom, the domain $\Omega(t)$ occupied by the
fluid at time $t$ is therefore given by
$$
\Omega(t)=\{(X,z)\in \R^{d+1}, -h_0+b(x)<z<\zeta(t,X)\}.
$$
\begin{figure}
\includegraphics[width=\textwidth]{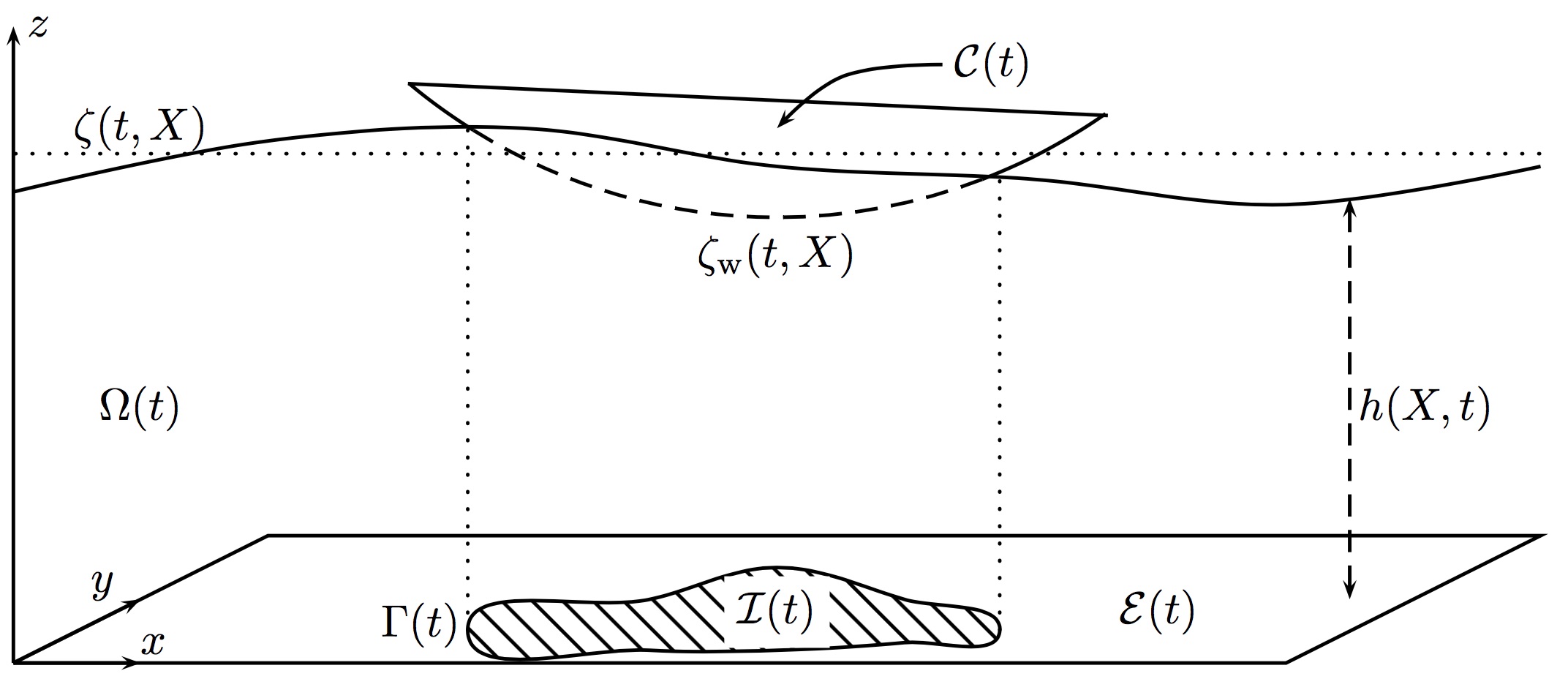}
\end{figure}
\begin{notation}\label{notaintext}
For any function $f$ defined on $\R^d$, we denote with a subscript ${\rm i}$ its restriction to the interior domain ${\mathcal I}(t)$ and with a subscript ${\rm e}$ its restriction to the exterior domain ${\mathcal E}(t)$,
$$
f_{\rm i}=f_{\vert_{{\mathcal I}(t)}}\quad\mbox{ and }\quad f_{\rm e}=f_{\vert_{{\mathcal E}(t)}}.
$$
\end{notation}

We assume that the flow is incompressible, irrotational, of constant density $\rho$, and inviscid. We can then formulate the equations as  a set of equations in $\Omega(t)$, complemented with boundary conditions and a constraint associated to the presence of the immersed structure:
\begin{itemize}
\item {\it Equations in the fluid domain $\Omega(t)$}. Denoting by ${\bf U}$ and $P$ the velocity and pressure fields, the equations  are given by
\begin{eqnarray}
\label{Eul1}
\partial_t {\bf U}+{\bf U}\cdot \nabla_{X,z}{\bf U}&=&-\frac{1}{\rho}\nabla_{X,z}P-g{\bf e}_z\\
\label{Eul2}
\dive \bU&=&0,\\
\label{Eul3}
\curl \bU&=&0,
\end{eqnarray}
where $g$ is the acceleration of gravity and $\rho$ the constant density of the water. 
\item {\it Boundary conditions at the surface}. The surface being bounding (i.e. no fluid particle crosses it), one gets the traditional kinematic equation
\begin{eqnarray}
\label{Eul5}
\dt \zeta-\uU\cdot N&=&0\quad \quad \mbox{ with }\quad N=\left(\begin{array}{c} -\nabla\zeta \\1 \end{array}\right) ,
\end{eqnarray}
where we denoted $\uU(t,X)=\bU(t,X,\zeta(t,X))$ the trace of the the velocity field $\bU$ at the free surface.\\
 The assumption that the pressure is given by the constant atmospheric pressure $P_{\rm atm}$ on the exterior domain (i.e. on the portion of the surface of the fluid that is not in contact with the immersed structure), gives if we write $\underline{P}=P_{\vert_{z=\zeta}}$ and with Notation \ref{notaintext},
\begin{equation}
\label{Eul4}
\underline{P}_{\rm e}=P_{\rm atm}.
\end{equation}
\item {\it Boundary condition at the bottom}.
Assuming that the bottom is impermeable, we get another boundary condition at the bottom
\begin{equation}\label{Eul6}
U_b\cdot N_b=0  \quad \mbox{ with }\quad N_b=\left(\begin{array}{c} -\nabla b \\1 \end{array}\right) ,
\end{equation}
where we denoted by $U_b$ the trace of $\bU$ at the bottom.
\item {\it Constraint in the interior domain}. By definition, the surface of the fluid coincides with the bottom of the solid structure in the interior domain; according to Notation \ref{notaintext}, this yields
\begin{equation}\label{constraint}
 \zeta_{\rm i}=\zeta_{\rm w}.
\end{equation}
\end{itemize}
It is important to insist on the fact that the {\it interior pressure} $\underline{P}_{\rm i}$ 
is not known and must be determined from the above equations. Similarly, the interior and exterior domains ${\mathcal I}(t)$ and ${\mathcal E}(t)$ are also unknowns of the problem that we must determine. To this end, we need another set of boundary, or {\it transition}, conditions at the contact line (defined as the part of the bottom of the boat which is in contact both with air and water). Let us first give some notations.
\begin{notation}
We denote by $\Gamma(t):=\partial {\mathcal I}(t)=\partial {\mathcal E}(t)$ the projection of the contact line on the horizontal plane 
\end{notation}
It is implicitly assumed that $\zeta_{\rm i}$, $\underline{P}_{\rm i}$, etc. (resp.  $\zeta_{\rm e}$, $\uP_{\rm e}$, etc.) are smooth in ${\mathcal I}(t)$ (resp. ${\mathcal E(t)}$) and that they can be extended by continuity on the closure of these domains; however, they are certainly not smooth on the whole horizontal plane $\R^d$. We only have the following boundary conditions at the contact line:
\begin{itemize}
\item {\it Continuity of the surface elevation}. There is no jump of the surface elevation at the contact line,
\begin{equation}\label{contact3}
\zeta_{\rm e}(t,\cdot)= \zeta_{\rm i}(t,\cdot) \quad\mbox{ on }\quad \Gamma(t).
\end{equation}
\item {\it Continuity of the pressure at the contact line}. We assume that
\begin{equation}\label{contact1}
\uP_{\rm i}(t,\cdot)=P_{\rm atm} \quad\mbox{ on }\quad \Gamma(t).
\end{equation}
\end{itemize}

The equations \eqref{Eul1}-\eqref{Eul6} together with the constraint \eqref{constraint} and the boundary conditions at the contact line \eqref{contact3}-\eqref{contact1} form the {\it free surface Euler equation in the presence of a floating body ${\mathcal C}(t)$}, that we can also see as   {\it constrained free surface Euler equations}.
\begin{remark}
The transition conditions \eqref{contact3}-\eqref{contact1} obviously exclude the configuration where the boundaries of the solid are vertical at the contact line. We show in \S \ref{sectvertical} how to handle such configurations and generalize \eqref{contact3}-\eqref{contact1}.
\end{remark}
Of course, further information is needed on the motion of the solid: it can be fixed, in prescribed motion, or freely floating for instance. These situations will be addressed in Section \ref{sectsolid}.

\subsection{A formulation of the classical water waves equation in $(\zeta,Q)$ variables}\label{sectFSEzV}

Taking ${\mathcal I}(t)=\emptyset$ and ${\mathcal E}(t)=\R^d$, equations \eqref{Eul1}-\eqref{Eul6} are the classical (i.e. without any floating structure) free surface Euler equations, also called water-waves equations. These equations are cast on the $(d+1)$- dimensional domain $\Omega(t)$ which is itself unknown. \\
Several reformulations of these equations have been proposed in order to work with a set of equations on a fixed domain. Among these reformulations, one of the most popular is the Zakharov-Craig-Sulem formulation \cite{Zakharov, CSS}, which is a set of two scalar equations on $\zeta$ and on $\psi$, the trace of the velocity potential at the surface. The dimension reduction of this formulation is one of its main features: $\zeta$ and $\psi$ depend only on the horizontal space variables, so that the $z$ dependency has  been removed.

Working with $\zeta$ and $\psi$ is therefore quite usual when analyzing the water waves problem, but the asymptotic models that are used for applications in oceanography are generally not cast in terms of $\zeta$ and $\psi$. For these models, a dimension reduction is also done to eliminate the vertical variable $z$, but this reduction is performed using a different procedure, namely, {\it vertical integration}. Consequently, the asymptotic models (such as the nonlinear shallow water equations, the Serre-Green-Naghdi equations, etc.) are cast in terms of $\zeta$ and $Q$, where $Q$ is the horizontal discharge defined as follows.
\begin{notation}
We denote by  $Q$ the horizontal discharge defined as 
$$
Q(t,X):=\int_{-h_0+b(X)}^{\zeta(t,X)}V(t,X,z)dz,
$$
where $V$ is the horizontal component of the velocity field $\bU$.
\end{notation}
 Such a formulation in $(\zeta,Q)$ variables is also much more adapted than the classical $(\zeta,\psi)$ formulation to handle the transition conditions  at the contact line. The aim of this section is therefore to derive a new formulation of the full water waves equations in terms of $\zeta$ and $Q$. 

\subsubsection{The integrated Euler equations}

Integrating along the vertical variable $z$ the momentum equation \eqref{Eul1}, it is well known  \cite{Teshukov,CL1} that, in absence of any immersed solid, one can derive a set of equations coupling the water depth $h$ to the vertically integrated horizontal component of the velocity $Q$ and given by
\begin{equation}\label{Eulerav1}
\begin{cases}
\dsp \dt \zeta+\nabla\cdot Q=0,\\
\dsp \dt Q+\nabla\cdot ( \int_{-h_0+b}^\zeta V\otimes
V)+gh\nabla \zeta+\frac{1}{\rho}\int_{-h_0+b}^\zeta \nabla P_{\rm NH}=0,
\end{cases}
\end{equation}
where the non hydrostatic pressure $P_{\rm NH}$ is given by
\begin{equation}\label{PNH}
P_{\rm NH}(t,X,z)=\rho \int_{z}^{\zeta(t,X)}\big(\dt w+{\bf U}\cdot
\nabla_{X,z}w\big).
\end{equation}
 In \eqref{Eulerav1}-\eqref{PNH} however, several quantities are not explicit functions of $\zeta$ and $Q$; it is therefore necessary to prove that the full velocity field $\bU$ in $\Omega(t)$ can be recovered from the knowledge of $\zeta$ and $Q$. In the next
 section, the technical tools for such a reconstruction are provided.
 
 \subsubsection{The average and reconstruction mappings}
 
 It is convenient to introduce here the Dirichlet-Neumann operator which plays a central role in the Zakharov-Craig-Sulem formulation. We recall first that the Beppo-Levi spaces $\dot{H}^s(\R^d)$ and $\dot{H}^1(\Omega)$ are defined for all $s\geq 0$ by
 \begin{align}
\label{BL1} \dot{H}^s(\R^d)&=\{f\in L^2_{\rm loc}(\R^d),\quad\nabla f\in H^{s-1}(\R^d)^d \}\\
 \nonumber \dot{H}^1(\Omega)&=\{f\in L^2_{\rm loc}(\Omega),\quad \nabla_{X,z} f \in L^2(\Omega)^{d+1}\}
 \end{align}
 and are endowed with the (semi) norms $\abs{f}_{\dot{H}^s}=\abs{\nabla f}_{H^{s-1}}$ and $\Abs{f}_{\dot{H}^1}=\Abs{\nabla_{X,z}f}_2$ respectively. Note that the fact that the following definition makes sense stems from Proposition 3.3 in \cite{L_book}.
 \begin{definition}\label{defDN}
Let $\zeta,b\in W^{1,\infty}(\R^d)$ be such that $\inf_{\R^d}(h_0+\zeta(X)-b(X))>0$. The Dirichlet-Neumann operator $G[\zeta ]$ is defined as
$$
G[\zeta ]:\begin{array}{lcl}
\dot{H}^{1/2}(\R^d)&\to&H^{-1/2}(\R^d)\\
\psi&\mapsto&\sqrt{1+\abs{\nabla\zeta}^2}\partial_n \Phi_{\vert_{z=\zeta}}
\end{array}
$$ 
where $\Phi\in \dot{H}^1(\Omega)$ is the variational solution of the boundary value problem
$$
\left\lbrace
\begin{array}{l}
\Delta_{X,z}\Phi=0\quad \mbox{ in }\quad \Omega,\\
\Phi_{\vert_{z=\zeta}}=\psi,\qquad \partial_n\Phi_{\vert_{z=\zeta}}=0.
\end{array}\right.
$$
\end{definition}
We can now state the following proposition that shows that the velocity field $\bU$ can be reconstructed from $\zeta$ and $\ovV$,
where $\ovV$ is the vertically averaged horizontal component of the velocity  defined as
\begin{equation}\label{defavV}
\ovV(t,X)=\frac{1}{h(t,X)}\int_{-h_0+b(X)}^{\zeta(t,X)}V(t,X,z)dz\quad \mbox{ with }\quad h=h_0+\zeta-b,
\end{equation}
and where we recall that the velocity in the fluid domain is $\bU=(V,w)$.  Quite obviously, $Q$ and $\ovV$ are related through the identity
$$
Q=h\ovV  \quad \mbox{ with }\quad h=h_0+\zeta-b,
$$
so that the proposition also implies that one can reconstruct $\bU$ from $\zeta$ and $Q$. In the statement, the notation  $L^2_b(\Omega,\dive\!\!,\curl\!\!)$ is used for the set of admissible velocity fields,
$$
L^2_b(\Omega,\dive\!\!,\curl\!\!):=\{\bU\in L^2(\Omega)^{d+1},\dive\bU=0, \curl \bU=0 \mbox{ \textnormal{and} }U_b\cdot N_b=0\}.
$$

\begin{proposition}\label{propclosed}
Let $\zeta,b\in W^{1,\infty}(\R^d)$ be such that $\inf_{\R^d}(h_0+\zeta(X)-b(X))>0$. \\
The average mapping 
$$
{\mathfrak A}[\zeta ]: \begin{array}{lcl}
L^2_b(\Omega,\dive\!\!,\curl\!\!)&\to & H^{1/2}(\R^d)^d\\
\bU:=\left(\begin{array}{c} V\\ w \end{array}\right) &\mapsto & \ovV:=\frac{1}{h}\int_{-h_0+b}^\zeta V
\end{array}
$$
and the reconstruction mapping
$$
{\mathfrak R}[\zeta ]: \begin{array}{lcl}
H^{1/2}(\R^d)^d&\to & L^2_b(\Omega,\dive\!\!,\curl\!\!)  \\
\ovV &\mapsto & \nabla_{X,z}\Phi,
\end{array}
\quad \mbox{ with }\quad
\begin{cases}
\Delta_{X,z}\Phi=0\mbox{ in }\Omega\\
\Phi_{\vert_{z=\zeta}}=-G[\zeta ]^{-1}\big(\nabla\cdot (h\ovV)\big)\\
\partial_n\Phi_{\vert_{z=-h_0+b}}=0
\end{cases}
$$
are well defined and ${\mathfrak R}[\zeta]$ is a right-inverse of ${\mathfrak A}[\zeta]$.
\end{proposition}
\begin{proof}
In order to prove that the average mapping is well defined, we just need to prove that $\ovV$ belongs to $H^{1/2}(\R^d)$ if $\bU$ is in 
$L^2_b(\Omega,\dive\!\!,\curl\!\!) $. This is done in the following lemma.
\begin{lemma}
Let $\zeta,b \in W^{1,\infty}(\R^d)$ and $\bU=(V,w)\in L^2_b(\Omega,\dive\!\!,\curl\!\!)$. Then one has
$$
\ovV\in H^{1/2}(\R^d)^d,\quad\mbox{ with }\quad \ovV=\frac{1}{h}\int_{-h_0+b}^\zeta V(X,z)dz.
$$
\end{lemma}
\begin{proof}[Proof of the lemma]
Let us define $\sigma(X,z)=\frac{1}{h_0}(\zeta(X)-b(X))z+\zeta(X)$. Denoting also ${\mathcal U}=({\mathcal V},{{\scriptstyle\mathcal W}})$, with ${\mathcal U}(X,z)=\bU(X,z+\sigma(X,z))$, one has
$$
\ovV=\frac{1}{h_0}\int_{-h_0}^0{\mathcal V}(X,z)dz.
$$
We then define $\widetilde{V}$ on the strip ${\mathcal S}=\R^d\times (-h_0,0)$ by
$$
\forall (X,z)\in {\mathcal S},\qquad \widetilde{V}(X,z)=\frac{1}{h_0}\chi(z\abs{D})\int_{-h_0}^{z}{\mathcal V}(X,z)dz,
$$
where $\chi:\R\to \R$ is a smooth, even, function that is compactly supported and equal to $1$ in a neighborhood of the origin (therefore, for $z<0$, $\chi(z\abs{D})$ is a smoothing operator). One readily remarks that $\ovV=\widetilde{V}_{\vert{z=0}}$, so that the result follows from the trace theorem if we can establish that $\widetilde{V}\in H^1({\mathcal S})^d$. \\
Since ${\mathcal S}$ is bounded in the vertical direction and since $\widetilde{V}$ vanishes at the bottom, it is enough by the Poincar\'e inequality to prove that all the components of $\nabla_{X,z} \widetilde{V}$ are in $ L^2({\mathcal S})$. The strategy of the proof is as follows: first, we prove that $\nabla\cdot \widetilde V$ and $\nabla^\perp\cdot \widetilde V$ are in $L^2({\mathcal S})$, which implies that all the horizontal derivatives of $\widetilde V$ are in $L^2$. We then prove that $\dz \widetilde{V}$ is also in $L^2({\mathcal S})$.\\
- {\it Control of $\nabla\cdot \widetilde V$.} From the definition of $\widetilde{V}$, one computes
\begin{align*}
\nabla\cdot \widetilde{V}=\frac{1}{h_0}\chi(z\abs{D})\int_{-h_0}^{z}\nabla\cdot {\mathcal V}(X,z)dz.
\end{align*}
We also know that $\bU$ is a divergence free vector field; after the change of variable $z\mapsto z+\sigma(X,z)$, this yields
$$
\nabla^\sigma \cdot {\mathcal V}+\frac{h_0}{h}\dz {\scriptstyle\mathcal W}=0,
\quad\mbox{ where }\quad
\nabla^\sigma=\nabla -\frac{h_0}{h}\nabla\sigma\dz,
$$
so that
$$
\nabla\cdot {\mathcal V}=\frac{h_0}{h}\nabla\sigma \cdot \dz{\mathcal V}-\frac{h_0}{h}\dz {\scriptstyle\mathcal W}.
$$
Plugging this expression into the above integral and integrating by parts, we obtain
\begin{align*}
h_0\nabla\cdot \widetilde{V}&=\chi(z\abs{D})\big[\frac{h_0}{h}\int_{-h_0}^{z}\big(\nabla\sigma\cdot \dz {\mathcal V}-\dz{\scriptstyle\mathcal W}\big)dz\big]\\
&=\chi(z\abs{D})\big[-\frac{1}{h}\int_{-h_0}^{z}\nabla h\cdot {\mathcal V}
+\frac{h_0}{h}\big(\nabla\sigma\cdot {\mathcal V}-{\scriptstyle\mathcal W}+U_b\cdot N_b\big)\big],
\end{align*}
where we used the fact that ${\mathcal U}_{\vert_{z=-h_0}}=\bU_{\vert_{z=-h_0+b}}=U_b$. Since by assumption $U_b\cdot N_b=0$, this yields
$$
h_0\nabla\cdot \widetilde{V}
=
\chi(z\abs{D})\big[-\frac{1}{h}\int_{-h_0}^{z}\nabla h\cdot {\mathcal V}
+\frac{h_0}{h}\big(\nabla\sigma\cdot {\mathcal V}-{\scriptstyle\mathcal W}\big)\big].
$$
Since ${\mathcal V}\in L^2({\mathcal S})$ and $h,\sigma\in W^{1,\infty}({\mathcal S})$, this implies easily that $\nabla\cdot \widetilde V\in L^2({\mathcal S})$ (we did not use the presence of the smoothing operator $\chi(z\abs{D})$ here).\\
- {\it Control of $\nabla^\perp\cdot \widetilde V$.} Since $\bU$ is irrotational, one has $\nabla^\perp\cdot V=0$; after the same change of variables as above, this yields $(\nabla^\sigma)^\perp\cdot {\mathcal V}=0$, or equivalently
$$
\nabla^\perp\cdot {\mathcal V}=\frac{h_0}{h}\nabla^\perp \sigma \cdot \dz {\mathcal V}.
$$ 
Proceeding as for the previous step, we deduce that
$$
h_0\nabla^\perp\cdot \widetilde{V}
=
\chi(z\abs{D})\big[-\frac{1}{h}\int_{-h_0}^{z}\nabla^\perp h\cdot {\mathcal V}
+\frac{h_0}{h}\nabla^\perp\sigma\cdot {\mathcal V}-\frac{h_0}{h}\nabla^\perp b\cdot V_b\big],
$$
with $V_b=V_{\vert_{z=-h_0+b}}$. We can proceed as above for the first two components of the bracket, so that the only thing that remains to prove is that the bottom contribution is in $L^2$, namely, that
$$
\chi(z\abs{D})\big[\nabla^\perp b\cdot V_b\big] \in L^2({\mathcal S})
$$
(we removed the factor $h_0/h$ since it belongs to $W^{1,\infty}({\mathcal S})$ and therefore plays no role for this regularity claim). Using the smoothing properties of Poisson kernels (see Lemma 2.20 in \cite{L_book} for instance), it is enough to prove that $\nabla^\perp b\cdot V_b\in H^{-1/2}(\R^d)$, which is a classical consequence of the fact that $\bU\in L^2(\Omega)$ is curl-free.\\
- {\it Control of $\widetilde{V}$} in $L^2((-h_0,0);H^1(\R^d))$. This follows directly from the previous two points. Note that the statement remains true if $\chi$ is replaced by $\chi'$ in the definition of $\widetilde{V}$.\\
- {\it Control of $\dz \widetilde V$.} One directly gets from the definition of $\widetilde{V}$ that
\begin{align*}
h_0\dz \widetilde{V}=\abs{D}\Big(\chi'(z\abs{D})\int_{-h_0}^{z} {\mathcal V}(X,z)dz\Big)+\chi(z\abs{D}) {\mathcal V}.
\end{align*}
The first term in the r.-h.-s. belongs to $L^2({\mathcal S})$ thanks to the previous point, while the second term is trivially  in $L^2({\mathcal S})$. This proves the claim and concludes the proof of the lemma.
\end{proof}
We now need to prove that the reconstruction mapping is also well defined, i.e. that it is indeed possible to construct $\Phi$ according to the procedure given in the statement of the proposition. This is done in the following lemma.
\begin{lemma}\label{lemmarec}
Let $\zeta,b \in W^{1,\infty}(\R^d)$ and $\ovV\in H^{1/2}(\R^d)^d$. \\
{\bf i.} The quantity $\psi:=-G[\zeta]^{-1}(\nabla\cdot (h\ovV))$ is well defined in $\dot{H}^{-1/2}(\R^d)$.\\
{\bf ii.} There exists a unique variational solution $\Phi\in \dot{H}^1(\Omega)$ to the boundary value problem
$$
\begin{cases}
\Delta_{X,z}\Phi=0\quad\mbox{ in }\quad \Omega,\\
\Phi_{\vert_{z=\zeta}}=\psi,\qquad \partial_n \Phi_{\vert_{z=-h_0+b}}=0.
\end{cases}
$$
{\bf iii.} Denoting $\bU=\nabla_{X,z}\Phi$, one has $\bU\in L^2_b(\Omega,\dive\!\!,\curl\!\!)$.
\end{lemma}
\begin{proof}[Proof of the lemma]
For the first point, one needs to show that there exists a unique $\psi\in \dot{H}^{1/2}(\R^d)$ such that $G[\zeta]\psi=-\nabla\cdot (h \ovV)$. 
Equivalently, one needs to show that there exists a unique $\Phi\in \dot{H}^1(\Omega)$ such that
$$
\begin{cases}
\Delta_{X,z}\Phi=0\quad\mbox{ in }\quad \Omega,\\
\partial_n\Phi_{\vert_{z=\zeta}}=-\nabla\cdot (h\ovV),\qquad \partial_n\Phi_{\vert_{z=-h_0+b}}=0
\end{cases}
$$
or, in a variational form,
$$
\forall \varphi\in \dot{H}^1(\Omega), \qquad \int_\Omega \nabla_{X,z}\Phi\cdot \nabla_{X,z}\varphi=-\int_{\R^d}\nabla\cdot(h\ovV)\varphi_{\vert_{z=\zeta}}.
$$
Remarking that for all $\varphi\in C^\infty(\Omega)\cap \dot{H}^1(\Omega)$ (which is dense in $\dot{H}^1(\Omega)$ as shown in \cite{DL} or Proposition 2.3 of \cite{L_book}) one has
\begin{align*}
-\int_{\R^d}\nabla\cdot(h\ovV)\varphi_{\vert_{z=\zeta}}&=\int_{\R^d} \Lambda^{1/2}(h\ovV)\cdot \Lambda^{-1/2}\nabla (\varphi_{\vert_{z=\zeta}})\\
&\leq \abs{h\ovV}_{H^{1/2}}\abs{ \Lambda^{-1/2}\nabla (\varphi_{\vert_{z=\zeta}})}_2\\
&\leq \abs{h}_{W^{1,\infty}}\abs{\ovV}_{H^{1/2}}\Abs{\nabla_{X,z}\varphi}_2,
\end{align*}
the last inequality stemming from standard product estimates and Remark 3.14 in \cite{L_book}. It follows that the right-hand-side in the above variational formulation defines a linear form on $\dot{H}^1(\Omega)$; the existence and uniqueness of $\Phi$ and therefore of $\psi=\Phi_{\vert_{z=\zeta}}$ follows classically from the Lax-Milgram theorem.\\
The proof of the last two points of the lemma is straightforward and therefore omitted.
\end{proof}
The only thing left to prove is that  ${\mathfrak R}[\zeta]$ is a right inverse to ${\mathfrak A}$, i.e. that for all $\bU\in L^2_b(\dive\!\!,\!\curl\!\!)$, one has
${\mathfrak R}[\zeta]{\mathfrak A}[\zeta]\bU=\bU$. Let us therefore denote $\bU'={\mathfrak R}[\zeta]{\mathfrak A}[\zeta]\bU$ and show that $\bU'=\bU$. By construction, one has $\underline{U}'\cdot N=-\nabla\cdot (h\ovV)$ (with $\underline{U}':=\bU'_{\vert_{z=\zeta}}$). But since $\bU$ is divergence free and that its normal trace vanishes at the bottom, one also gets by integrating the incompressibility relation that $\underline{U}\cdot N=-\nabla\cdot (h\ovV)$. It follows that the normal traces of $\bU$ and $\bU'$ coincide at the surface and at the bottom (where they both vanish). Since they are also divergence and curl free, one deduces that $\bU=\bU'$.
\end{proof}

\subsubsection{The classical water waves equations in the $(\zeta,Q)$ variables}

It follows from Proposition \ref{propclosed} that one can replace $\bU=(V ,w)$ by ${\mathfrak R}[\zeta]\ovV$ in the formulation \eqref{Eulerav1}-\eqref{PNH}, hereby obtaining a closed system of equations in $(\zeta,Q)$. More precisely, writing $\ovV=Q/h$ and defining the "Reynolds"\footnote{This terminology introduced in \cite{CL1} is of course improper but the analogy with turbulence can be useful. Replacing statistical averaging by vertical integration, ${\bf R}$ measures the importance of the variations of the horizontal velocity field $V$ with respect to its average. These variations are only due to non-hydrostatic (dispersive) effects since the flow is assumed to be irrotational; in the general case with vorticity, ${\mathbf R}$ takes also into account the shear effects induced by the vorticity \cite{CL1}.} tensor ${\bf R}$ and 
the non hydrostatic acceleration ${\bf a}_{\rm NH}$ as
\begin{align}
\label{defRey}
{\bf R}(h,Q)&=\int_{-h_0+b}^\zeta (\mathfrak R[\zeta ]\ovV-\ovV)\otimes (\mathfrak R[\zeta ]\ovV-\ovV),\\
\label{defFNH}
{\bf a}_{\rm NH}(h,Q)&=\frac{1}{h}\int_{-h_0+b}^\zeta \nabla \big[\int_z^{\zeta} \big(\dt \mathfrak R[\zeta ]\ovV+ (\mathfrak R[\zeta]\ovV)\cdot \nabla_{X,z} \mathfrak R[\zeta ]\ovV\big)\cdot {\bf e}_z\big],
\end{align}
with ${\bf e}_z$ the vertical upward unit vector, we can rewrite  \eqref{Eulerav1}-\eqref{PNH} under a closed form. The following proposition is therefore  a direct consequence of Proposition \ref{propclosed}.
\begin{proposition}
If $\zeta$ and $\bU$ solve the free-surface Euler equations \eqref{Eul1}-\eqref{Eul6}, then $(\zeta,Q)$, with $Q=h\ovV$ and $\ovV$ as in \eqref{defavV}, solve the following closed system of equations in $(\zeta,Q)$,
\begin{equation}\label{Eulerav1closed}
\begin{cases}
\dsp \dt \zeta+\nabla\cdot Q=0,\\
\dsp \dt Q+\nabla\cdot (\frac{1}{h} Q\otimes Q)+gh \nabla\zeta+\nabla\cdot {\bf R}(h,Q)+h{\bf a}_{\rm NH}(h,Q)=0,
\end{cases}
\end{equation}
where ${\bf R}(h,Q)$ and ${\bf a}_{\rm NH}(h,Q)$ are as in \eqref{defRey}.
\end{proposition}
\begin{remark}[The energy in $(\zeta,Q)$ variables]\label{remNRG}
The total energy of the fluid
$$
E_{\rm fluid}=\frac{1}{2}g \int_{\R^d}\zeta^2 +\frac{1}{2}\int_{\Omega(t)}\abs{\bU}^2
$$
is formally conserved by the free surface Euler equations \eqref{Eul1}-\eqref{Eul6}. In the Zakharov-Craig-Sulem formulation, this energy can be written in terms of $\zeta$ and $\psi$ (where $\psi=\Phi_{\vert_{z=\zeta}}$ and $\bU=\nabla_{X,z}\Phi$), namely,
$$
E_{\rm fluid}=\frac{1}{2}g \int_{\R^d}\zeta^2 +\frac{1}{2}\int_{\R^d}\psi G[\zeta]\psi.
$$ 
As seen in the proof of Lemma \ref{lemmarec}, one has $\psi=-G[\zeta]^{-1}\nabla\cdot Q$, and we can express $E$ in terms of $\zeta$ and $Q$,
\begin{equation}\label{defEfluid}
E_{\rm fluid}=\frac{1}{2}g \int_{\R^d}\zeta^2 +\frac{1}{2}\int_{\R^d}\nabla\cdot QG[\zeta]^{-1}\nabla\cdot Q .
\end{equation}
\end{remark}

\subsection{The water waves equations with a floating object in the $(\zeta,Q)$ variables}\label{sectFSEzVfloat}

Our purpose in this section is to generalize the formulation  \eqref{Eulerav1closed} of the water waves equation as a closed system of equations in terms of $(h,Q)$ in the presence of a floating solid. Before we state this generalization, let us remark that, in absence of any immersed device, the acceleration ${\bf a}_{\rm FS}:=\dt^2 \zeta$ of the surface of the fluid can be deduced from \eqref{Eulerav1closed},
\begin{align}
\nonumber
{\bf a}_{\rm FS}(h,Q)&=-\nabla\cdot \dt Q\\
&=\nabla\cdot
\Big[\nabla\cdot (\frac{1}{h} Q\otimes Q)+gh \nabla\zeta+\nabla\cdot \big({\bf R}(h,Q)\big)+h{\bf a}_{\rm NH}(h,Q)\Big].
\label{defba}
\end{align}
Under the floating structure, the acceleration of the surface is imposed by the motion of the structure, i.e. one has $\dt^2\zeta=\dt^2 \zeta_{\rm w}$, and the relation $\dt^2\zeta={\bf a}_{\rm FS}$ is no longer true. This implies that an additional term must be added to the momentum equation to account for the presence of the structure. More precisely, one has the following proposition in which the interior pressure is expressed as a Lagrange multiplier associated to the constraint \eqref{constraint}. We recall that we use the notation
$$
\underline{P}_{\rm e}=\underline{P}_{\vert_{\mathcal E}(t)}
\quad\mbox{ and }\quad
\underline{P}_{\rm i}=\underline{P}_{\vert_{\mathcal I}(t)},
$$
that ${\bf R}(h,Q)$ and ${\bf a}_{\rm NH}(h,Q)$ are defined in \eqref{defRey}-\eqref{defFNH}, and that ${\bf a}_{\rm FS}(h,Q)$ is defined in \eqref{defba}.
\begin{proposition}\label{propWWst}
Let us consider a solution of the free surface Euler equations in the presence of a floating structure \eqref{Eul1}-\eqref{Eul6} and \eqref{constraint}-\eqref{contact1}, and let in particular $\zeta$ and $\bU$ be the associated surface elevation and velocity field.
Then $\zeta$ and $Q$, with $Q=h\ovV$ and $\ovV$ as in \eqref{defavV}, solve the following system on $\R^d$,
\begin{equation}\label{Eulerav2}
\begin{cases}
\dsp \dt \zeta+\nabla\cdot Q=0,\\
\dsp \dt Q+\nabla\cdot (\frac{1}{h} Q\otimes Q)+gh \nabla\zeta+\nabla\cdot {\bf R}(h,Q)+h{\bf a}_{\rm NH}(h,Q)=- \frac{h}{\rho}\nabla \underline{P},
\end{cases}
\end{equation}
with  the surface pressure $\underline{P}$ given by
\begin{equation}\label{eqPw}
\underline{P}_{\rm e}=P_{\rm atm}
\quad\mbox{\textnormal{ and }}\quad
\begin{cases}
-\nabla\cdot (\frac{h}{\rho}\nabla P_{\rm i})=-\dt^2 \zeta_{\rm w}+{\bf a}_{\rm FS}(h,Q) \quad\mbox{\textnormal{on}}\quad {\mathcal I}(t),\\
{P_{\rm i}}_{\vert_{\Gamma(t)}}=P_{\rm atm},
\end{cases}
\end{equation}
and with the coupling conditions at the contact line
\begin{equation}\label{coupling}
\zeta_{\rm e}=\zeta_{\rm i} \quad \mbox{ and }\quad Q_{\rm e}=Q_{\rm i} \quad\mbox{ on }\quad \Gamma(t).
\end{equation}
Conversely, if $\zeta$, $Q$ and ${\mathcal I}(t)$ solve \eqref{Eulerav2}-\eqref{coupling}, and if moreover the initial conditions $(\zeta^0,Q^0)$ satisfy
\begin{equation}\label{assCI}
\zeta^0={\zeta_{\rm
    w}}_{\vert_{t=0}} \quad\mbox{ and }\quad \nabla\cdot Q^0=-\dt {\zeta_{\rm
    w}}_{\vert_{t=0}} \quad\mbox{ on }\quad {\mathcal I}(0),
\end{equation}
then for all $t\geq 0$, one has $\zeta(t,\cdot )=\zeta_{\rm w}(t,\cdot)$ on ${\mathcal I}(t)$.
\end{proposition}
\begin{remark}
One of the advantages of working with the $(\zeta,Q)$ formulation of the water waves equations is that, in \eqref{coupling}, the transition condition on $Q$ at the contact line can be expressed very simply. This would not be the case if we had worked in the hamiltonian variables $(\zeta,\psi)$ of the Zakharov-Craig-Sulem formulation.
\end{remark}
\begin{remark}
The interior pressure $\uP_{\rm i}$ is given by the simple elliptic equation \eqref{eqPw} cast in the interior region $\cI(t)$. In the component ${\bf a}_{\rm FS}$ of the source term of this elliptic equation, time derivatives of the velocity field are present (through the non-hydrostatic acceleration ${\bf a}_{\rm NH}$). In the configurations investigated in this paper it is very convenient to proceed this way; however, it is also possible to express these time derivatives of the velocity in terms of the pressure field (using Euler's equations). This latter approach leads to a different equation for the interior pressure (equivalent of course to \eqref{eqPw}) which can also be of interest (in deep water settings or for the mathematical analysis of the equations for instance). We therefore derive it in Appendix \ref{Appinteriorpressure}.
\end{remark}
\begin{remark}\label{remNRG2}
One can use the classical balance of energy for the water waves equations when the pressure at the surface is not constant  to see that the energy conservation in $(h,Q)$ variables given in Remark \ref{remNRG} must be modified as follows in the presence of a floating structure,
$$
\frac{d}{dt}E_{\rm fluid}=-\int_{\cI(t)}\dt \zeta \frac{\uP_{\rm i}-P_{\rm atm}}{\rho}=-\int_{\cI(t)} Q\cdot \frac{\nabla \uP_{\rm i}}{\rho}.
$$
\end{remark}
For the sake of convenience, we introduce the following terminology (note that according to the last point of the proposition, the constraint \eqref{constraint} is equivalent to the assumption \eqref{assCI} on the initial data).
\begin{definition}\label{defiCWW}
The set of equations  \eqref{Eulerav2}-\eqref{coupling} form the \emph{water waves equations with a floating structure in $(\zeta,Q)$ variables}. It is always assumed that the initial condition satisfies \eqref{assCI} so that the constraint \eqref{constraint} is automatically satisfied.
\end{definition}
\begin{proof}
The mass conservation equation
$$
\dt \zeta+\nabla\cdot Q=0
$$
is obtained classically by integrating the incompressibility condition \eqref{Eul2} and using the kinematic condition \eqref{Eul5} and the impermeability condition \eqref{Eul6}.\\
Integrating vertically the horizontal component of the momentum equation \eqref{Eul1}, one gets
$$
\dsp \dt Q+\nabla\cdot ( \int_{-h_0+b}^{\zeta} V\otimes
V)+\int_{-h_0+b}^{\zeta} \nabla P=0.
$$
Denoting by $\underline{P}$ the trace of the pressure at the surface of the fluid, one can write
\begin{align*}
(\nabla P)(t,X,z)&=\nabla \big(\underline{P}+\int_{z}^{\zeta(t,X)}-\dz
P(t,X,z')dz'\big)\\
&=\nabla \big(\underline{P}+\int_{z}^{\zeta(t,X)} \rho g-\dz
P_{\rm NH}(t,X,z')dz'\big)
\end{align*}
where we used the vertical component of \eqref{Eul1} to derive the second identity. We therefore get
\begin{align*}
(\nabla P)(t,X,z)&=\nabla \underline{P}+\rho g \nabla \zeta+\nabla P_{\rm NH}
\end{align*}
and the averaged momentum equations takes the form
$$
\dsp \dt Q+\nabla\cdot ( \int_{-h_0+b}^\zeta V\otimes
V)+gh\nabla \zeta+\frac{1}{\rho}\int_{-h_0+b}^\zeta \nabla P_{\rm NH}=-h \frac{1}{\rho}\nabla \underline{P}.
$$
Using Proposition \ref{propclosed}, one can rewrite this equation as 
$$
 \dt Q+\nabla\cdot (\frac{1}{h} Q\otimes Q)+gh \nabla\zeta+\nabla\cdot {\bf R}(h,Q)+h{\bf a}_{\rm NH}(h,Q)=-h \frac{1}{\rho}\nabla \underline{P}.
$$
On the exterior domain, one has $\underline{P}=P_{\rm atm}$ by \eqref{Eul4} and the right-hand-side vanishes; in the interior domain, the right-hand-side is equal to $-h \frac{1}{\rho}\nabla \uP_{\rm i}$ with $\uP_{\rm i}$ to be determined. In order to do so, we use the mass conservation equation together with the constraint \eqref{constraint} to obtain that
$$
\nabla\cdot Q=-\dt \zeta_{\rm w} \quad\mbox{ in }\quad {\mathcal I}(t).
$$
Taking the divergence of the momentum equation, one gets therefore the
following elliptic equation  for $\uP_{\rm i}$,
$$
-\nabla\cdot (\frac{h}{\rho}\nabla \uP_{\rm i})=-\dt^2 \zeta_{\rm w}+{\bf a}_{\rm FS}(h,Q),
$$
and we deduce from \eqref{contact1} the boundary conditions ${\uP_{\rm i}}=P_{\rm atm}$ on the boundary $\Gamma(t)=\partial{\mathcal I}(t)$. To obtain the boundary condition on $Q$ stated in \eqref{coupling}, we just need to remark that
$$
Q_{\rm e}(t,X)-Q_{\rm i}(t,X)=\int_{-h_0+b}^\zeta \big(V_{\rm e}(t,X,z)-V_{\rm i}(t,X,z)\big)dz; 
$$
since the flow is incompressible and irrotational, we know by standard elliptic theory that the velocity field $\bU$ and therefore $V$ is smooth in the interior of $\Omega$. This implies that the r.-h.-s. in the above expression vanishes, and therefore that $Q_{\rm e}=Q_{\rm i}$ on $\Gamma(t)$.
This achieves the proof of the first part of the proposition. \\
For the second part, we just need to remark that \eqref{Eulerav2}-\eqref{eqPw} imply that $\dt^2\zeta=\dt^2\zeta_{\rm w}$ so that $\zeta=\zeta_{\rm w}$ provided that $(\zeta,\dt\zeta)$ and $(\zeta_{\rm w}, \dt \zeta_{\rm w})$ coincide at $t=0$, leading to the assumptions on the initial conditions made in the statement of the proposition.
\end{proof}

\section{Coupling with the solid dynamics}\label{sectsolid}

We address in this section the coupling of the water waves equations with a floating structure \eqref{Eulerav2}-\eqref{eqPw} with the motion of the partially immersed solid which at time $t$ occupies the volume ${\mathcal C}(t)$. This coupling was already present in \eqref{Eulerav2}-\eqref{eqPw} but through the presence of the second time derivative $\dt^2 \zeta_{\rm w}$ which is not a natural quantity to describe the dynamics of the solid. We therefore want to derive a version of the equations \eqref{Eulerav2}-\eqref{eqPw} in terms of the velocity of the center of mass of the solid, and of its angular velocity.\\
We first consider in \S \ref{sectpresc} the case where the motion of the solid is prescribed; the case of a freely floating object is then treated in \S \ref{sectfloat}. In both cases, the general two-dimensional case is treated first, and the one-dimensional case where considerable simplifications can be performed is considered subsequently.
\begin{notation}
Throughout this section, we shall denote by $G(t)=(X_G(t),z_G(t))\in \R^{2+1}$ the coordinates of the center of mass of the solid and by
${\bf U}_G$ its velocity
$$
{\bf U}_G(t)=\left(\begin{array}{c} V_G(t) \\ w_G(t) \end{array}\right)=\left(\begin{array}{c} \dot{X}_G(t) \\ \dot{z}_G(t) \end{array}\right),   $$
where the dot stands for the time derivative.\\
 We also denote by $\bom(t)=(\bom_{\rm h}(t),\omega_{\rm v}(t))\in \R^{2+1}$ the angular velocity of the solid.
  \end{notation}

As for the kinematic condition \eqref{Eul5}, one easily derives that
\begin{equation}\label{Eul5sol}
\dt \zeta_{\rm w}-\underline{U}_{\rm w}\cdot N_{\rm w}=0 \quad\mbox{ on }\quad {\mathcal I}(t)
\quad\mbox{ with }\quad N_{\rm w}=\left(\begin{array}{c}-\nabla\zeta_{\rm w}\\ 1\end{array}\right)
\end{equation}
and where $\underline{U}_{\rm w}$ denotes the velocity of the solid on the wetted surface,
$$
\forall X\in {\mathcal I}(t),\qquad \underline{U}_{\rm w}(t,X)={\bf U}_{\mathcal C}(t,X,\zeta(t,X)),
$$
and ${\bf U}_{\mathcal C}(t,X,z)$ is the velocity at time $t$ of the point $(X,z)\in {\mathcal C}(t)$. From standard solid mechanics, we have therefore, 
\begin{equation}\label{defUc}
\underline{U}_{\rm w}={\bf U}_G+\bom\times {\bf r}_G
\quad \mbox{ with }\quad {\mathbf r}_G(t,X)=\left(\begin{array}{c} X-X_G(t)\\ \zeta_{\rm w}(t,X)-z_G(t)\end{array}\right),
\end{equation}
so that \eqref{Eul5sol} gives the following relation
\begin{align}
\label{eqdtzeta}
\dt \zeta_{\rm w}=&\big({\bf U}_G+\bom\times {\bf r}_G\big)\cdot N_{\rm w}\quad\mbox{ in }\quad {\mathcal I}(t).
\end{align}
We now have to distinguish two different situations
\begin{itemize}
\item The solid is in prescribed motion, in which case $G$ and $\bom$ are known functions of time
\item The solid is  freely floating, in which case the evolution of $G$ and $\bom$ are unknown functions whose evolution is coupled to the wave motion.
\end{itemize}
\subsection{The case of a structure with a prescribed motion}\label{sectpresc}

When the motion of the solid is prescribed, there is no influence of the flow on its motion, but the flow is of course affected by the presence of the solid. This influence is taken into account by the interior pressure $\underline{P}_{\rm i}$ in the equations \eqref{Eulerav2}-\eqref{eqPw} (the flow is pressurized). In the following proposition, we show how this pressure can be computed in terms of the position of the center of mass and of the rotation matrix. 
We consider first the most general case $d=2$; the simplifications in the one dimensional case $d=1$ where many computations can be carried out explicitly are described in \S \ref{sect1dpresc}.
\subsubsection{The general two dimensional case}
Before stating the main result of this section, it is convenient to introduce the following notations. We first define the second fundamental form associated to the solid structure. Denoting by ${{\bf n}}_{\rm w}=\frac{1}{\abs{N_{\rm w}}}N_{\rm w}$ the upward unit normal vector to the solid ${\mathcal C}(t)$ on the wetted surface $\partial_{\rm w}\cC(t)$, and by $T_X\partial_{\rm w}\cC$ the tangent plane to this surface at the point $(X,\zeta_{\rm w}(t,X))$, the second fundamental form is the bilinear mapping
$$
{\bf{\rm II}}: \begin{array}{lcl}
T_X\partial_{\rm w}\cC \times T_X\partial_{\rm w}\cC & \to & \R\\
({\bf t}, {\bf t'})&\mapsto& -(\nabla_{\bf t}{\bf n}_{\rm w},{\bf t}'),
\end{array}
$$
where $\nabla_{{\bf t}}{{\bf n}}_{\rm w}$ is the directional derivative of ${{\bf n}}_{\rm w}$ in the direction ${\bf t}$. We also define the tangent vector $\uU_{{\rm w},\tau}$ as
$$
\uU_{{\rm w},\tau}=\uU_{{\rm w}}-(\uU_{{\rm w}}\cdot N_{\rm w}){\bf e}_{z}=\left(\begin{array}{c} \uV_{{\rm w}} \\ \uV_{{\rm w}}\cdot \nabla\zeta_{\rm w}\end{array}\right),
$$
and we also define ${\mathcal Q}[{\bf r}_G](\cdot)$ as the quadratic form
$$
{\mathcal Q}[{\bf r}_G](V_G,\bom)
=(\bom\times N_{\rm w})\cdot \big(\bom\times {\bf r}_G-2\uU_{{\rm w},\tau}\big)-\sqrt{1+\abs{\nabla\zeta_{\rm w}}^2}{\bf{\rm II}}(\uU_{{\rm w},\tau},\uU_{{\rm w},\tau}).
$$
We recall that it is always assumed that the initial condition satisfies \eqref{assCI} so that the constraint \eqref{constraint} is automatically satisfied.
\begin{proposition}\label{proppresc}
Denoting by  $\bug$ the velocity  of the center of mass and by $\bom$ the angular velocity, the water waves equations with a floating structure then take the form
$$
\begin{cases}
\dsp \dt \zeta+\nabla\cdot Q=0,\\
\dsp \dt Q+\nabla\cdot (\frac{1}{h} Q\otimes Q)+gh \nabla\zeta+\nabla\cdot {\bf R}(h,Q)+h{\bf a}_{\rm NH}(h,Q)=S^{\rm I}+S^{\rm II}+S^{\rm III},
\end{cases}
$$
with the coupling conditions at the contact line
$$
\zeta_{\rm e}=\zeta_{\rm i} \quad \mbox{ and }\quad Q_{\rm e}=Q_{\rm i} \quad\mbox{ on }\quad \Gamma(t),
$$
and with the source terms  given in the exterior and interior domains by
$$
S^{j}_{\rm e}=0 \quad\mbox{ and }\quad S^{j}_{\rm i}=-\frac{h}{\rho}\nabla \uP_{\rm i}^j\qquad (j={\rm I,II,III})
$$
where: - $\uP_{\rm i}^{\rm I}$ corresponds to the interior pressure one would have if the solid were fixed,
$$
\begin{cases}
-\nabla\cdot (\frac{h}{\rho}\nabla \uP_{\rm i}^{\rm I})=
{\bf a}_{\rm FS}(h,Q) \quad\mbox{\textnormal{on}}\quad {\mathcal I}(t),\\
{\uP_{\rm i}^{\rm I}}_{\vert_{\Gamma(t)}}=P_{\rm atm},
\end{cases}
$$
- $\uP_{\rm i}^{\rm II}$ depends linearly on the first time derivatives of ${\bf U}_G$ and $\bom$,
$$
\begin{cases}
-\nabla\cdot (\frac{h}{\rho}\nabla \uP_{\rm i}^{\rm II})=-\big(\dot{{\bf U}}_G+\dot{\bom}\times {\bf r}_G\big)\cdot N_{\rm w} \quad\mbox{\textnormal{on}}\quad {\mathcal I}(t),\\
{\uP_{\rm i}^{\rm II}}_{\vert_{\Gamma(t)}}=0,
\end{cases}
$$
- $\uP_{\rm i}^{\rm III}$ gathers the quadratic terms in $V_G$ and $\bom$,
$$
\begin{cases}
-\nabla\cdot (\frac{h}{\rho}\nabla \uP_{\rm i}^{\rm III})={\mathcal Q}[{\bf r}_G](V_G,\bom)
 \quad\mbox{\textnormal{on}}\quad {\mathcal I}(t),\\
{\uP_{\rm i}^{III}}_{\vert_{\Gamma(t)}}=0.
\end{cases}
$$
\end{proposition}
\begin{remark}
The formula for $\uP_{\rm i}^{\rm II}$ and $\uP_{\rm i}^{\rm III}$ involve ${\bf r}_G$ and $N_{\rm w}$ which require the knowledge of $G=(X_G,z_G)$ and $\zeta_{\rm w}$. They can both be deduced from $\bug$ and $\bom$. Indeed, the position of the center of mass is found by solving (denoting by $G_0$ the initial position of the center of mass)
$$
\dot G=\bug,\qquad G(0)=G_0
$$
while $\zeta_{\rm w}$ is determined by the position of the solid which, at time $t$, is given by
$$
{\mathcal C}(t)=\big\{G(t)+\Theta(t)(M-G_0), M\in {\mathcal C}(0)\big\}
$$
with the rotation matrix $\Theta$ satisfying
$$
\dot\Theta=\bom\times \Theta,\qquad \Theta(0)=\mbox{\rm Id}.
$$
\end{remark}
\begin{remark}
\item- If ${\mathcal C}(t)$ is a sphere with a fixed center of mass (i.e. if $\bug=0$), then ${\mathcal Q}[{\bf r}_G](V_G,\bom)=0$ and $S^{\rm III}$ is therefore identically zero. This follows from simple computations and from the observation that, for a sphere of radius $R$, one has 
$$
{\bf{\rm II}}({\bf t}, {\bf t}')=\frac{1}{R}({\bf t}, {\bf t}'), \qquad \uU_{{\rm w},\tau}=\uU_{{\rm w}}=\bom\times{\bf r}_G
\quad\mbox{ and }\quad N_{\rm w}=-\frac{1}{R}\sqrt{1+\abs{\nabla\zeta_{\rm w}}^2}{\bf r}_G.
$$
\item - Similarly, ${\mathcal Q}[{\bf r}_G](V_G,\bom)=0$ if the solid is constrained to move vertically (so that $V_G=0$ and $\bom=0$). We shall use this remark in \S \ref{sectdiscrete} for the configuration for which we provide numerical schemes and simulations.
\end{remark}
\begin{proof}
By linearity of \eqref{eqPw}, it is enough to prove that 
\begin{equation}\label{P61}
\dt^2\zeta_{\rm w}=(\dbug+\dot{\bom}\times {\bf r}_G)\cdot N_{\rm w}- {\mathcal Q}[{\bf r}_G](V_G,\bom)
\quad\mbox{\textnormal{on}}\quad {\mathcal I}(t).
\end{equation}
 Time differentiating \eqref{eqdtzeta}, one gets in ${\mathcal I}(t)$,
$$
\dt^2 \zeta_{\rm w}=\big(\dbug+\dot{\bom}\times {\bf r}_G+\bom\times \dot{\bf r}_G\big)\cdot N_{\rm w}
+\uU_{{\rm w}}\cdot \dt {N}_{\rm w},
$$
and we therefore look closer at the terms $(\bom\times \dot{{\bf r}}_G)\cdot N_{\rm w}$ and $\uU_{{\rm w}}\cdot \dt N_{\rm w}$:
\begin{itemize}
\item[-] The term $(\bom\times \dot{{\bf r}}_G)\cdot N_{\rm w}$. By definition of ${\bf r}_{G}$ and using the fact that $\dt \zeta_{\rm w}=\uU_{{\rm w}}\cdot N_{\rm w}$, we get
\begin{align*}
(\bom\times \dot{{\bf r}}_G)\cdot N_{\rm w}&=-(\bom\times N_{\rm w})\cdot \dot{{\bf r}}_G\\
&=(\bom\times N_{\rm w})\cdot \big(\bU_G-(\uU_{{\rm w}}\cdot N_{\rm w}){\bf e}_z\big),
\end{align*}
and therefore
$$
(\bom\times \dot{{\bf r}}_G)\cdot N_{\rm w}=(\bom\times N_{\rm w})\cdot (\uU_{{\rm w},\tau}-\bom\times {\bf r}_G).
$$
\item[-] The term $\uU_{{\rm w}}\cdot \dt {N}_{\rm w}$. Since the vertical component of $N_{\rm w}$ is time and space independent, and denoting by $\uV_{\rm w}$ the horizontal component of $\uU_{{\rm w}}$, one has
\begin{align*}
\uU_{{\rm w}}\cdot \dt {N}_{\rm w}=&-\uV_{{\rm w}}\cdot \nabla (\uU_{{\rm w}}\cdot N_{\rm w})\\
=&-\big((\uV_{{\rm w}}\cdot \nabla) \uU_{{\rm w}}\big)\cdot N_{\rm w}+\uV_{{\rm w}}\cdot \big((\uV_{{\rm w}}\cdot \nabla)\nabla\zeta_{\rm w}\big).
\end{align*}
Recalling that $\uU_{{\rm w}}=\bU_G+\bom\times {\bf r}_G$, we deduce that
\begin{align*}
\uU_{{\rm w}}\cdot \dt {N}_{\rm w}
=&-\big[\bom\times \big((\uV_{{\rm w}}\cdot \nabla) {\bf r}_G\big)\big]\cdot N_{\rm w}+\uV_{{\rm w}}\cdot H(\zeta_{\rm w})\uV_{{\rm w}},
\end{align*}
where $H(\zeta_{\rm w})$ denotes the Hessian matrix of $\zeta_{\rm w}$. Remarking further that 
$(\uV_{{\rm w}}\cdot \nabla) {\bf r}_G=\uU_{{{\rm w}},\tau}$, we finally get
$$
\uU_{{\rm w}}\cdot \dt {N}_{\rm w}=(\bom\times N_{\rm w})\cdot \uU_{{{\rm w}},\tau}+\uV_{{\rm w}}\cdot H(\zeta_{\rm w})\uV_{{\rm w}}.
$$
\end{itemize}
Gathering all these elements, we get that
$$
\dt^2 \zeta_{\rm w}=\big(\dbug+\dot{\bom}\times {\bf r}_G\big)\cdot N_{\rm w}
+(\bom\times N_{\rm w})\cdot \big(2\uU_{{\rm w},\tau}-\bom\times {\bf r}_G  \big)+\uV_{{\rm w}}\cdot H(\zeta_{\rm w})\uV_{{\rm w}}.
$$
The identity \eqref{P61} follows therefore if we can show that 
$$
{\bf {\rm II}}(\uU_{{\rm w},\tau},\uU_{{\rm w},\tau})=\frac{1}{\sqrt{1+\abs{\nabla\zeta_{\rm w}}^2}} \uV_{{\rm w}}\cdot H(\zeta_{\rm w})\uV_{{\rm w}}.
$$
In the canonical basis $({\bf t}_1,{\bf t}_2)$ of the tangent space, with ${\bf t}_1=(1,0,\partial_x\zeta_{\rm w})^T$ and ${\bf t}_2=(0,1,\dy\zeta_{\rm w})^T$, the of the second fundamental form is $\frac{1}{\sqrt{1+\abs{\nabla\zeta_{\rm w}}^2}}H(\zeta_{\rm w})$, and the tangent vector $\uU_{{\rm w},\tau}$ is represented by $\uV_{{\rm w}}$, so that the result follows.
\end{proof}

\subsubsection{Simplification in the one dimensional case}\label{sect1dpresc}

When the horizontal dimension $d$ is equal to $1$, the number of unknown variables reduces:
\begin{itemize}
 \item For the solid. In dimension $d=1$, the velocity of the center of mass has no transverse component, $\bug=(u_G,0,w_G)^T$, and $\bom=(0,\omega,0)^T$ is perpendicular to the $(x,z)$ plane. We therefore adapt our notations for the sake of simplicity
 $$
 G=(x_G,z_G)^T,\qquad \bug=(u_G,w_G)^T,\qquad {\bf r}_G=(x-x_G,\zeta_{\rm w}-z_G)^T.
 $$
 Instead of  the six components vector $(\bug,\bom)$, the motion of the solid is determined by the three dimensional vector $(u_G,w_G,\omega)$. 
\item  In the fluid. Similarly, in the fluid, the velocity $\bU=(u,0,w)^T$ has no transverse component, and the horizontal discharge takes the form $Q=(q,0)$. The water waves equations \eqref{Eulerav2} in $(\zeta,Q)$ variables therefore simplify into a system of two scalar equations on $(\zeta,q)$, in which the operators ${\bf R}(h,Q)$ and ${\bf a}_{\rm NH}(h,Q)$ defined in \eqref{defRey}-\eqref{defFNH} are therefore denoted ${\bf R}(h,q)$ and ${\bf a}_{\rm NH}(h,q)$ for the sake of clarity.
\item For the interior domain. Assuming (as we shall always do without loss of generality in dimension $d=1$) that the interior domain is an interval, we write 
$$
{\mathcal I}(t)=\big(x_-(t),x_+(t)\big).
$$
\end{itemize}
The water waves equations with a floating structure \eqref{Eulerav2}-\eqref{coupling}, as well as the equations for the interior pressure given in Proposition \ref{proppresc} take a much simpler form due the smaller number of variables. The most striking simplification however is that among the three components of the interior pressure described in Proposition \ref{proppresc}, the computations of the last two -- that take into account the motion of the solid structure -- can be carried out explicitly. It is convenient at this point to introduce the following notation for a horizontal averaging in the interior domain that take into account the shape of the immersed region of the solid.
\begin{notation}\label{notav}
If $f$ is a function defined on ${\mathcal I}(t)=\big(x_-(t),x_+(t)\big)$, we define its average and oscillating components as
$$
\av{f}:=\frac{1}{\dsp \int_{x_-}^{x_+}\frac{1}{h}}\int_{x_-}^{x_+}\frac{f}{h}
\quad \mbox{ and }\quad
f^*:=f-\av{f}.
$$
\end{notation}
We can now state the following proposition in which it is shown that the contributions due to the motion of the solid in the momentum equation for the fluid can be computed explicitly.
 \begin{proposition}\label{propWWCNU}
Assume that $d=1$ and that the position of the center of mass and the angular velocity are some given functions of time $t\mapsto G(t)=(x_G(t),z_G(t))$ and $t\mapsto \omega(t)$. The water waves equations with a floating structure \eqref{Eulerav2}-\eqref{coupling} take the form
$$
\begin{cases}
\dsp \dt \zeta+\dx q=0,\\
\dsp \dt q+\dx (\frac{1}{h}q^2)+gh \dx\zeta+\dx {\bf R}(h,q)+h{\bf a}_{\rm NH}(h,q)=S^{\rm I}+S^{\rm II}+S^{\rm III},
\end{cases}
$$
where the source terms $S^{\rm I}$, $S^{\rm II}$ and $S^{\rm III}$ are given by
\begin{align*}
S^{\rm I}_{\rm e}&=0 \quad \mbox{ and }\quad S^{\rm I}_{\rm i}=\big[\dx (\frac{1}{h}q^2)+gh\dx\zeta+\dx {\bf R}(h,q)+h{\bf a}_{\rm NH}(h,q)\big]^*\\
S^{\rm II}_{\rm e}&=0 \quad \mbox{ and }\quad S^{\rm II}_{\rm i}=\dot \bU_G^\perp\cdot {\bf r}_G^*+\frac{1}{2}\dot\omega (\abs{{\bf r}_G}^2)^*\\
S^{\rm III}_{\rm e}&=0 \quad \mbox{ and }\quad S^{\rm III}_{\rm i}=-u_G^2(\dx\zeta_{\rm w})^*+2u_G\omega\big({\bf r}_G\cdot N^\perp_{\rm w}\big)^* +\omega^2 \big((\zeta_{\rm w}-z_G){\bf r}_G\cdot N^\perp_{\rm w}\big)^*
\end{align*}
and with the coupling conditions at the contact points
$$
\zeta_{\rm e}=\zeta_{\rm i} \quad \mbox{ and }\quad q_{\rm e}=q_{\rm i} \quad\mbox{ at }\quad x=x_\pm(t).
$$
\end{proposition}
\begin{proof}
We shall repeatedly use the following lemma in this proof.
\begin{lemma}
\label{lemexplicit}
Let $x_-<x_+$ and $g,h\in C([x_-,x_+])$ be such that $\inf h>0$ on $[x_-,x_+]$. There exists a unique solution $P\in C^1([x_-,x_+])$ to the boundary value problem
$$
\begin{cases}
-\dx (h \dx P)=\dx g \quad\mbox{ in }\quad (x_-,x_+),\\
P(x_\pm)=0,
\end{cases}
$$
and moreover one has, using Notation \ref{notav},
$$
-h\dx P= g^*\quad\mbox{ in }\quad [x_-,x_+].
$$
\end{lemma}
\begin{proof}[Proof of the lemma]
Integrating, one gets
$$
-h\dx P=g+c
$$
for some integration constant $c$. Dividing by $h$, using the fact that $P(x_-)=0$, and integrating again yields
$$
-P=\int_{x_-}^x\frac{g}{h}+c\int_{x_-}^x\frac{1}{h}.
$$
The value of $c$ is then given by the fact that $P(x_+)=0$, namely,
$$
c=- \av{g};
$$
plugging this into the expression for $-h\dx P$ derived above, this gives the result.
\end{proof}
In dimension $d=1$, the equation for $P_{\rm i}^{\rm I}$ becomes, in $\big(x_-(t),x_+(t)\big)$,
\begin{align*}
-\dx\big(\frac{h}{\rho}\dx \underline{P}_{\rm i}^{\rm I})&={\bf a}_{\rm FS}(h,q) \\
&=\dx\big[\dx (\frac{1}{h}q^2)+gh\dx\zeta+\dx {\bf R}(h,q)+h{\bf a}_{\rm NH}(h,q)\big],
\end{align*}
with the boundary condition ${\uP_{\rm i}}_{\vert_{x_\pm(t)}}=P_{\rm atm}$. It follows from Lemma \ref{lemexplicit} that
$$
-\frac{h}{\rho}\dx \uP_{\rm i}^{\rm I}=\big[\dx (\frac{1}{h}q^2)+gh\dx\zeta+\dx {\bf R}(h,q)+h{\bf a}_{\rm NH}(h,q)\big]^*.
$$
Similarly, the equation for $\uP_{\rm i}^{\rm II}$ is
\begin{align*}
-\dx (\frac{h}{\rho}\dx \uP_{\rm i}^{\rm II})&=\dot{u}_G\dx\zeta_{\rm w}-\dot{w}_G+\dot{\omega}  \big((x-x_G)+(\zeta_{\rm w}-z_G)\dx\zeta_{\rm w}\big)   \\
&=\dx\big[\dot{\bU}^\perp_G\cdot {\bf r}_G +\frac{1}{2}\dot{\omega} \vert {\bf r}_G \vert^2\big],
\end{align*}
with the boundary conditions ${\uP_{\rm i}^{\rm II}}_{\vert_{x_\pm(t)}}=0$. It follows therefore from Lemma \ref{lemexplicit} that
$$
-\frac{h}{\rho}\dx \uP_{\rm i}^{\rm II}=\dot \bU_G^\perp\cdot {\bf r}_G^*+\frac{1}{2}\dot\omega (\abs{{\bf r}_G}^2)^*.
$$
Finally, one has for $\uP^{\rm III}_{\rm i}$,
\begin{align*}
-\dx(\frac{h}{\rho} \dx \uP_{\rm i}^{\rm III})=&-u_G^2\dx^2\zeta_{\rm w}-2u_G\omega \big(1+(\dx\zeta_{\rm w})^2+(\zeta_{\rm w}-z_G)\dx^2\zeta_{\rm w}\big)\\
&-\omega^2 \big((x-x_G)\dx\zeta_{\rm w}+(\zeta_{\rm w}-z_G)(1+2(\dx\zeta_{\rm w})^2)+(\zeta_{\rm w}-z_G)^2\dx^2\zeta_{\rm w}\\
=&-\dx \big[ u_G^2 \dx\zeta_{\rm w}+2u_G\omega \big((x-x_G)+(\zeta_{\rm w}-z_G)\dx\zeta_{\rm w}\big)\\
&\phantom{-\dx\big[}+\omega^2 \big((x-x_G)(\zeta_{\rm w}-z_G)+(\zeta_{\rm w}-z_G)^2\dx\zeta_{\rm w}\big]
\end{align*}
with the boundary conditions ${\uP_{\rm i}^{\rm III}}_{\vert_{x_\pm(t)}}=0$. We therefore get from Lemma \ref{lemexplicit} that
$$
-\frac{h}{\rho} \dx \uP_{\rm i}^{\rm III}=-u_G^2(\dx\zeta_{\rm w})^*-2u_G\omega\big({\bf r}_G^\perp\cdot N_{\rm w}\big)^* -\omega^2 \big((\zeta_{\rm w}-z_G){\bf r}_G^\perp\cdot N_{\rm w}\big)^*.
$$
Setting $S^{\rm j}_{\rm i}=-\frac{h}{\rho}\dx \uP_{\rm i}^{\rm j}$ ($j={\rm I,II,III}$) then gives the result.
\end{proof}

\subsection{The case of a freely floating structure}\label{sectfloat}

When the solid is freely floating, its  motion is still determined by the velocity of its center of mass and by its angular velocity. These two quantities are however no longer prescribed functions of time and must be found by solving Newton's laws in which the force and torque exerted by the fluid on the solid play an important role. This strong coupling is investigated here, and we exhibit in particular the added mass effect that it induces.\\
We consider first the most general case $d=2$; the simplifications in the one dimensional case $d=1$ are described in \S \ref{sect1dfloat} below.
\subsubsection{The general two dimensional case}

As above for the case of a prescribed motion, the motion of the solid is determined by the velocity of its center of mass $\bug(t)=(V_G(t),w_G(t))\in \R^{2+1}$ and its angular velocity $\bom(t)=(\bom_{\rm h}(t),\omega_{\rm v}(t))\in \R^{2+1}$. The difference is that these functions are not a priori known any more and must be determined through Newton's laws for the floating solid.
 We shall need the following notations.
\begin{notation}
 We denote by ${\mathfrak m}$ the mass of the solid object, and by  ${\mathfrak I}(t)$ the inertia matrix of the body relative to the center of mass and measured in the Eulerian frame; this frame being inertial, the inertia matrix depends on time. Its value is determined from its value at time $t=0$ through the formula
 \begin{equation}\label{eqI}
 {\mathfrak I}(t)=\Theta(t) {\mathfrak I}(0)\Theta(t)^T,
 \end{equation}
 where $\Theta(t)\in SO(3)$ is the rotation matrix found by solving the ODE
 \begin{equation}\label{eqTheta}
 \dot\Theta=\bom\times \Theta,\qquad \Theta(0)=\mbox{Id}_{3\times 3}.
 \end{equation}
\end{notation}  
We can now state Newton's laws for the conservation of linear momentum and angular momentum,
\begin{align}
\label{Newton1}
{\mathfrak m} \dbug&=-{\mathfrak m} g{\bf e}_z+F_{\rm fluid}\\
\label{Newton2}
\frac{d}{dt}\big({\mathfrak I}{\bom}\big)&=T_{\rm fluid},
\end{align}
where $F_{\rm fluid}$ and $T_{\rm fluid}$ are respectively the resulting force and torque exerted by the fluid on the solid\footnote{It is actually the resulting force and torque after deducing the contribution due to the atmospheric pressure.},
\begin{equation}\label{forcetorque}
F_{\rm fluid}=\int_{{\mathcal I}(t)}(\uP_{\rm i}-P_{\rm atm}) N_{\rm w} 
\quad\mbox{ and }\quad
T_{\rm fluid}=\int_{{\mathcal I}(t)} (\uP_{\rm i}-P_{\rm atm}){\bf r}_G\times N_{\rm w}, 
\end{equation}
where we recall that  $N_{\rm w}=\left(\begin{array}{c}-\nabla\zeta_{\rm w} \\ 1\end{array}\right)$.
We shall show that part of the contribution of $F_{\rm fluid}$ and $T_{\rm fluid}$ can be put under the form of an added mass operator in Newton's laws \eqref{Newton1}-\eqref{Newton2}; to this end, we need to introduce the elementary potentials.
\begin{definition}[Elementary potentials]\label{defelempot}
Let ${\mathcal I}\subset{\R^2}$ be a bounded domain with Lipschitz boundary $\Gamma$,  $\zeta_{\rm w}\in W^{1,\infty}(\cI)$ and $h\in C(\overline{\cI})$ be such that $\inf_{\cI} h>0$. Let also $G=(X_G,z_G)\in \R^3$ and denote by $N_{\rm w}$ and ${\bf r}_G$ the vectors fields 
$$
\forall X\in \cI, \qquad N_{\rm w}(X)=\left(\begin{array}{c} -\nabla \zeta_{\rm w}(X) \\ 1 \end{array}\right)
 \quad\mbox{ and }\quad
{\bf r}_G(X)=\left(\begin{array}{c} X-X_G \\ \zeta_{\rm w}(X)-z_G\end{array}\right).
$$
We define the elementary potentials $\Phi_{\mathcal I}^{(j)}$ ($j=1,\dots,6$) as the unique solutions of the boundary value problems, for $j=1,2,3$,
$$
\begin{cases}
-\nabla\cdot h\nabla \Phi^{(j)}_{{\mathcal I}}=(N_{\rm w})_j \quad \mbox{on}\quad {\mathcal I}\\
\Phi^{(j)}_{{\mathcal I}}\,_{\vert_\Gamma}=0
\end{cases}
\mbox{and }\quad
\begin{cases}
-\nabla\cdot h\nabla \Phi^{(j+3)}_{{\mathcal I}}=({\bf r}_G\times N_{\rm w})_j \quad \mbox{on}\quad {\mathcal I}\\
\Phi^{(j+3)}_{{\mathcal I}}\,_{\vert_\Gamma}=0.
\end{cases}
$$
\end{definition}
\begin{definition}\label{defmassiner} The \emph{mass-inertia} matrix is the (time-dependent) $6\times 6$ block diagonal matrix defined as
$$
{\mathcal M}(t):=\mbox{\rm diag}\big( {\mathfrak m}\mbox{\rm Id}_{3\times 3},{\mathfrak I}(t)\big).
$$
Using the elementary potentials introduced in Definition \ref{defelempot}, we define the \emph{added mass-inertia matrix} as
$$
{\mathcal M}_{\rm a}[h,\boldsymbol{\Phi}_{\mathcal I}]:=\rho \big(\int_{{\mathcal I}}  \frac{1}{h}(h\nabla \Phi^{(j)}_ {{\mathcal I}}) \cdot (h\nabla \Phi^{(k)}_{{\mathcal I}}) \big)_{1\leq j,k\leq 6},
$$
with the notation $\boldsymbol{\Phi}_{\mathcal I}=\big(\Phi^{(1)}_ {{\mathcal I}},\dots, \Phi^{(6)}_ {{\mathcal I}}\big)$. \\
If $S_{\rm i}$ is a $\R^3$-valued function defined on $\cI$, we also define ${\mathcal F}[h,\boldsymbol{\Phi}_{\mathcal I}]S_{\rm i}\in \R^6$ as
$$
{\mathcal F}[h,\boldsymbol{\Phi}_{\mathcal I}]S_{\rm i}=-\rho \big(\int_{{\mathcal I}}  \frac{1}{h}(h\nabla \Phi^{(j)}_ {{\mathcal I}}) \cdot S_{\rm i} \big)_{1\leq j\leq 6}.
$$
\end{definition}
We can now state the following proposition describing the influence of the fluid on the solid motion. Note that the equation for the motion of the solid body are given in the inertial Eulerian frame. It might be convenient in some situation to write these equations in a reference frame moving with the body. Such a formulation is provided in Appendix \ref{appbodyframe}.
\begin{proposition}\label{propfloat}
For a freely floating body, the water waves equation with a floating structure  take the form
$$
\begin{cases}
\dsp \dt \zeta+\nabla\cdot Q=0,\\
\dsp \dt Q+\nabla\cdot (\frac{1}{h}Q\otimes Q)+gh \nabla\zeta+\nabla\cdot {\bf R}(h,Q)+h{\bf a}_{\rm NH}(h,Q)=S^{\rm I}+S^{\rm II}+S^{\rm III},
\end{cases}
$$
with the coupling conditions at the contact line
$$
\zeta_{\rm e}=\zeta_{\rm i} \quad \mbox{ and }\quad Q_{\rm e}=Q_{\rm i} \quad\mbox{ on }\quad \Gamma(t),
$$
and with the source terms  $S^j_{\rm i}$ ($j={\rm I,II,III}$) as in Proposition \ref{proppresc}. Moreover, the velocity $\bug$ of the center of mass and the angular velocity $\bom$ satisfy the ODE 
$$\big({\mathcal M}+{\mathcal M}_{\rm a}[h,\boldsymbol{\Phi}_{{\mathcal I}(t)}]\big)\left(\begin{array}{c} {\dbug} \\ \dot{\bom} \end{array}\right)=\left(\begin{array}{c} -{\mathfrak m}g {\bf e}_z \\
{\mathfrak I}\bom\times \bom 
\end{array}\right)
+{\mathcal F}[h,\boldsymbol{\Phi}_{{\mathcal I}(t)}](S_{\rm i}^{\rm I}+S_{\rm i}^{\rm III}\big).
$$
In particular, one has conservation of the total energy,
$$
\frac{d}{dt}E_{\rm tot}=0 \quad \mbox{ with }\quad E_{\rm tot}:=E_{\rm fluid}+E_{\rm solid}
$$
where $E_{\rm fluid}$ is as in \eqref{defEfluid} while $E_{\rm solid}$ is given by
$$
E_{\rm solid}={\mathfrak m}g z_G+\frac{1}{2}{\mathcal M}\left(\begin{array}{c} {\bug} \\ {\bom} \end{array}\right)\cdot \left(\begin{array}{c} {\bug} \\ {\bom} \end{array}\right).
$$
\end{proposition}
\begin{remark}
It is shown in Appendix \ref{appadded} that ${\mathcal M}_{\rm a}[h,\boldsymbol{\Phi}_{{\mathcal I}}]$ is not the exact added mass-inertia matrix -- indeed, some of the terms in  ${\mathcal F}[h,\boldsymbol{\Phi}_{\mathcal I}](S_{\rm i}^{\rm I}+S_{\rm i}^{\rm III}\big)$ contribute to it. It is however much more convenient to work with the present form, in particular for the derivation of simplified asymptotic models in shallow water in Section \ref{sectAsfloat}.
\end{remark}
\begin{remark}
In order to compute the elementary potentials that appear in the expression for the added mass, one needs to solve a $d$-dimensional elliptic problem in the (bounded) interior region $\cI(t)$. This has to be compared with the $(d+1)$-dimensional elliptic equation one has to solve in the (unbounded) fluid region $\Omega(t)$ in order to compute the Kirchoff potential that appear classically in the expression for the added mass (see for inctance \cite{GST,GMS}).
 \end{remark}
\begin{remark}
As we shall see in \S \ref{sectSW}, Archimedes' force is contained in ${\mathcal F}[h,\boldsymbol{\Phi}_{{\mathcal I}(t)}]S_{\rm i}^{\rm I}$.
\end{remark}
\begin{proof}
The first step is to rewrite the equation for the angular momentum \eqref{Newton2} under the form
\begin{equation}\label{identangular}
{\mathfrak I}\dot\bom ={\mathfrak I}\bom \times \bom+T_{\rm fluid};
\end{equation}
this is a classical computation in solid mechanics that we reproduce for the sake of completeness. Using \eqref{eqI} and \eqref{eqTheta}, one has
\begin{align*}
\frac{d}{dt}({\mathfrak I}\bom)&=\dot\Theta {\mathfrak I}(0)\Theta^T\bom+\Theta{\mathfrak I}(0)\dot{\Theta}^T\bom+{\mathfrak I}\dot\bom\\
&=\bom\times {\mathfrak I}\bom-{\mathfrak I}\dot\Theta\Theta^T\bom+{\mathfrak I}\dot\bom.
\end{align*}
Since $\dot{\Theta}\Theta^T\bom=\bom\times (\Theta\Theta^T\bom)=\bom\times\bom=0$, \eqref{identangular} follows.\\
Next, with the notations of Proposition \ref{proppresc}, the interior pressure can be decomposed as
$$
\uP_{\rm i}=\uP_{\rm i}^{\rm I}+\uP_{\rm i}^{\rm II}+\uP_{\rm i}^{\rm III};
$$
  we can accordingly decompose the force $F_{\rm fluid}$ and the torque $T_{\rm fluid}$ as
$$
F_{\rm fluid}=F_{\rm fluid}^{\rm I}+F_{\rm fluid}^{\rm II}+F_{\rm fluid}^{\rm III}
\quad\mbox{ and }\quad
T_{\rm fluid}=T_{\rm fluid}^{\rm I}+T_{\rm fluid}^{\rm II}+T_{\rm fluid}^{\rm III},
$$
with
$$
F_{\rm fluid}^{\rm I}=\int_{{\mathcal I}(t)}(\uP_{\rm i}^{\rm I}-P_{\rm atm}) N_{\rm w} 
\quad\mbox{ and }\quad
T_{\rm fluid}^{\rm I}=\int_{{\mathcal I}(t)} (\uP_{\rm i}^{\rm I}-P_{\rm atm}){\bf r}_G\times N_{\rm w},
$$
and
$$
F_{\rm fluid}^{j}=\int_{{\mathcal I}(t)}\uP_{\rm i}^{j} N_{\rm w} 
\quad\mbox{ and }\quad
T_{\rm fluid}^{j}=\int_{{\mathcal I}(t)} \uP_{\rm i}^{j}{\bf r}_G\times N_{\rm w}
\qquad (j={\rm II, III}).
$$
In order to rewrite \eqref{Newton1} and \eqref{identangular} under the desired form, the only things to prove are therefore that
\begin{equation}\label{pkey}
\left(\begin{array}{c} F_{\rm fluid}^{\rm II} \\ T_{\rm fluid}^{\rm II} \end{array}\right)=-{\mathcal M}_{\rm a}[h,\boldsymbol{\Phi}_{{\mathcal I}(t)}] \left(\begin{array}{c} \dbug \\ \dot\bom \end{array}\right)
\end{equation}
and 
\begin{equation}\label{pkey2}
\left(\begin{array}{c} F_{\rm fluid}^{\rm j} \\ T_{\rm fluid}^{\rm j} \end{array}\right)=-\frac{1}{\rho}{\mathcal F}[h,\boldsymbol{\Phi}_{\mathcal I}]\nabla \uP_{\rm i}^{j}\quad (j={\rm I}, {\rm III}).
\end{equation}
By definition of the elementary potentials, we have
\begin{align*}
F_{\rm fluid}^{\rm II}&=-\sum_{j=1}^3\int_{{\mathcal I}(t)} P_{\rm i}^{\rm II} \nabla\cdot (h\nabla \Phi^{(j)}_ {{\mathcal I}(t)}){\bf e}_j,
\end{align*}
so that, after integration by parts, we get
\begin{align*}
F^{\rm II}
&=-\sum_{j=1}^3\int_{{\mathcal I}(t)} \nabla\cdot (h\nabla P_{\rm i}^{\rm II})  \Phi^{(j)}_ {{\mathcal I}(t)}{\bf e}_j\\
&=-\rho \sum_{j=1}^3\int_{{\mathcal I}(t)} (\dbug+\dot\bom\times {\bf r}_G)\cdot N  \Phi^{(j)}_ {{\mathcal I}(t)}{\bf e}_j,
\end{align*}
where we used the definition of $P_{\rm i}^{\rm II}$ given in Proposition \ref{proppresc} to derive the second identity. Using again the definition of the elementary potentials, we get further that
\begin{align*}
F_{\rm fluid}^{\rm II}
=\rho \sum_{j=1}^3\sum_{k=1}^3 &\Big[\big(\int_{{\mathcal I}(t)}   \Phi^{(j)}_ {{\mathcal I}(t)} \nabla\cdot(h\nabla \Phi^{(k)}_{{\mathcal I}(t)})  {\bf e}_j\otimes {\bf e}_k\big) \dbug \\
&+
\big(\int_{{\mathcal I}(t)}   \Phi^{(j)}_ {{\mathcal I}(t)} \nabla\cdot(h\nabla \Phi^{(k+3)}_{{\mathcal I}(t)})  {\bf e}_j\otimes {\bf e}_k\big) \dot\bom\Big]
\end{align*}
and, after integration by parts,
\begin{align*}
F_{\rm fluid}^{\rm II}=-\rho \sum_{j=1}^3\sum_{k=1}^3&\Big[ \big(\int_{{\mathcal I}(t)}   \nabla \Phi^{(j)}_ {{\mathcal I}(t)} \cdot h\nabla \Phi^{(k)}_{{\mathcal I}(t)}  {\bf e}_j\otimes {\bf e}_k\big) \dbug  \\
&+
\big(\int_{{\mathcal I}(t)}   \nabla \Phi^{(j)}_ {{\mathcal I}(t)} \cdot h\nabla \Phi^{(k+3)}_{{\mathcal I}(t)})  {\bf e}_j\otimes {\bf e}_k\big) \dot\bom\Big].
\end{align*}
Proceeding similarly for the torque, we have
\begin{align*}
T_{\rm fluid}^{\rm II}&=-\sum_{j=1}^3 \int_{{\mathcal I}(t)}P_{\rm i}^{\rm II}\nabla\cdot (h\nabla \Phi_{{\mathcal I}(t)}^{j+3}){\bf e}_j;
\end{align*}
integrating by parts and proceeding as above, we then get
\begin{align*}
T_{\rm fluid}^{\rm II}&=-\rho \sum_{j=1}^3\int_{{\mathcal I}(t)} (\dbug+\dot\bom\times {\bf r})\cdot N  \Phi^{(j+3)}_ {{\mathcal I}(t)}{\bf e}_{j}
\end{align*}
and using again the definition of the elementary potentials, we finally get
\begin{align*}
T_{\rm fluid}^{\rm II}=-\rho \sum_{j=1}^3\sum_{k=1}^3&\Big[ \big(\int_{{\mathcal I}(t)}   \nabla \Phi^{(j+3)}_ {{\mathcal I}(t)} \cdot h\nabla \Phi^{(k)}_{{\mathcal I}(t)}  {\bf e}_j\otimes {\bf e}_k\big) \dbug  \\
&+\big(\int_{{\mathcal I}(t)}   \nabla \Phi^{(j+3)}_ {{\mathcal I}(t)} \cdot h\nabla \Phi^{(k+3)}_{{\mathcal I}(t)})  {\bf e}_j\otimes {\bf e}_k\big) \dot\bom\Big].
\end{align*}
These expressions for $F^{\rm II}_{\rm fluid}$ and $T_{\rm fluid}^{\rm II}$ yield \eqref{pkey}.\\
For \eqref{pkey2}, we just need to remark that
\begin{align*}
 F_{\rm fluid}^{\rm I} = \sum_{j=1}^3\int_\cI \uP_{\rm i}^{\rm I} N_j {\bf e}_j=- \sum_{j=1}^3\int_\cI \uP_{\rm i}^{\rm I} \nabla \cdot h \nabla \Phi_\cI^{(j)}{\bf e}_j
 \end{align*}
 Integrating by parts, we deduce that
 \begin{align*}
 F_{\rm fluid}^{\rm I} 
 &=\sum_{j=1}^3\int_\cI \nabla \uP_{\rm i}^{\rm I}\cdot h \nabla \Phi_\cI^{(j)}{\bf e}_j
 \end{align*}
so that, proceeding similarly for the torque and for the component $\uP_{\rm i}^{\rm III}$ of the pressure, \eqref{pkey2} follows easily.\\
In order to prove the conservation of energy, let us recall first that owing to Remark \ref{remNRG2}, one has
$$
\frac{d}{dt}E_{\rm fluid}=-\int_{\cI}\dt \zeta \frac{\uP_{\rm i}-P_{\rm atm}}{\rho},
$$
and we therefore turn to compute the time derivative of $E_{\rm solid}$. One gets
\begin{align*}
\frac{d}{dt} E_{\rm solid}&={\mathfrak m}g w_G+{\mathcal M} \left(\begin{array}{c} {\bug} \\ {\bom} \end{array}\right)\cdot \left(\begin{array}{c} {\dbug} \\ \dot{\bom} \end{array}\right)+\frac{1}{2}\dot{\mathcal M}\left(\begin{array}{c} {\bug} \\ {\bom} \end{array}\right)\cdot \left(\begin{array}{c} {\bug} \\ {\bom} \end{array}\right)\\
&=\left(\begin{array}{c} {\bug} \\ {\bom} \end{array}\right)\cdot \big[{\mathcal M}\left(\begin{array}{c} {\dbug} \\ \dot{\bom} \end{array}\right)+mg {\bf e}_z\big],
\end{align*}
where we used the fact that $\dot{{\mathfrak I}}\bom\cdot \bom=0$. By Newton's laws, this gives
\begin{align*}
\frac{d}{dt} E_{\rm solid}&=\left(\begin{array}{c} {\bug} \\ {\bom} \end{array}\right)\cdot  \int_\cI \frac{\uP_{\rm i}-P_{\rm atm}}{\rho} \left(\begin{array}{c}
N_{\rm w}\\{\bf r}_G\times N_{\rm w} \end{array}\right)\\
&=\int_\cI \frac{\uP_{\rm i}-P_{\rm atm}}{\rho} \uU_{\rm w}\cdot N_{\rm w},
\end{align*}
the last line stemming from \eqref{defUc}. Since $\uU_{\rm w}\cdot N_{\rm w}=\dt \zeta$ in the interior region $\cI$, one has $\frac{d}{dt} E_{\rm solid}=-\frac{d}{dt} E_{\rm fluid}$ and the result follows\footnote{The result could also be inferred from the fact the if we write the fluid equation in Zakharov variables $(\zeta,\psi)$, the equations on $(\zeta,\psi,G,\Theta)$ are formally Hamiltonian \cite{DGZ}.}.
\end{proof}

\subsubsection{Simplifications in the one dimensional case} \label{sect1dfloat}

When the horizontal dimension $d$ is equal to $1$, the velocity of the center of mass has no transverse component, $\bug=(u_G,0,w_G)$, and $\bom=(0,\omega,0)$ is perpendicular to the $(x,z)$ plane; the inertia matrix is given by ${\mathfrak I}=\mbox{diag}(0,{\mathfrak i}_0,0)$, with ${\mathfrak i}_0$ independent of time and the rotation matrix $\Theta$ takes the form
$$
\Theta(t)=\left(\begin{array}{ccc}
\cos \theta(t) & 0 &- \sin \theta(t) \\
0 & 1 & 0\\
\sin \theta(t) & 0 &\cos \theta(t) 
\end{array}\right)
$$
and one has $\omega=-\dot \theta$ (this sign convention ensures that $\theta$ is orientated according to the standard trigonometric convention in the plane $(Oxz)$). Newton's laws therefore reduce to
\begin{align}
\label{Newton1_1d}
{\mathfrak m} \left(\begin{array}{c} \dot u_G \\ \dot w_G \end{array}\right)&=-{\mathfrak m} g{\bf e}_z+F_{\rm fluid}\\
\label{Newton2_1d}
{\mathfrak i}_0\dot{\omega}&=T_{\rm fluid}
\end{align}
with
\begin{equation}\label{defFT_1d}
F_{\rm fluid}=\int_{{\mathcal I}(t)}(\uP_{\rm i}-P_{\rm atm}) N_{\rm w}
\quad\mbox{ and }\quad
T_{\rm fluid}=-\int_{{\mathcal I}(t)} (\uP_{\rm i}-P_{\rm atm}){\bf r}_G^\perp\cdot N_{\rm w},
\end{equation}
and where ${\bf r}_G=(x-x_G,\zeta_{\rm w}-z_g)^T$ and $N_{\rm w}=(-\dx\zeta_{\rm w},1)^T$. The mass-inertia matrix is now a $3\times3$ diagonal matrix independent of time,
$$
{\mathcal M}_0=\mbox{\rm diag}({\mathfrak m}, {\mathfrak m}, {\mathfrak i}_0).
$$
In addition to the simplifications already seen in Proposition \ref{propWWCNU}, the elementary potentials can be computed explicitly in dimension $d=1$; consequently, the force and torque exerted by the fluid on the solid take a much simpler form. 
Denoting
$$
{\bf T}({\bf r}_G)=\left(\begin{array}{c} -{\bf r}_G^\perp \\ \frac{1}{2}\abs{{\bf r}_G}^2\end{array}\right)
$$
we can define (recall that the definition of the oscillating component $f^*$ of a function $f$ has been given in Notation \ref{notav})
$$
\widetilde{\mathcal M}_{\rm a}[h,{\bf r}_G]=\int_{\cI}\frac{1}{h}{\bf T}( {\bf r}_G)^*\otimes {\bf T}({\bf r}_G)^*
\quad\mbox{ and }\quad
\widetilde{\mathcal F}[h,{\bf r}_G]S_{\rm i}^j=\int_{\cI}\frac{1}{h}S_{\rm i}^j {\bf T}({\bf r}_G)^*;
$$
we can now state the following proposition describing the interaction of water waves with a freely floating object.
\begin{proposition}
Assume that $d=1$ and that the body is freely floating.  The water waves equations with a floating structure then take the form
$$
\begin{cases}
\dsp \dt \zeta+\dx q=0,\\
\dsp \dt q+\dx (\frac{1}{h} q^2)+gh \dx\zeta+\dx {\bf R}(h,\ovu)+h{\bf a}_{\rm NH}(h,\ovu)=S^{\rm I}+S^{\rm II}+S^{\rm III},
\end{cases}
$$
with the coupling conditions at the contact points
$$
\zeta_{\rm e}=\zeta_{\rm i} \quad \mbox{ and }\quad q_{\rm e}=q_{\rm i} \quad\mbox{ at }\quad x=x_\pm(t)
$$
and with the source terms $S^{\rm I}$, $S^{\rm II}$ and $S^{\rm III}$ as in Proposition \ref{propWWCNU}. Moreover, the velocity of the center of mass $\bug$ and the angular velocity $\omega$ satisfy the ODE
$$
\big({\mathcal M}_0+\widetilde{\mathcal M}_{\rm a}[h,{\bf r}_G]\big)\left(\begin{array}{c} \dbug\\\dot\omega\end{array}\right)=
\left(\begin{array}{c}-{\mathfrak m} g{\bf e}_z\\
0
\end{array}\right)+\widetilde{\mathcal F}[h,{\bf r}_G](S_{\rm i}^{\rm I}+S_{\rm i}^{\rm III}).
$$
\end{proposition}
\begin{proof}
In dimension $d=1$, only three elementary potentials are necessary; relabelling for the sake of simplicity, these potentials are given by the equations on $\cI(t)=(x_-(t),x_+(t))$,
$$
-\dx (h \dx \Phi_\cI^{(1)})=-\dx \zeta_{\rm w},\qquad -\dx (h \dx \Phi_\cI^{(2)})=1,\qquad -\dx (h \dx \Phi_\cI^{(3)})=-{\bf r}_G^\perp\cdot N_{\rm w},
$$
with the boundary conditions ${\Phi_\cI^{(j)}}=0$ at $x=x_\pm(t)$ ($j=1,2,3$).
 We therefore write $\boldsymbol{\Phi}_{\cI}$ the three-dimensional vector with coordinates $(\boldsymbol{\Phi}_{\cI})_j=\Phi_{\cI}^{(j)}$. 
A straightforward adaptation of Proposition \ref{propfloat} to the one dimensional case shows that \eqref{Newton1_1d}-\eqref{Newton2_1d} can be put under the form
$$
\big({\mathcal M}_0+{\mathcal M}_{\rm a}[h,\boldsymbol{\Phi}_{\cI}]\big)\left(\begin{array}{c} \dbug\\\dot\omega\end{array}\right)=
\left(\begin{array}{c}-{\mathfrak m} g{\bf e}_z\\
0
\end{array}\right)+{\mathcal F}[h,\boldsymbol{\Phi}_\cI](S_{\rm i}^{\rm I}+S_{\rm i}^{\rm III})
$$
with 
$$
{\mathcal M}_{\rm a}[h,\boldsymbol{\Phi}_{\cI}]=\rho \int_{\cI}\frac{1}{h} (h\dx \boldsymbol{\Phi}_\cI) \otimes (h\dx \boldsymbol{\Phi}_\cI)
$$
and
$$
{\mathcal F}[h,\boldsymbol{\Phi}_\cI]S_{\rm i}^{\rm I}=-\int_{\cI} \frac{1}{h}(h\dx \uP_{\rm i}^{\rm I}) (h\dx \boldsymbol{\Phi}_\cI),\qquad
{\mathcal F}[h,\boldsymbol{\Phi}_\cI]S_{\rm i}^{\rm III}=-\int_{\cI}\frac{1}{h}h\dx \uP_{\rm i}^{\rm III}(h\dx \boldsymbol{\Phi}_\cI)  .
$$
Using the definition of ${\boldsymbol \Phi}_\cI$ and Lemma \ref{lemexplicit}, and using the Notation \ref{notav}, one gets the following expression for $-h\dx {\boldsymbol \Phi}_\cI$,
$$
-h\dx \boldsymbol{\Phi}_{\cI}=\left(\begin{array}{c} ({\bf r_G}^*)^\perp \\ -\frac{1}{2}(\abs{{\bf r}_G}^2)^*\end{array}\right)
$$
so that ${\mathcal M}_{\rm a}(h,\boldsymbol{\Phi}_{\cI})=\widetilde{\mathcal M}_{\rm a}[h,{\bf r}_G]$ 
and ${\mathcal F}[h,\boldsymbol{\Phi}_\cI]S_{\rm i}^{j}=\widetilde{\mathcal F}[h,{\bf r}_G]S_{\rm i}^{j}$, and the proposition follows.
\end{proof}
\section{Comments on the evolution of the contact line}\label{sectcontact}

The evolution of the contact line $\Gamma(t)$, and therefore of the interior and exterior domains ${\mathcal I}(t)$ and ${\mathcal E}(t)$, is governed by the equations \eqref{Eul1}-\eqref{contact1}; this evolution is however quite implicit, and the goal of this section is to derive more explicit formulations of this evolution. We first consider the one dimensional case $d=1$ and then turn to the general two dimensional situation.

\subsection{Evolution of the contact line in the one dimensional case ($d=1$)}\label{sectcontact1d}

Assuming that the wetted surface is connected, one can write the interior domain as an interval
$$
{\mathcal I}(t)=\big(x_-(t),x_+(t)\big),
$$
and we need to find the time evolution of the boundary points $x_\pm(t)$. The following proposition gives an expression for the time derivative $\dot{x}_\pm(t)$ of $x_\pm(t)$ in terms of the position and velocity of the center of mass,  and of the angular velocity of the solid.
\begin{proposition}
Denoting by $G=(x_G,z_G)$ the center of mass of the solid and by $\omega$ its angular velocity, the contact points $x_\pm$ satisfy the nonlinear ODEs in time
\begin{align*}
\dot{x}_{\pm}=&-\frac{(\dx \zeta_{\rm w})_\pm}{(\dx \zeta_{\rm e})_\pm-(\dx \zeta_{\rm w})_\pm}\big(\dot{x}_G+\omega({\zeta_{{\rm w},\pm}}-z_G)\big) \\
&+\frac{1}{(\dx \zeta_{\rm e})_\pm-(\dx \zeta_{\rm w})_\pm}\big(
\dot{z}_G-\omega(x_\pm-x_G)+(\dx q_{\rm e})_\pm \big),
\end{align*}
where for any function $f(t,x)$, we used the notation $f_\pm(t)=f(t,x_\pm(t))$.
\end{proposition}
\begin{remark}
It is instructive to compare the evolution equation of the proposition to the equation describing the evolution of the shoreline in the case of a vanishing depth. In this latter case, the condition \eqref{contact3bis1d} below, should be replaced by $h_{\rm e}(t,\gamma(t,\alpha))=0$. Following the same steps as in the proof below, one would obtain 
\begin{align*}
\dot{x}_\pm (\dx \zeta_{\rm e})_\pm&=(\dx q_{\rm e})_\pm\\
&=(\dx h_{\rm e})_\pm \ovu_\pm+ h_{\rm e,\pm}(\dx \ovu)_\pm.
\end{align*}
Since by definition $h_{\rm e,\pm}=0$, one obtains the kinematic equation
$$
\dot{x}_\pm=\ovu.
$$
The evolution equation of the contact line stated in the proposition involves derivatives of $\zeta_{\rm e}$ and $q_{\rm e}$ and is therefore more singular than the kinematic equation obtained for the evolution of the shoreline.
\end{remark}
\begin{proof}
By definition of $x_\pm(t)$ and using the boundary condition \eqref{contact3} and the constraint \eqref{constraint}, one gets that for all $t$,
\begin{equation}\label{contact3bis1d}
h_{\rm e}(t,x_\pm(t))=h_{\rm w}(t,x_\pm(t)).
\end{equation}
Differentiating this relation yields
$$
\dot{x}_\pm\big((\dx h_{\rm e})_\pm - (\dx h_{\rm w})_\pm\big)=- \big((\dt \zeta_{\rm e})_\pm - (\dt \zeta_{\rm w})_\pm\big).
$$
Using the first equation of \eqref{Eulerav2}, one can replace $\dt h_{\rm e}=-\dx q_{\rm e}$ so that
$$
\dot{x}_\pm\big((\dx \zeta_{\rm e})_\pm - (\dx \zeta_{\rm w})_\pm\big)=(\dt h_{\rm w})_\pm+(\dx q_{\rm e})_\pm.
$$
We also know from \eqref{eqdtzeta} that 
$$
\dt h_{\rm w}=-\big(\dot{x}_G+\omega(\zeta-z_G)\big)\dx \zeta_{\rm w}+\dot{z}_G-\omega(x-x_G),
$$
 so that the formula of the proposition follows easily.
\end{proof}

\subsection{Evolution of the contact line in the two dimensional case ($d=2$)}\label{sectcontact2d}

Assuming that the wetted surface is connected and that the boundary $\Gamma(t)$ of the interior domain can be parametrized by a closed curve $\gamma(t,\cdot):[0,1]\to \R^2$, namely,
$$
\Gamma(t):=\{\gamma(t,\alpha),\alpha\in [0,1]\},
$$
we need to determine the time evolution of $\gamma$.
\begin{proposition}\label{propmvt2}
Denote by $G=(X_G,z_G)$ the center of mass of the solid and by $\bom$ its angular velocity, and assume that on the time interval $[0,T]$ the contact line 
$\Gamma(t)$ is parametrized by a $C^2$ function $\gamma:[0,T]\times[0,1]\to \R^2$, regular everywhere (i.e. $\partial_\alpha\gamma$ never vanishes).\\
{\bf i.} The function $\gamma$ solves
\begin{align}
\nonumber
\dt \gamma=&-\left\lbrace\frac{(\nabla\zeta_{\rm e}-\nabla\zeta_{\rm w})\otimes\nabla\zeta_{\rm w}}{\Abs{\nabla \zeta_{\rm e}-\nabla \zeta_{\rm w}}^2}\right\rbrace_{\vert_\gamma} \big(\dot{X}_G-\bom_{\rm h}^\perp({\zeta_{\rm w}}_{\vert_\gamma}-z_G)+\omega_{\rm v}(\gamma-X_G)^\perp\big)\\
\label{equatgamma}
&+\left\lbrace\frac{(\nabla\zeta_{\rm e}-\nabla\zeta_{\rm w})}{\Abs{\nabla \zeta_{\rm e}-\nabla \zeta_{\rm w}}^2}\right\rbrace_{\vert_\gamma}
\big(\dot{z}_G+\bom_{\rm h}^\perp\cdot (\gamma-X_G)+(\nabla\cdot Q_{\rm e})_{\vert_{\gamma}}\big)+a\partial_\alpha\gamma,
\end{align}
for some scalar function $a\in C^1([0,T]\times [0,1])$ and where for any function $f(t,X)$, we used the notation $f_{\vert_\gamma}(t,\alpha)=f(t,\gamma(t,\alpha))$. \\
{\bf ii.} Conversely, if there exists a scalar function $a\in C^1([0,T]\times [0,1])$ such that $\gamma$ solves \eqref{equatgamma} and if $\gamma(0,\cdot)$ is a parametrization of the contact line at $t=0$, then it is a parametrization of $\Gamma(t)$ for all times.\\
{\bf iii.} Different choices of the scalar function $a$ in \eqref{equatgamma} correspond to different parametrizations of the same curve.
\end{proposition}
\begin{remark}
Choosing a particular function $a$ in \eqref{equatgamma} is equivalent to choosing a particular parametrization for the curve $\Gamma(t)$. For instance, if it can be parametrized as a polar curve by choosing
$$
\gamma(t,\alpha):=\big(\rho(t,\alpha)\cos(2\pi \alpha), \rho(t,\alpha)\sin(2\pi\alpha)\big), 
$$
one gets
$$
\dt \rho=-\frac{1}{\partial_\rho \zeta_{\rm e}-\partial_\rho \zeta_{\rm w}}\big(\nabla\cdot Q_{\rm e}-(\bug+\bom\times{\bf r}_G)\cdot N_{\rm w}\big)_{\vert_{\gamma}},
$$
where we used the notation $\partial_\rho =\cos(2\pi\alpha)\dx \zeta+\sin(2\pi\alpha)\dy$. The  polar parametrization is therefore unique; it corresponds to a particular choice of the function $a$ in Proposition \ref{propmvt2}.
\end{remark}
\begin{proof}
By definition of $\gamma$ and using the boundary condition \eqref{contact3} and the constraint \eqref{constraint}, one gets that for all $t$ and all $\alpha\in [0,1]$,
\begin{equation}\label{contact3bis}
h_{\rm e}(t,\gamma(t,\alpha))=h_{\rm w}(t,\gamma(t,\alpha)).
\end{equation}
Differentiating this relation with respect to $t$ gives
$$
\dt \zeta_{\rm e}+\dt\gamma\cdot \nabla h_{\rm e}=\dt \zeta_{\rm w}+\dt\gamma\cdot \nabla h_{\rm w}
$$
so that we can write
$$
\dt \gamma=-\frac{\nabla h_{\rm e}-\nabla h_{\rm w}}{\Abs{\nabla h_{\rm e}-\nabla h_{\rm w}}^2}(\dt \zeta_{\rm e}-\dt \zeta_{\rm w})+f(\nabla h_{\rm e}-\nabla h_{\rm w})^\perp
$$
for some scalar function $f$. Differentiating \eqref{contact3bis} with respect to $\alpha$, we get that
$$
\partial_\alpha\gamma\cdot (\nabla h_{\rm e}-\nabla h_{\rm w})=0,
$$
so that $\partial_\alpha \gamma$ is proportional to $(\nabla h_{\rm e}-\nabla h_{\rm w})^\perp$. Since $\partial_\alpha\gamma\neq 0$, we deduce from the above that
\begin{equation}\label{eqgamma}
\dt \gamma=-\big[\frac{\nabla h_{\rm e}-\nabla h_{\rm i}}{\Abs{\nabla h_{\rm e}-\nabla h_{\rm i}}^2}(\dt \zeta_{\rm e}-\dt \zeta_{\rm w})\big]_{\vert_\gamma}+a\partial_\alpha \gamma
\end{equation}
for some scalar function $a$. Taking the scalar product of this expression with $\partial_\alpha\gamma$ and using the regularity assumptions made on $\gamma$, we deduce that $a$ is $C^1$ in space and time.\\
 Conversely, if $\gamma$ solves an equation of the form \eqref{eqgamma}, and if $\gamma(0,\cdot)$ is a parametrization of the contact line at $t=0$,  then one easily gets that \eqref{contact3bis} holds for all times, so that $\gamma(t,\cdot)$ is a parametrization of the contact line for all times.\\
Let us show now that different choices of $a$ in \eqref{eqgamma} correspond to different parametrizations of the $\Gamma(t)$. More precisely, let $a$ and $\tilde a$ be two $C^1$ functions of space and time, and let us show that there exists a reparametrization $\varphi(t,\cdot):[0,1]\to [0,1]$ such that if $\gamma$ solves \eqref{eqgamma} then $\tilde\gamma(t,\alpha):=\gamma(t,\varphi(t,\alpha))$ solves \eqref{eqgamma} with $a$ replaced by $\tilde a$. From the definition of $\tilde\gamma$, one has
\begin{align*}
\dt \tilde\gamma&=\dt\gamma\circ \varphi+\dt\varphi \partial_\alpha \gamma\circ\varphi\\
&=-\big[\frac{\nabla h_{\rm e}-\nabla h_{\rm w}}{\Abs{\nabla h_{\rm e}-\nabla h_{\rm w}}^2}(\dt \zeta_{\rm e}-\dt \zeta_{\rm w})\big]_{\vert_{\tilde\gamma}}+\big(a\circ\varphi+\dt \varphi\big) \partial_\alpha \gamma\circ\varphi,
\end{align*}
so that the claim is proved by solving the ODE $\dt\varphi=b\circ\varphi-a\circ \varphi$ and taking $\tilde\alpha=\varphi(t,\alpha)$ as new parameter.\\
The last step of the proof consists in showing that \eqref{eqgamma} can be put under the form \eqref{equatgamma}. Using the relation
$$
\dt \zeta_{\rm e}=-\nabla\cdot Q_{\rm e}, 
$$
which corresponds to the first equation of \eqref{Eulerav2}, together with  \eqref{eqdtzeta},
\begin{align*}
\dt \zeta_{\rm w}&=\big(\bU_G+\bom\times {\bf r}_G\big)\cdot N_{\rm w}\\
&=-\big(\dot{X}_G-\bom_{\rm h}^\perp(\zeta-z_G)+\omega_{\rm v}(X-X_G)^\perp\big)\cdot \nabla\zeta_{\rm w}+\big(\dot{z}_G+\bom_{\rm h}^\perp\cdot (X-X_G)\big),
\end{align*}
the result follows directly from \eqref{eqgamma}. 
\end{proof}

\subsection{The case of vertical walls}\label{sectvertical}

The boundary conditions \eqref{contact3}-\eqref{contact1} at the contact line are valid under the condition that in the neighborhood of the contact line, the boundary of the solid is not vertical. Since such a configuration, represented in Figure \ref{figvert}, is also of interest (and we shall use it in Sections \ref{sectdiscrete} and \ref{sectcomput} for the numerical aspects), we show here how to handle it. Of course, the walls would not stay vertical if the solid were allowed to rotate along the horizontal axis, and we must therefore assume that the motion of the solid is constrained (by some additional exterior force) to avoid these situations. The angular velocity is therefore of the form $\bom=(0,0,\omega_{\rm v})$. 
\begin{figure}
\includegraphics[width=\textwidth]{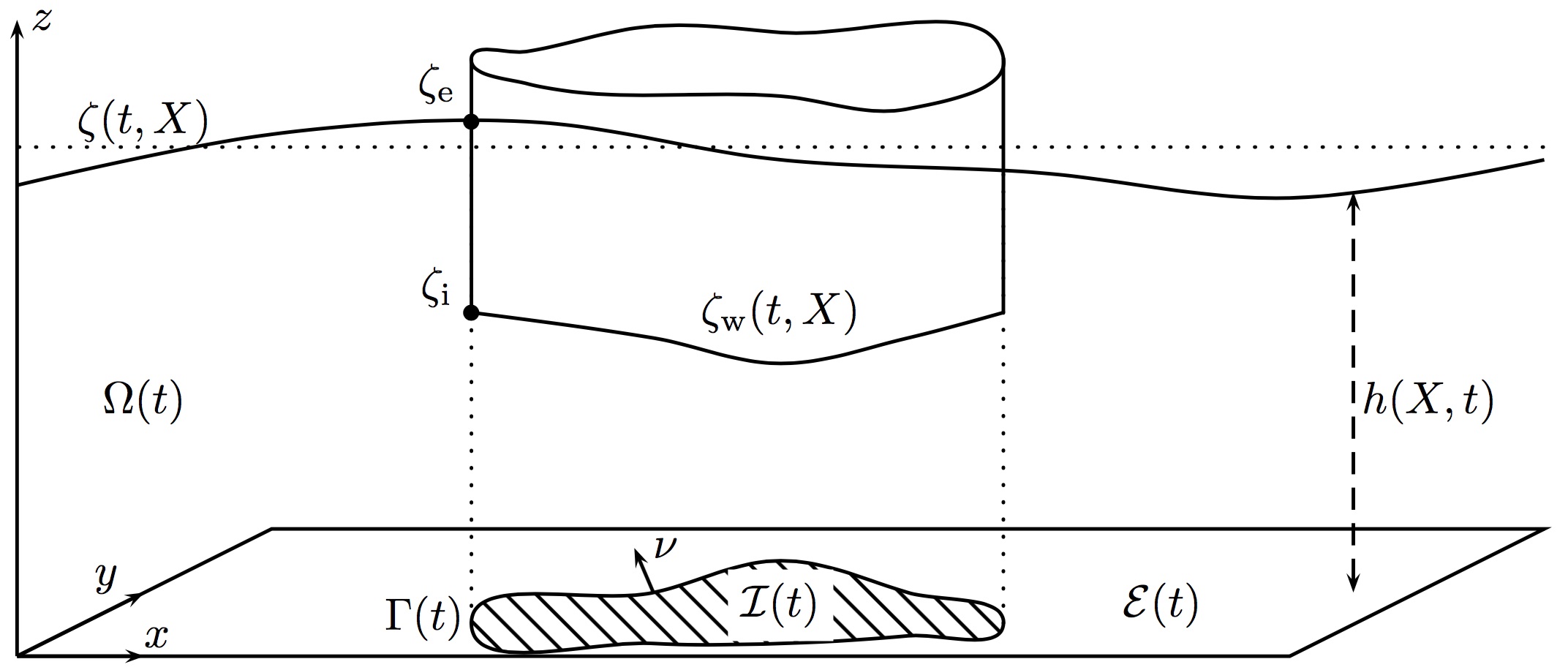}
\caption{The case of vertical walls}
\label{figvert}
\end{figure}

In the presence of vertical walls, we relax the continuity condition \eqref{contact3} on the water elevation, and consequently replace the continuity condition \eqref{contact1} on the pressure by a more general expression. This generalization of the boundary conditions \eqref{contact3}-\eqref{contact1} is the following:
\begin{itemize}
\item {\it Continuity of the normal velocity at the vertical walls}. Denoting by $\nu\in \R^d$ the unit normal vector to $\Gamma(t)$ pointing towards the exerior region ${\mathcal E}(t)$, one has
\begin{equation}\label{contactV}
V\cdot \nu=V_{\mathcal C}\cdot \nu \quad\mbox{ on the immersed part of the vertical walls}, 
\end{equation}
where we recall that $V$ and $V_{{\mathcal C}}$ denote the horizontal velocities of the fluid and of the solid respectively.
\item {\it Consistency of the pressure jump at the contact line}. Integrating the vertical component of Euler's equation \eqref{Eul1} between $z=\zeta_{\rm i}$ and $z=\zeta_{\rm e}$ yields the condition
\begin{equation}\label{contact1_vert}
\uP_{\rm i}(t,\cdot)=P_{\rm atm}+\rho g(\zeta_{\rm e}-\zeta_{\rm i}) +\rho\int_{\zeta_{\rm i}}^{\zeta_{\rm e}}(\dt w+\bU\cdot \nabla_{X,z}w)\quad\mbox{ on }\quad \Gamma(t),
\end{equation}
where $w$ is the vertical component of the velocity field $\bU$ in the fluid domain.
\end{itemize}
Of course, whenever $\zeta_{\rm i}=\zeta_{\rm e}$ (in particular when the boundary at the contact line is not vertical), \eqref{contact1_vert} coincides with \eqref{contact1}. Allowing for the possibility of vertical walls imposes the presence of a fourth source term $S^{\rm IV}$ in the momentum equation of Proposition \ref{proppresc} (and Proposition \ref{propfloat} when the solid is freely floating), and the transition condition on $Q$ at the contact line must be modified. We shall use the following notations.
\begin{notation}\label{notaspec}
 For the sake of simplicity, we still denote by $\big({\mathcal M}+{\mathcal M}_{\rm a}[h,\boldsymbol{\Phi}_{{\mathcal I}}]\big)$ the $4\times 4$ matrix 
with entries $\big({\mathcal M}+{\mathcal M}_{\rm a}[h,\boldsymbol{\Phi}_{{\mathcal I}}]\big)_{ij}$ ($i,j=1,2,3,6$) and by ${\mathcal F}[h,\boldsymbol{\Phi}_{\mathcal I}]$ the four dimensional vector with entries $({\mathcal F}[h,\boldsymbol{\Phi}_{\mathcal I}])_i$ ($i=1,2,3,6$).
\end{notation}
\begin{proposition}\label{proppresc_vert}
Denoting by  $\bug=(V_G,w_G)$ the velocity  of the center of mass of the solid and assuming that the motion is constrained so that its angular velocity is of the form $\bom=(0,0,\omega_{\rm v})$, the water waves equations with a floating structure can be written
$$
\begin{cases}
\dsp \dt \zeta+\nabla\cdot Q=0,\\
\dsp \dt Q+\nabla\cdot (\frac{1}{h} Q\otimes Q)+gh \nabla\zeta+\nabla\cdot {\bf R}(h,Q)+h{\bf a}_{\rm NH}(h,Q)\\
\phantom{\dsp \dt Q+\nabla\cdot (\frac{1}{h} Q\otimes Q)+gh \nabla\zeta+\nabla\cdot {\bf R}(h,Q)}=S^{\rm I}+S^{\rm II}+S^{\rm III}+S^{\rm IV},
\end{cases}
$$
with the coupling condition at the contact line
$$
 \big(Q_{\rm e}-Q_{\rm i}\big)\cdot \nu=(\zeta_{\rm e}-\zeta_{\rm i})\big(V_G+\omega_{\rm v}(X-X_G)^\perp\big)\cdot \nu \quad\mbox{ on }\quad \Gamma(t),
$$
and $S^{\rm I}$, $S^{\rm II}$ and $S^{\rm III}$ as in Proposition \ref{proppresc}, and with  
$S^{\rm IV}_{\rm e}=0$ and 
$$
S^{\rm IV}_{\rm i}=-\frac{h}{\rho}\nabla P_{\rm i}^{\rm IV} \quad\mbox{ where }\quad 
\begin{cases}
-\nabla\cdot (\frac{h}{\rho}\nabla \uP_{\rm i}^{\rm IV})=0 \quad\mbox{\textnormal{on}}\quad {\mathcal I}(t),\\
{\uP_{\rm i}^{\rm IV}}_{\vert_{\Gamma(t)}}=\rho g(\zeta_{\rm e}-\zeta_{\rm i}) +\rho\int_{\zeta_{\rm i}}^{\zeta_{\rm e}}(\dt w+\bU\cdot \nabla_{X,z}w).
\end{cases}
$$
If the solid structure is freely floating, the evolution of $\bug$ and $\bom$ is given by
\begin{align*}
\big({\mathcal M}+{\mathcal M}_{\rm a}[h,\boldsymbol{\Phi}_{{\mathcal I}}]\big)\left(\begin{array}{c} {\dbug} \\ \dot{\omega_{\rm v}} \end{array}\right)=&\left(\begin{array}{c} -{\mathfrak m}g {\bf e}_z \\
0
\end{array}\right)
+{\mathcal F}[h,\boldsymbol{\Phi}_{\mathcal I}](S_{\rm i}^{\rm I}+S_{\rm i}^{\rm III}\big) \\
&+\int_\cI \uP_{\rm i}^{\rm IV}\left(\begin{array}{c}  N_{\rm w}\\   (X-X_G)\cdot \nabla^\perp\zeta_{\rm w}
\end{array}\right).
\end{align*}
\end{proposition}
\begin{remark}\label{remd1vert}
The adatptation of the proposition to the one dimensional case $d=1$ is straighforward. Moreover, the source term $S^{\rm IV}$ can then be computed explicitly; with the notations of Proposition \ref{propWWCNU}, one has
$$
S^{\rm IV}=-\frac{p_+-p_-}{\rho}\frac{1}{\int_{x_-}^{x^+}1/h}
\quad\mbox{ with }\quad \frac{p_{\pm}}{\rho}=\big[ g(\zeta_{\rm e}-\zeta_{\rm i})+\int_{\zeta_{\rm i}}^{\zeta_{\rm e}}(\dt w+\bU\cdot \nabla_{X,z}w)\big]_{\vert_{x=x_\pm}} ,
$$
and the corresponding pressure $\uP^{\rm IV}_{\rm i}$ is
$$
\uP^{\rm IV}_{\rm i}=p_-+(p_+-p_-)\frac{\int_{x_-}^x 1/h}{\int_{x_-}^{x_+}1/h}.
$$
\end{remark}
\begin{proof}
We just prove the transition condition on $Q$ at the contact line. The rest of the proof is a close adaptation of the proof of Propositions \ref{propfloat} and \ref{proppresc}. By definition, one has
$$
Q_{\rm e}(t,X)=\int_{-h_0+b}^{\zeta_{\rm e}(t,x)} V(t,X,z)dz
\quad\mbox{ and }\quad
Q_{\rm i}(t,X)=\int_{-h_0+b}^{\zeta_{\rm i}(t,x)} V(t,X,z)dz.
$$
As for the proof of Proposition \ref{propWWst}, we know that $V$ is smooth in the interior of the fluid region, so that
$$
Q_{\rm e}(t,X)-Q_{\rm i}(t,X)=\int_{\zeta_{\rm i}(t,X)}^{\zeta_{\rm e}(t,X)}V(t,X,z)dz
$$
Taking the scalar product with $\nu$ and using the transition condition \eqref{contactV}, we get
$$
(Q_{\rm e}-Q_{\rm i})\cdot \nu=\int_{\zeta_{\rm i}}^{\zeta_{\rm e}}V_{\mathcal C}\cdot \nu,
$$
which yields the result since $V_{\mathcal C}\cdot \nu$ does not depend on $z$ and because $V_{\mathcal C}=V_G+\omega_{\rm v}(X-X_G)^\perp$.
\end{proof}
\section{Asymptotic models}\label{sectAsfloat}

It is classical in the theory of water waves to derive simpler models from the governing equations. The same approximations lead to simplified versions of  the water waves equations with a floating structure (with the terminology of Definition \ref{defiCWW}) . We shall consider here two important regimes: the nonlinear shallow water equations which is a fully nonlinear model (in the sense that no smallness assumption is made on the size of the waves), and the Boussinesq model which is a weakly nonlinear model, but which takes into account the non-hydrostatic dispersive effects neglected in the nonlinear shallow water equations. The former is studied in \S \ref{sectSW} and the latter in \S \ref{sectBouss}.

\subsection{The shallow water approximation}\label{sectSW}

In absence of any immersed structure, the shallow water approximation consists in performing two approximations\footnote{For the classical water waves equations (without floating body) these approximations are valid at leading order under the assumption that $\mu\ll 1$, where $\mu=\frac{h_0^2}{L^2}$ is the shallowness parameter given by the square of the ratio of the depth over the typical horizontal scale. The nonlinear shallow water equations are obtained by neglecting all the terms of size $O(\mu)$ in the dimensionless water waves equations; see for instance \cite{L_book}.} on the averaged Euler equations \eqref{Eulerav1closed},
\begin{enumerate}
\item Neglect the vertical variations of the horizontal velocity in the quadratic term. This leads to
$$
 \int_{-h_0+b}^\zeta V\otimes V \approx h \ovV\otimes \ovV=\frac{1}{h}Q\otimes Q.
$$
\item Neglect the non-hydrostatic acceleration,
$$
{\bf a}_{\rm NH}\approx 0.
$$
\end{enumerate}
We show in this section how to simplify the water waves equations with a floating structure  \eqref{Eulerav2}-\eqref{coupling}  under these approximations.
The same simplifications as above must be performed on the momentum equation in \eqref{Eulerav2}, but must also be consistently made in \eqref{eqPw} for the computations of the interior pressure $\uP_{\rm i}$. This means that the terms ${\bf R}(\zeta,Q)$ and ${\bf a}_{\rm NH}(\zeta,Q)$ must be neglected in \eqref{Eulerav2}, but also in \eqref{defba}. The resulting  \emph{nonlinear shallow water equations with a floating structure} are given by
\begin{equation}\label{SWf}
\begin{cases}
\dsp \dt \zeta+\nabla\cdot Q=0,\\
\dsp \dt Q+\nabla\cdot (\frac{1}{h} Q\otimes Q)+gh\nabla \zeta=-h \frac{1}{\rho}\nabla \underline{P},
\end{cases}
\end{equation}
where $\uP_{\rm e}=P_{\rm atm}$ in ${\mathcal E}(t)$ and $\uP_{\rm i}$ is given by
\begin{equation}\label{eqPwSW}
\begin{cases}
-\nabla\cdot (\frac{h}{\rho}\nabla \uP_{\rm i})=-\dt^2 \zeta_{\rm w}+\nabla\cdot\big[\nabla\cdot (\frac{1}{h}Q\otimes Q)+gh\nabla\zeta\big] \quad\mbox{in}\quad {\mathcal I}(t),\\
{\uP_{\rm i}}_{\vert_{\Gamma(t)}}=P_{\rm atm},
\end{cases}
\end{equation}
and with the boundary conditions at the contact line
\begin{equation}\label{contSW}
Q_{\rm e}=Q_{\rm i} \quad \mbox{ and }\quad \zeta_{\rm e}=\zeta_{\rm i}\quad\mbox{ on }\Gamma(t).
\end{equation}
\begin{remark}
As for the full water waves equations with a floating structure (see Proposition \ref{propWWst}), the constraint
$$
\zeta=\zeta_{\rm w} \quad\mbox{ on }\quad {\mathcal I}(t)
$$
is satisfied at all time provided that the initial conditions verify
$$
\qquad \zeta^0={\zeta_{\rm
    w}}_{\vert_{t=0}}\quad \mbox{ and }\quad \nabla\cdot Q^0=-\dt {\zeta_{\rm
    w}}_{\vert_{t=0}} \quad\mbox{ on }\quad {\mathcal I}(0),
$$
which we shall always assume.
\end{remark}
\begin{remark}
The equations \eqref{SWf}-\eqref{eqPwSW} can alternatively be written as 
$$
\begin{cases}
\dsp \dt \zeta+\nabla\cdot Q=0,\\
\dsp \dt Q+\nabla\cdot (\frac{1}{h} Q\otimes Q)+gh\nabla \zeta=-h \frac{1}{\rho}\nabla \underline{P},\\
h\leq h_{\rm w}, \qquad (h-h_{\rm w})P=0.
\end{cases}
$$
This is a typical example of {\it congested flow}; in the case $g=0$, this model appears in various contexts such as traffic flows \cite{BDMR,DH}, granular flows \cite{LLM,Perrin,PZ}, hydrodynamics in pipes \cite{BEG}, compressible-low Mach coupling in gaz dynamics \cite{PDD}, etc. As remarked in \cite{DH} the transition conditions on the contact line play a crucial role, and the computation of the evolution of the corresponding free boundary $\Gamma(t)$ is very delicate. For this reason, the "compressible" part of of the equations is often approximated as the limit of a singular incompressible system \cite{BPZ,DH,GPSW}. Our approach offers an alternative to this method; we establish in \cite{IguchiLannes} a well-posedness result for \eqref{SWf}-\eqref{contSW} in the one dimensional case, which describes in particular the evolution of the contact line.
\end{remark}
The simpler form of the elliptic equation for the interior pressure allows us to give a more explicit and more instructive form of the hydrodynamic forces acting on the solid; we shall denote by ${\mathcal F}_{\rm Arch}$ and ${\mathcal F}_{\rm NL}$ the Archimedes\footnote{The standard Archimedes force is the vertical component of the force, given by $-g\int_\cI \zeta$, which is the opposite of the weight of the fluid that the body displaces (with respect to the still water level). Note that this force can be oriented {\it downwards}, for instance  if the object is floating on a large amplitude wave so that $\zeta=\zeta_{\rm w}\geq 0$ in the interior region $\cI(t)$.} and nonlinear force/torque respectively, given by the surfacic integrals
\begin{equation}\label{defForce1}
{\mathcal F}_{\rm Arch}=-g\int_{\cI(t)}\zeta_{}  \left(\begin{array}{c}N_{} \\ {\bf r}_G\times N_{}\end{array}\right)
\quad\mbox{ and }\quad
 {\mathcal F}_{\rm NL}=-\int_{\cI(t)}\nabla\cdot \big(\frac{1}{h_{}}Q\otimes Q\big)\cdot \nabla {\boldsymbol \Phi}_\cI,
\end{equation}
and by ${\mathcal F}_{\rm \Gamma}$ the contribution\footnote{As we shall see in \S \ref{sectsolfloatvert}, this component contains damping forces as well as excitation forces coming from the wave field.} coming from the contact line
\begin{equation}\label{defForce2}
{\mathcal F}_{\rm \Gamma}=-g\int_{\Gamma(t)} h_{} \zeta_{}\partial_\nu {\boldsymbol \Phi}_\cI;
\end{equation}
as usual $\boldsymbol{\Phi}_{\mathcal I}=\big(\Phi^{(1)}_ {{\mathcal I}},\dots, \Phi^{(6)}_ {{\mathcal I}}\big)$ is as in Definition \ref{defelempot}, 
and we denoted by $\nu\in \R^2$  the outward unit normal vector to $\Gamma(t)$.
\begin{proposition}\label{propprescSW}
Denoting by  $\bug$ the velocity  of the center of mass and by $\bom$ the angular velocity, the nonlinear shallow water equations with a floating structure then take the form
$$
\begin{cases}
\dsp \dt \zeta+\nabla\cdot Q=0,\\
\dsp \dt Q+\nabla\cdot (\frac{1}{h} Q\otimes Q)+gh \nabla\zeta=S_{\rm SW}^{\rm I}+S^{\rm II}+S^{\rm III},
\end{cases}
$$
with the coupling conditions at the contact line
$$
 Q_{\rm e}=Q_{\rm i} 
\quad \mbox{ and }\quad
\zeta_{\rm e}=\zeta_{\rm i} 
\quad\mbox{ on }\quad \Gamma(t),
$$
and with the source terms $S^{\rm II}$ and $S^{\rm III}$ as in Proposition \ref{proppresc}, while $S^{\rm I}_{\rm SW}$ is given by
$$
S^{\rm I}_{\rm SW,e}=0 \quad\mbox{ and }\quad S^{\rm I}_{\rm SW,i}=-\frac{h}{\rho}\nabla \uP_{\rm SW,i}^{\rm I}
$$
where
$$
\begin{cases}
-\nabla\cdot (\frac{h}{\rho}\nabla \uP_{\rm SW,i}^{\rm I})=
\nabla\cdot\big[\nabla\cdot (\frac{1}{h} Q\otimes Q)+gh\nabla\zeta\big]  \quad\mbox{\textnormal{on}}\quad {\mathcal I}(t),\\
{\uP_{\rm SW,i}^{\rm I}}_{\vert_{\Gamma(t)}}=P_{\rm atm}.
\end{cases}
$$
In the case where the solid is freely floating, the evolution of $\bug$ and $\bom$ is given by the ODE
$$\big({\mathcal M}+{\mathcal M}_{\rm a}[h,\boldsymbol{\Phi}_{{\mathcal I}}]\big)\left(\begin{array}{c} {\dbug} \\ \dot{\bom} \end{array}\right)=\left(\begin{array}{c} -{\mathfrak m}g {\bf e}_z \\
{\mathfrak I}\bom\times \bom 
\end{array}\right)
+{\mathcal F}_{\rm Arch}+{\mathcal F}_{\rm \Gamma}+{\mathcal F}_{\rm NL}
+{\mathcal F}[h,\boldsymbol{\Phi}_{{\mathcal I}}] S_{\rm i}^{\rm III},
$$
with the same notations as in Proposition \ref{propfloat}. In particular, one has conservation of the total energy,
$$
\frac{d}{dt}E_{\rm SW,tot}=0\quad \mbox{ with }\quad E_{\rm SW,tot}=E_{\rm SW}+E_{\rm solid},
$$
and where $E_{\rm solid}$ is as in Proposition \ref{propfloat} while $E_{\rm SW}$ is given by
$$
E_{\rm SW}=\frac{1}{2}\int_{\R^d} \frac{1}{h}\abs{Q}^2+\frac{1}{2}\int_{\R^d}g \zeta^2.
$$
\end{proposition}
\begin{remark}[Vertical walls] \label{remSWvert}
Following what has been done for the full water waves equations in \S \ref{sectvertical}, it is possible to allow for the possibility of vertical walls by removing the condition $\zeta_{\rm e}=\zeta_{\rm i}$ from the coupling conditions at the contact line, and by adding a source term $S^{\rm IV}_{\rm SW}$ to the momentum equation
with  
$S^{\rm IV}_{\rm SW,e}=0$ and 
$$
S^{\rm IV}_{\rm SW,i}=-\frac{h}{\rho}\nabla P_{\rm SW,i}^{\rm IV} \quad\mbox{ where }\quad 
\begin{cases}
-\nabla\cdot (\frac{h}{\rho}\nabla \uP_{\rm SW,i}^{\rm IV})=0 \quad\mbox{\textnormal{on}}\quad {\mathcal I}(t),\\
{\uP_{\rm SW,i}^{\rm IV}}_{\vert_{\Gamma(t)}}=\rho g(\zeta_{\rm e}-\zeta_{\rm i});
\end{cases}
$$
the difference between $P^{\rm IV}_{\rm i}$ and $P^{\rm IV}_{\rm SW,i}$ is that the non-hydrostatic term $\rho\int_{\zeta_{\rm i}}^{\zeta_{\rm e}}(\dt w+\bU\cdot \nabla_{X,z}w)$ has been neglected in the latter, consistently with the approximations made to derive the nonlinear shallow water equations.
\end{remark}
\begin{proof}
The only point that deserves a proof is the derivation of the ODE for $\bom$ and $\bug$ and the energy conservation. It is given by the same ODE as in Proposition \ref{propfloat}, with $S^{\rm I}$ replaced by $S^{\rm I}_{\rm SW}$, namely,
$$\big({\mathcal M}+{\mathcal M}_{\rm a}[h,\boldsymbol{\Phi}_{{\mathcal I}(t)}]\big)\left(\begin{array}{c} {\dbug} \\ \dot{\bom} \end{array}\right)=\left(\begin{array}{c} -{\mathfrak m}g {\bf e}_z \\
{\mathfrak I}\bom\times \bom 
\end{array}\right)
+{\mathcal F}[h,\boldsymbol{\Phi}_{{\mathcal I}(t)}](S_{\rm SW,i}^{\rm I}+S_{\rm i}^{\rm III}\big).
$$
We therefore need to prove that ${\mathcal F}[h,\boldsymbol{\Phi}_{{\mathcal I}(t)}]S_{\rm SW,i}^{\rm I}={\mathcal F}_{\rm Arch}+{\mathcal F}_{\rm \Gamma}+{\mathcal F}_{\rm NL}$, with ${\mathcal F}_{\rm Arch}$, ${\mathcal F}_{\rm \Gamma}$ and ${\mathcal F}_{\rm NL}$ as in \eqref{defForce1}-\eqref{defForce2}. By definition, one has
\begin{align*}
{\mathcal F}[h,\boldsymbol{\Phi}_{{\mathcal I}(t)}]S_{\rm SW,i}^{\rm I}&=\int_{\cI}\uP_{\rm SW,i} \left(\begin{array}{c}N_{\rm w} \\ {\bf r}_G\times N_{\rm w}\end{array}\right)
\\
&=-\int_{\cI}\uP_{\rm SW,i} \nabla\cdot (h \nabla{\boldsymbol \Phi}_\cI),
\end{align*}
by definition of the elementary potentials. Integrating by parts and using the definition of $\uP_{\rm SW,i}$, one gets
\begin{align*}
{\mathcal F}[h,\boldsymbol{\Phi}_{{\mathcal I}(t)}]S_{\rm SW,i}^{\rm I}&=\int_\cI  \nabla\big[ gh\nabla\zeta+\nabla\cdot(\frac{1}{h}Q\otimes Q)\big]{\boldsymbol \Phi}_\cI,
\end{align*}
so that the result follows from a simple integration by parts.\\
For the energy conservation, one readily gets
$$
\frac{d}{dt}E_{\rm sw}=-\int_{{\mathcal I}(t)} \frac{\uP_{\rm SW,i}-P_{\rm atm}}{\rho}\dt \zeta_{\rm w}.
$$
with $\uP_{\rm SW,i}=\uP_{\rm SW,i}^{\rm I}+\uP_{\rm i}^{\rm II}+\uP_{\rm i}^{\rm III}$. The fact that $\frac{d}{dt}E_{\rm solid}=-\frac{d}{dt}E_{\rm SW}$ is then obtained as in the proof of Proposition \ref{propfloat}.
\end{proof}

\subsection{The Boussinesq approximation}\label{sectBouss}

In the nonlinear shallow water model used in the previous section, the (non hydrostatic) dispersive effects are neglected, which is not satisfactory for many applications. We consider here a Boussinesq model, which is the simplest nonlinear model that includes dispersive effects.
In absence of any immersed structure, the Boussinesq approximation consists in performing the following three approximations\footnote{The Boussinesq regime consists in assuming that $\mu\ll 1$ as for the nonlinear shallow water equations, but also requires a smallness assumption on the amplitude of the surface variations, namely, $\eps=O(\mu)$, where $\eps=\frac{a}{h_0}$ is the ratio of the typical amplitude of the surface variations over the depth at rest. A similar smallness assumption is also made on the bottom variations. The Boussinesq equations are obtained by dropping the terms of order $O(\mu^2)$ in the dimensionless water waves equations. If the smallness assumption on $\eps$, namely, $\eps=O(\mu)$, is removed, more terms should be kept for the non-hydrostatic acceleration. The corresponding regime is called Serre-Green-Naghdi (or fully nonlinear Boussinesq). We refer to \cite{LB} for more details, and to \cite{L_book} for a mathematical justification of these approximations. We treat here the Boussinesq regime for the sake of clarity, but the Serre-Green-Naghdi regime could be treated similarly, albeit with more complicated expressions.} on the averaged Euler equations \eqref{Eulerav1closed},
\begin{enumerate}
\item Neglect the vertical variations of the horizontal velocity in the quadratic. This leads, as in the shallow water approximation, to
$$
 \int_{-h_0+b}^\zeta V\otimes V \approx \frac{1}{h} Q\otimes Q.
$$
\item Neglect the variations of the surface elevation and of the bottom in the above approximations,
$$
 \frac{1}{h} Q\otimes Q\approx \frac{1}{h_0}Q\otimes Q
$$
\item Take into account the leading order term of the non-hydrostatic acceration
$$
h{\bf a}_{\rm NH}\approx -\frac{h_0^2}{3}\Delta \dt Q.
$$
\end{enumerate}
The resulting  \emph{Boussinesq equations with a floating structure} are given by
\begin{equation}\label{Boussf}
\begin{cases}
\dsp \dt h+\nabla\cdot Q=0,\\
\dsp \big[1-\frac{h_0^2}{3}\Delta\big]  \dt Q+\nabla\cdot (\frac{1}{h_0} Q\otimes Q)+gh\nabla \zeta=-h \frac{1}{\rho}\nabla \underline{P},
\end{cases}
\end{equation}
where $\uP_{\rm e}=P_{\rm atm}$ in ${\mathcal E}(t)$ and
\begin{equation}\label{eqPwBouss}
\begin{cases}
-\nabla\!\cdot\! (\frac{h}{\rho}\nabla \uP_{\rm i})=-\big[1-\frac{h_0^2}{3}\Delta\big] \dt^2 \zeta_{\rm w}+\nabla\!\cdot\!\big[\nabla\!\cdot\! (\frac{1}{h_0} Q\otimes Q)+gh\nabla\zeta\big] \quad \mbox{in } {\mathcal I}(t),\\
{\uP_{\rm i}}_{\vert_{\Gamma(t)}}=P_{\rm atm},
\end{cases}
\end{equation}
and with the boundary conditions at the contact line
\begin{equation}\label{contBouss}
Q_{\rm e}=Q_{\rm i} \quad \mbox{ and }\quad \zeta_{\rm e}=\zeta_{\rm i}\quad\mbox{ on }\Gamma(t)
\end{equation}
(and here again with the assumption \eqref{assCI} on the initial condition so that the constraint \eqref{constraint} is automatically satisfied). As shown in the following proposition, the source terms  created by the motion of the solid in the momentum equations must be modified, as well as the added mass-inertia matrix. The proof being a simple adaptation of the proof of Proposition \ref{propfloat}, it is omitted.
\begin{proposition}\label{propprescBouss}
Denoting by  $\bug$ the velocity  of the center of mass and by $\bom$ the angular velocity, the Boussinesq equations with a floating structure  take the form
$$
\begin{cases}
\dsp \dt \zeta+\nabla\cdot Q=0,\\
\dsp \big[1-\frac{h_0^2}{3}\Delta\big] \dt Q+\nabla\cdot (\frac{1}{h_0} Q\otimes Q)+gh \nabla\zeta=S_{\rm B}^{\rm I}+S^{\rm II}_{\rm B}+S^{\rm III}_{\rm B},
\end{cases}
$$
with the coupling conditions at the contact line
$$
 Q_{\rm e}=Q_{\rm i} 
\quad \mbox{ and }\quad
\zeta_{\rm e}=\zeta_{\rm i} 
\quad\mbox{ on }\quad \Gamma(t),
$$
and with the source terms $S^{\rm j}_{\rm B}$ ($j={\rm I,II,III}$)  given by
$$
S^{\rm j}_{\rm B,e}=0 \quad\mbox{ and }\quad S^{\rm j}_{\rm B,i}=-\frac{h}{\rho}\nabla \uP_{\rm B,i}^{\rm j} 
$$
where ${\uP_{\rm B, i}^{\rm I}}_{\vert_{\Gamma(t)}}=P_{\rm atm}$ and ${\uP_{\rm B,i}^{\rm II}}_{\vert_{\Gamma(t)}}={\uP_{\rm B,i}^{\rm III}}_{\vert_{\Gamma(t)}}=0$, and
$$
\begin{cases}
\dsp-\nabla\cdot (\frac{h}{\rho}\nabla \uP_{\rm B,i}^{\rm I})&=\nabla\cdot \big[ \nabla\cdot\big(\frac{1}{h_0}Q\otimes Q \big) +gh \nabla\zeta \big] \\
\dsp -\nabla\cdot (\frac{h}{\rho}\nabla \uP_{\rm B,i}^{\rm II})&=-(1-\frac{h_0^2}{3}\Delta)\big[\big(\dot{{\bf U}}_G+\dot{\bom}\times {\bf r}_G\big)\cdot N_{\rm w} \big]\\
\dsp -\nabla\cdot (\frac{h}{\rho}\nabla \uP_{\rm B,i}^{\rm III})&=(1-\frac{h_0^2}{3}\Delta){\mathcal Q}[{\bf r}_G](V_G,\bom),
 \end{cases}
 \quad\mbox{\textnormal{on}}\quad {\mathcal I}(t).
$$
In the case where the solid is freely floating, the evolution of $\bug$ and $\bom$ is given by the ODE
$$\big({\mathcal M}+{\mathcal M}_{\rm a}[h,\boldsymbol{\Phi}_{{\mathcal I}}]+{\mathcal M}_{\rm B}\big)\left(\begin{array}{c} \!\!{\dbug}  \!\!\\  \!\!\dot{\bom}  \!\!\end{array}\right)=\left(\begin{array}{c} -{\mathfrak m}g {\bf e}_z \\
{\mathfrak I}\bom\times \bom 
\end{array}\right)
+{\mathcal F}_{\rm Arch}+{\mathcal F}_{\rm \Gamma}+{\mathcal F}_{\rm B,NL}
+{\mathcal F}[h,\boldsymbol{\Phi}_{\mathcal I}]S_{\rm B,i}^{\rm III},
$$
where ${\mathcal F}_{\rm Arch}$ and ${\mathcal F}_{\rm \Gamma}$  are as in Proposition \ref{propprescSW} and
$$
{\mathcal F}_{\rm B,NL}= -\int_{\cI(t)}\nabla\cdot \big(\frac{1}{h_{0}}Q\otimes Q\big)\cdot \nabla {\boldsymbol \Phi}_\cI
\quad\mbox{ and }\quad {\mathcal M}_{\rm B}=-\rho\frac{h_0^2}{3}\int_\cI \boldsymbol{\Phi}_\cI\otimes \Delta\left(\begin{array}{c}\!\! N_{} \!\!\\ \!\!{\bf r}_G\times N_{}\!\!
\end{array}\right).
$$
\end{proposition}
\begin{remark}
Contrary to ${\mathcal M}_{\rm a}[h,\boldsymbol{\Phi}_{{\mathcal I}}]$, the dispersive correction ${\mathcal M}_{\rm B}$ is not necessarily positive nor symmetric.
\end{remark}
\begin{remark}
The energy formally conserved in the case of a freely floating object is $E_{\rm B,tot}:=E_{\rm B}+E_{\rm solid}$, with $E_{\rm solid}$ as in Proposition \ref{propfloat} and
$$
E_{\rm B}=\frac{1}{2}\int_{\R^d} \big(\frac{1}{h_0}\abs{Q}^2+\frac{h_0}{3}\abs{\dx Q}^2\big)+\frac{1}{2}\int_{\R^d}g \zeta^2.
$$
\end{remark}
\begin{remark}[Vertical walls]\label{remBoussvert}
Similarly to what has been done in Remark \ref{remSWvert} for the shallow water model, it is possible to allow for the possibility of vertical walls, provided one adds the same\footnote{In the Boussinesq regime, the term $\rho\int_{\zeta_{\rm i}}^{\zeta_{\rm e}}(\dt w+\bU\cdot \nabla_{X,z}w)$ can also be neglected in the definition of $P^{\rm IV}$ provided in \S \ref{sectvertical}. In order to check that this is the case, we recall that the horizontal velocity does not depend on $z$ at leading order in $\mu$, so that one gets from the incompressibility condition that $w\approx-(z+h_0)\nabla\cdot \ovV$, and therefore, neglecting the nonlinear terms that are smaller by a factor $\eps=O(\mu)$ in the Boussinesq regime,
\begin{align*}
\rho \int_{\zeta_{\rm i}}^{\zeta_{\rm e}}(\dt w+\bU\cdot \nabla_{X,z}w)&\approx -\rho \nabla\cdot \dt \ovV\int_{\zeta_{\rm i}}^{\zeta_{\rm e}} (z+h_0)\\
&=-\rho \nabla\cdot \dt \ovV\frac{1}{2}(h_{\rm e}^2-h_{\rm i}^2).
\end{align*}
Now, in the exterior region, one has at leading order in the Boussinesq regime $\dt \ovV=-g\nabla\zeta$, and therefore
$$
\rho \int_{\zeta_{\rm i}}^{\zeta_{\rm e}}(\dt w+\bU\cdot \nabla_{X,z}w)\approx 
\rho \frac{1}{2}g(h_{\rm e}^2-h_{\rm i}^2) \Delta \zeta.
$$
Under the assumptions that the surface variations are small (in the sense that $\eps=O(\mu)$), one has $(h_{\rm e}^2-h_{\rm i}^2)=O(\eps)$ in dimensionless variables and $\rho \int_{\zeta_{\rm i}}^{\zeta_{\rm e}}(\dt w+\bU\cdot \nabla_{X,z}w)$ is therefore of size $O(\eps\mu)$ and must therefore be neglected at the precision of the model. Note that this would not be the case for the Green-Naghdi model for which the assumption $\eps=O(\mu)$ is removed.
} extra source term $S^{\rm IV}_{\rm SW}$ as in Remark \ref{remSWvert} to the momentum equation. 
\end{remark}

\section{On the discretization of the wave-structure interaction}\label{sectdiscrete}
 
We have derived in Section \ref{sectWWfloat} the equations describing the evolution of water waves in the presence of a floating structure; in Section \ref{sectAsfloat}, the same approach has been used to show how one has to modify simpler asymptotic models (such as the nonlinear shallow water equations) when a floating structure is present. The key point was that in order for the interior pressure to be a Lagrange multiplier associated to the constraint $\zeta=\zeta_{\rm w}$ in $\cI$ in the simplified model, one had to modify consistently the elliptic equation defining $\uP_{\rm i}$.  The goal of this section is to push this strategy one step further, namely, at the discrete level. More precisely, starting from a numerical scheme approximating some hydrodynamic model without any floating structure, we show what the corresponding discretization of the additional terms describing the wave-structure interactions should be in order for the interior pressure to be a discrete Lagrange multiplier.

For the sake of simplicity, we shall only consider simple configurations here:
\begin{itemize}
\item The hydrodynamical models we shall consider are the one-dimensional shallow water equations and the one-dimensional Boussinesq equations
\item We assume that the solid can only move vertically
\item We assume that the structure has vertical sidewalls, so that the interior region is independent of time, ${\mathcal I}=(x_-,x_+)$.
\end{itemize}
More complex configurations (moving contact points, more degrees of freedom for the solid structure, two dimensional case, etc) will be considered in future works; our point here is to show that the discretization of the terms describing the fluid-structure interaction must be chosen carefully and depend strongly on the numerical scheme used for the fluid model. To be more precise, an important feature of our formulations is that if the initial conditions satisfy the compatibility condition \eqref{assCI} then the constraint $\zeta_{\rm i}=\zeta_{\rm w}$ is automatically satisfied. The discretization of the source terms due to the floating structure must be done in such a way that this property is preserved at the discrete level.

\noindent
{\bf NB.} {\it For the sake of simplicity, we consider a flat bottom ($b=0$) throughout this section}.

\subsection{The equations for the nonlinear shallow water model}\label{sect_modSW}

 When the solid can only move vertically, the horizontal coordinate of the center of mass remains constant; we take it equal to zero for simplicity, i.e. $x_G=0$. The position of the solid is therefore fully determined by the vertical coordinate $z_G(t)$ of its center of mass; it is a given function of time when the motion of the solid structure is prescribed, and must be found through Newton's law when it is freely floating in the vertical direction.

\subsubsection{The case of a prescribed vertical motion}

We recall that Proposition \ref{propprescSW} describes the shallow water equations in the presence of a floating structure; in the particular case of vertical motion considered here, the source term $S^{\rm III}$ vanishes; moreover, the source terms $S^{\rm I}_{\rm SW}$ and $S^{\rm II}$ can be simplified in horizontal dimension $d=1$ as in Proposition \ref{propWWCNU} (and taking into account that $u_G=\omega=0$ here) into
\begin{align*}
S_{\rm SW,e}^{\rm I}=0,&\qquad S_{\rm SW,i}^{\rm I}=\big[\dx\big(\frac{1}{h}q^2+\frac{1}{2}gh^2 \big)\big]^*\\
S_{\rm SW,e}^{\rm II}=0,&\qquad S_{\rm SW,i}^{\rm II}=-\ddot z_G x^*,
\end{align*}
where we refer to Notation \ref{notav} for the definition of the oscillating component $f^*$; as shown in \S \ref{sectvertical} and Remark \ref{remSWvert} (see also Remark  \ref{remd1vert} for the simplifications in the case $d=1$), an additional term $S^{\rm IV}_{\rm SW}$ must also be added due to the fact that the walls are vertical at the contact points,
$$
S_{\rm SW,e}^{\rm IV}=0,\qquad S_{\rm SW,i}^{\rm IV}=-g\jump{\zeta_{\rm e}-\zeta_{\rm i}}\frac{1}{\int_{x_-}^{x_+}1/h},
$$
where we used the notation $\jump{f}=f(x_+)-f(x_-)$. Therefore, in conservative form, the equations take the form
\begin{equation}\label{eqnum1}
\dt U+\dx \big({\mathcal F}(U))=(0,S)^T
\end{equation}
with 
$$
U=\left( \begin{array}{c}\zeta \\ q \end{array}\right),\qquad
{\mathcal F}(U)=\left( \begin{array}{c}{\mathcal F}_1(U) \\ {\mathcal F}_2(U) \end{array}\right)=\left( \begin{array}{c} q \\ \frac{1}{h}q^2+\frac{1}{2}g h^2 \end{array}\right),
$$
where the source term $S$ is given by
\begin{equation}\label{eqnum2}
S_{\rm e}=0,\qquad S_{\rm i}=\big[\dx\big(\frac{1}{h}q^2+\frac{1}{2}gh^2\big) \big]^*-\ddot z_G x^*-g\jump{\zeta_{\rm e}-\zeta_{\rm i}}\frac{1}{\int_{x_-}^{x_+}1/h},
\end{equation}
and with the continuity condition at $x=x_\pm$
\begin{equation}\label{eqnum3}
q_{\rm e}=q_{\rm i}.
\end{equation}
It is also assumed that the initial condition satisfies the condition \eqref{assCI},
$$
\qquad \zeta^0={\zeta_{\rm
    w}}_{\vert_{t=0}}\quad \mbox{ and }\quad \dx q_{\rm i}^0=-\dt {\zeta_{\rm
    w}}_{\vert_{t=0}} \quad\mbox{ on }\quad {\mathcal I}(0),
$$
which ensures that the constraint $\zeta_{\rm i}=\zeta_{\rm w}$ is automatically satisfied at all times.

\subsubsection{The case of an object freely floating in the vertical direction}\label{sectsolfloatvert}

If the solid is freely floating, the motion of its center of mass is given by Proposition \ref{propprescSW}; we can in the present case (where the motion is purely vertical) simplify the differential equation on the center of mass, as shown in the  proposition below. We recall first that the average and oscillating components of a function $f$ defined on the interior region $(x_+,x_-)$ have been introduced in Notation \ref{notav}. In the present configuration of a purely vertical motion, the added mass coefficient can be easily expressed in terms of the associated variance.
\begin{notation}\label{notavar}
If $f$ is a scalar function defined on $\cI=(x_-,x_+)$, we defined its variance by
$$
\mbox{\rm Var}(f)=\av{f^2}-\av{f}^2.
$$
\end{notation}
We shall also denote by $\zeta_{\rm w,eq}$ and $z_{G,\rm eq}$  the parametrization of the wetted surface and the position of the center of mass at equilibrium, and similarly $h_{\rm w, eq}=h_0+\zeta_{\rm w,eq}$. Away from equilibrium, the water depth under the solid is therefore fully determined by the distance of the center of mass to its equilibrium position,  $\delta_G=z_G-z_{G,\rm eq}$; consequently, one has $h_{\rm w}=h_{\rm w,eq}+\delta_G$, and $\mbox{\rm Var}(x)$ and $\av{x}$ are functions of $\delta_G$ only. \begin{proposition}\label{propODE}
	If the hydrodynamic model is the nonlinear shallow water model \eqref{eqnum1} and the object is freely floating, the distance $\delta_G=z_G-z_{G,\rm eq}$ of the center of mass to its equilibrium position  satisfies the ODE
	\begin{equation}\label{ODE1}
	\begin{cases}
	\dsp \big({\mathfrak m}+{m}_{\rm a}(\delta_G)\big)\ddot \delta_G=-{\mathfrak c}\delta_G+\rho g(\zeta_{\rm _e,+}x^*_{+}-\zeta_{\rm e,-}x^*_{-})+F_{\rm NL}(\delta_G,\dot{\delta}_G,\av{q}),\\
	 \dsp \frac{d}{dt}{\av{q}}=-g \frac{\zeta_{\rm e,+}-\zeta_{\rm e,-}}{\alpha(\delta_G)}+H_{\rm NL}(\delta_G,\dot{\delta}_G,\av{q}),
 \end{cases}
 \end{equation}
	where $\zeta_{\rm e,\pm}(t)=\zeta_{\rm e}(t,x_\pm)$ and the added mass ${m}_{\rm a}(\delta_G)$, the stiffness coefficient ${\mathfrak c}$ and $\alpha(\delta_G)$ 
	are given by
	\begin{align*}
	{m}_a(\delta_G)=\rho \alpha(\delta_G){\mbox{\rm Var}(x)}, \qquad {\alpha}(\delta_G)=\int_{x_-}^{x_+}\frac{1}{h_{\rm w}}
	\quad\mbox{ and }\quad {\mathfrak c}&=\rho g  (x_+-x_-);
	\end{align*}
	denoting $q_{\rm i}=\av{q}-x^*\dot\delta_G$, the nonlinear terms  $F_{\rm NL}$ and $H_{\rm NL}$ are given by
	\begin{align*}
	F_{\rm NL}(\delta_G,\dot{\delta}_G,\av{q})&=\rho \alpha(\delta_G)\Big\langle x^*\dx \big[\frac{q_{\rm i}^2}{h_{\rm w}}\big]\Big\rangle \\
H_{\rm NL}(\delta_G,\dot\delta_G,\av{q})&=\big(\av{\frac{x}{h_{\rm w}}}-\av{x}\av{\frac{1}{h_{\rm w}}}\big)(\dot\delta_G)^2-\Big\langle \dx \big[\frac{q_{\rm i}^2}{h_{\rm w}}\big]\Big\rangle.
	\end{align*}
\end{proposition}
\begin{remark}
 In this equation  it is necessary to know the boundary values $\zeta_{\rm e,\pm}$ of the surface elevation in the exterior domain. These quantities are of course determined through the resolution of the fluid equations \eqref{eqnum1}-\eqref{eqnum2}. 
 \end{remark}
\begin{remark}
 The excitation forces are due to the incoming waves, while damping forces are due to the motion of the structure; these two forces are contained in the term $\rho g(\zeta_{\rm _e,+}x^*_{+}-\zeta_{\rm e,-}x^*_{-})$ in the equation for $\delta_G$.  In the {\it return to equilibrium} problem considered in Corollary \ref{coroeq} below, there are no incoming waves and this force reduces to a purely damping force.
\end{remark}
\begin{proof}
In the case where the solid structure is freely floating in the vertical direction, we deduce from Proposition \ref{propprescSW} and Remark \ref{remSWvert} that the vertical coordinate $z_G(t)$ of the center of mass ((or equivalently its distance to equilibrium $\delta_G$) is found by solving the second order ODE
$$
\big({\mathfrak m}+{\mathfrak m}_{\rm a}(h_{\rm w})\big)\ddot \delta_G=-{\mathfrak m} g+\rho\int_{x_-}^{x_+}\frac{x^*}{h_{\rm w}}S^{\rm I}_{\rm i,SW}+
\int_{x_-}^{x_+}P^{\rm IV},
$$
with $S^{\rm I}_{\rm i,SW}$ as given in the previous section, and where the added mass ${\mathfrak m}_{\rm a}(h_{\rm w})$ and the pressure jump $P^{\rm IV}$ at the vertical walls are given by
\begin{align*}
{\mathfrak m}_a(h_{\rm w})&=\rho \int_{x_-}^{x_+}\frac{(x^*)^2}{h_{\rm w}},\\
P^{\rm IV}&= p_-+(p_+-p_-)\frac{\int_{x_-}^x 1/h_{\rm w}}{\int_{x_-}^{x_+}1/h_{\rm w}}\quad\mbox{ with }\quad p_\pm=\rho g (\zeta_{\rm e}-\zeta_{\rm i})_{\vert_{x=x_\pm}}.
\end{align*}
We can therefore rewrite the equation on $\ddot{\delta}_G$ under the form
$$
\big({\mathfrak m}+{\mathfrak m}_{\rm a}(h_{\rm w})\big)\ddot \delta_G=-\big({\mathfrak m} g+\rho g \int_{x_-}^{x^+}\zeta_{\rm w}\big)+\rho g(\zeta_{\rm _e,+}x^*_{+}-\zeta_{\rm e,-}x^*_{\rm -})+\rho \int_{x_-}^{x_+}\frac{x^*}{h_{\rm w}}\dx (\frac{q_{\rm i}^2}{h_{\rm w}}).
$$
It is obvious that ${\mathfrak m}_{\rm a}(h_{\rm w})=m_a(\delta_G)$ with $m_a(\delta_G)$ as  given in the statement of the proposition; moreover, we easily get from the mass conservation equation that 
$$
q_{\rm i}=-(x-\av{x})\dot{\delta}_G+\av{q}.
$$
 The only thing left to prove is therefore that one has 
$$
-{\mathfrak m} g-\rho g \int_{x_-}^{x_+}\zeta_{\rm w}=-\rho g  (x_+-x_-)\delta_G.
$$
  By definition of  $\zeta_{\rm w,eq}$, one has 
$$
-{\mathfrak m}g - \rho g \int_{x_-}^{x_+}\zeta_{\rm w,eq}=0.
$$
Since the lateral boundaries of the solid are vertical, we have moreover that, away from equilibrium, 
$$
\rho g \int_{x_-}^{x_+}\zeta_{\rm w}(t)-\rho g \int_{x_-}^{x_+}\zeta_{\rm w,eq}=\rho g  (x_+-x_-) \delta_G,
$$
which proves the result.\\
We finally turn to derive the equation on $\av{q}$. Taking into account the formula for $q_{\rm i}$ derived above, the second equation of \eqref{eqnum1} can be written in the interior region as
\begin{align*}
-x^*\ddot\delta_G+\frac{d}{dt}\av{x}\dot\delta_G+\frac{d}{dt}\av{q}&=-\av{\dx\big(\frac{1}{h_{\rm w}}q_{\rm i}^2+\frac{1}{2}g h_{\rm w}^2\big)}-\ddot\delta_G x^*-g\jump{\zeta_{\rm e}-\zeta_{\rm i}}\frac{1}{\int_{x_-}^{x_+}1/h_{\rm w}}\\
&=-\ddot\delta_G x^*-\frac{1}{\alpha(\delta_G)}\big[ \int_{x_-}^{x_+}\frac{1}{h_{\rm w}}\dx\big(\frac{q_{\rm i}^2}{h_{\rm w}}\big)+g(\zeta_{\rm e,+}-\zeta_{\rm e,-})\big]
\end{align*}
and therefore
\begin{align*}
\frac{d}{dt}\av{q}&=-\frac{d}{dt}\av{x}\dot\delta_G-\frac{1}{\alpha(\delta_G)}\big[ \int_{x_-}^{x_+}\frac{1}{h_{\rm w}}\dx\big(\frac{q_{\rm i}^2}{h_{\rm w}}\big)+g(\zeta_{\rm e,+}-\zeta_{\rm e,-})\big]\\
&=:-\frac{1}{\alpha(\delta_G)}\big[g(\zeta_{\rm e,+}-\zeta_{\rm e,-})\big]+H_{\rm NL}(\delta_G,\dot\delta_G,\av{q});
\end{align*}
the result follows therefore from the observation that 
$$
\frac{d}{dt} \av{x}=-\big(\av{\frac{x}{h_{\rm w}}}-\av{x}\av{\frac{1}{h_{\rm w}}}\big)\dot\delta_G.
$$
\end{proof}

A particularly interesting situation is the {\it return to equilibrium problem}, which consists in starting from a configuration where the solid is not at its equilibrium state ($z_G\neq z_{G,\rm eq}$, or equivalently $\delta_G\neq 0$), with water at rest ($\zeta_{\rm e}=0$, $q=0$), and let it evolve towards its equilibrium state. This is a particular case of the situation considered in Proposition \ref{propODE} in which the ODE takes a more explicit form. In order to get an even simpler formulation, we assume that the solid is symmetric.
\begin{corollary}[Return to equilibrium problem]\label{coroeq}
Assume that the solid is symmetric around the axis $x=x_0$ where $x_0=\frac{1}{2}(x_++x_-)$. Then, for the return to equilibrium problem, 
and as long as the following smallness condition on the velocity is satisfied,
	$$
	\dot\delta_G<\frac{16}{27}\sqrt{gh_0}\frac{h_0}{x_+-x_-},
	$$
the position of the solid is fully determined by the ODE
\begin{equation}\label{ODE1expl}
	\begin{cases}\big({\mathfrak m}+{m}_{\rm a}(\delta_G)\big)\ddot \delta_G=-{\mathfrak c}\delta_G-\nu(\dot{\delta}_G)+\beta(\delta_G)(\dot{\delta}_G)^2,\\
	(\delta_G,\dot\delta_G)_{\vert_{t=0}}=(\delta_G^0,0).
	\end{cases}
	\end{equation}
	The nonlinear damping $\nu(\dot\delta_G)$ and the coefficient $\beta(\delta_G)$ are given by
	\begin{align*}
	\nu(\dot\delta_G)&=\rho g (x_+-x_-)\big[h_0-\big(\tau_0(\frac{x_+-x_-}{4\sqrt{g}}\dot{\delta}_G)\big)^2\big]\\
	\beta(\delta_G)&=\rho \int_{x_-}^{x_+}\frac{x-x_0}{h_{\rm w}}\dx \Big(\frac{(x-x_0)^2}{h_{\rm w}}\Big),
	\end{align*}
	where the function $\tau_0(\cdot)$ is as in \eqref{deftau0} below.
\end{corollary}
\begin{remark}
In the return to equilibrium problem, there is no incoming wave; the force $\rho g(\zeta_{\rm _e,+}x^*_{+}-\zeta_{\rm e,-}x^*_{-})$ reduces therefore to its damping component $-\nu(\dot\delta_G)$. The fact that it is indeed a (nonlinear) damping force comes from the observation that $\nu(\dot\delta_G)\dot\delta_G$ is always positive.
\end{remark}
\begin{remark}
Linearizing around the equilibrium state, the ordinary differential equation \eqref{ODE1expl} becomes
\begin{equation}\label{ODElin}
\big({\mathfrak m}+{m}_{\rm a}(0)\big)\ddot \delta_G=-{\mathfrak c}\delta_G-\rho g \frac{(x_+-x_-)^2}{2}\frac{\dot\delta_G}{\sqrt{gh_0}},
\end{equation}
which is a standard damped harmonic oscillator equation. This linear equation matches the equation derived in \cite{John1} under further assumptions on the shape of the solid\footnote{Actually, the mass ${\mathfrak m}$ is assumed to be negligible  with respect to $m_a(0)$ in \cite{John1}, and with our notations, equation (3.2.12) of \cite{John1} corresponds to 
$$
{m}_{\rm a}(0)\ddot \delta_G=-{\mathfrak c}\delta_G--\rho g \frac{(x_+-x_-)^2}{2}\frac{\dot\delta_G}{\sqrt{gh_0}}.
$$}. The nonlinear ODE \eqref{ODE1expl} can be numerically solved with standard tools. A comparison with the solution of the linear ODE \eqref{ODElin} is shown in Figure \ref{fig_ODE} in the case where the floating body is the same as the one described in \S \ref{sectnumSW} below, and with the solid density given by $\rho_{\rm s}=0.8 \rho$. These computations show that nonlinear effects play a significant role for large amplitudes and should therefore not be neglected for the description of floating structures in the presence of large amplitude waves for instance.
\begin{center}
\begin{figure}\includegraphics[width=0.75\textwidth]{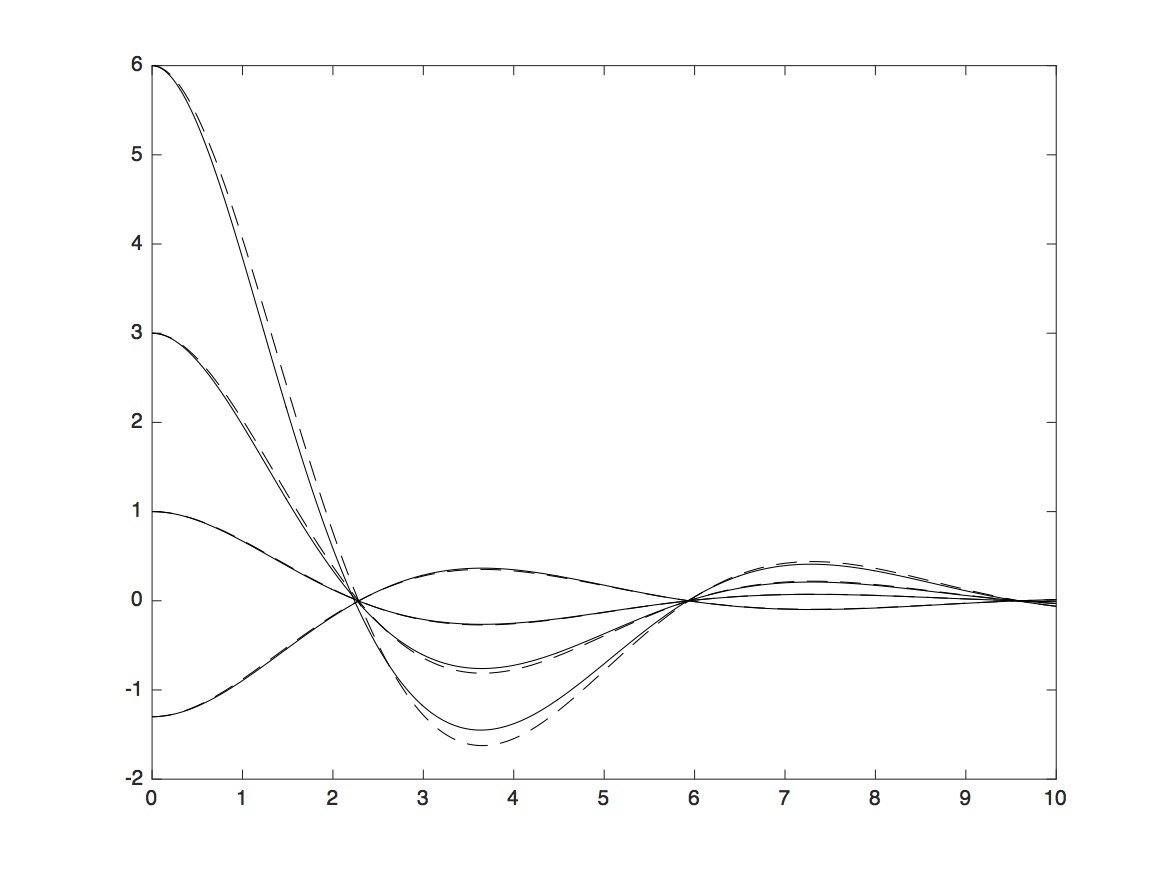}
\caption{Time evolution of the distance to equilibrium $\delta_G$ as given by the full nonlinear equation \eqref{ODE1expl} (full) and its linear approximation \eqref{ODElin} (dash). Four different initial positions are considered.}
\label{fig_ODE}
\end{figure}
\end{center}
\end{remark}
\begin{remark}
The corollary furnishes through \eqref{ODE1expl} an explicit solution for the return to equilibrium problem. We shall use this explicit solution to validate the numerical scheme derived in \S \ref{sectnumSW} below general wave-structure interactions.
\end{remark}
\begin{remark}\label{remvalforced}
A byproduct of the proof of the corollary is that the water elevation at the contact points $x_\pm$ is related to the velocity of the center of mass through the relation 
$$
\zeta_{\rm e}(t,x_\pm)=\big(\tau_0(\frac{x_+-x_-}{4\sqrt{g}}\dot{\delta}_G)\big)^2-h_0.
$$
This relation remains true when the solid is in forced oscillation in a fluid initially at rest (as for the return to equilibrium problem, there are no incoming waves); we shall therefore use it as a validation for the numerical computations of \S \ref{sectforcednum}.
\end{remark}
\begin{proof}
We just have to use Proposition \ref{propODE} and express $\zeta_{\rm e,\pm}$ and $\av{q}$ in terms of $\delta_G$ and $\dot{\delta}_G$.\\
By symmetry reasons, one has $\av{x}=x_0$ and $q_+=-q_-$ and therefore $\av{q}=0$.  Replacing in the formula for $F_{\rm NL}$ given in Proposition \ref{propODE}, we get
$$
F_{\rm NL}(\delta_G,\dot{\delta}_G,\av{q})=\rho (\dot{\delta}_G)^2\int_{x_-}^{x_+}\frac{(x-x_0)}{h_{\rm }}\dx \big(\frac{1}{h_{\rm w}}(x-x_0)^2\big),
$$
and the expression for the coefficient $\beta(\delta_G)$ follows easily.\\
In order to express $\zeta_{\rm e,\pm}$ in terms of $\delta_G$, let us recall first that in the exterior region, one has
$$
\dt R+(\sqrt{gh}+\frac{q}{h})R=0 \quad \mbox{ and }\quad \dt L-(\sqrt{gh}-\frac{q}{h})L=0,
$$
where $R$ and $L$ are respectively the right and left Riemann invariant associated to the nonlinear shallow water equations \eqref{eqnum1} and given by
$$
R=\frac{q}{h} + 2(\sqrt{gh}-\sqrt{gh_0})
 \quad\mbox{ and }\quad L=\frac{q}{h} - 2(\sqrt{gh}-\sqrt{gh_0}).
$$
Since the fluid is initially at rest, $R$ vanishes identically on $(-\infty,x_-)$ and $L$ vanishes identically on $(x_+,\infty)$. In particular, evaluating at $x=x_-$, one finds that  $\sqrt{h(t,x_-)}$ is a root of the third order polynomial equation
\begin{equation}\label{roots}
\tau^3-\sqrt{h_0}\tau^2+\frac{q_-}{2\sqrt{g}}=0.
\end{equation}
For each value of $r:=\frac{q_-}{2\sqrt{g}}$, there exists one or three real roots of this third order equation (see Figure \ref{figroot}):
\begin{itemize}
\item One positive real root $\tau_0(r)$ if $r<0$
\item Two positive roots $\tau_0(r)$ and $\tau_1(r)$, and one negative real root $\tau_2(r)$ if $0<r<\frac{4}{27}h_0^{3/2}=:r_0$.
\item One negative eigenvalue $\tau_2(r)$ if $\frac{q_-}{2\sqrt{g}}>r_0$.
\end{itemize}
Since the solid is dropped with zero initial velocity and with the fluid initially at rest, the relevant root is the one that passes through the point $(0,\sqrt{h_0})$ and it is given by
\begin{equation}\label{deftau0}
\tau_0(r)=\frac{1}{3}\Big(\sqrt{h_0}+C(r)+\frac{h_0}{C(r)}\Big)
\end{equation}
with the complex constant $C(r)$ equal to
$$
C(r)=\frac{3}{2}\Big(-4r+2r_0+4\sqrt{r(r-r_0)}\Big)^{1/3}
$$
(the smallness assumption made in the Corollary ensures that $r<r_0$).
It follows therefore that
\begin{align*}
\zeta_{\rm e}(t,x_-)&=\big(\tau_0(\frac{q_-}{2\sqrt{g}})\big)^2-h_0\\
&=\big(\tau_0(\frac{x_+-x_-}{4\sqrt{g}}\dot{\delta}_G)\big)^2-h_0,
\end{align*}
where we used the as above that $q_{\rm i}=-(x-x_0)\dot\delta_G$. Since moreover $\zeta_{\rm e}(t,x_+)=\zeta_{\rm e}(t,x_-)$ for symmetry reasons, the expression for $\rho g(\zeta_{\rm _e,+}x^*_{+}-\zeta_{\rm e,-}x^*_{-})$ simplifies into $-\nu(\dot\delta_G)$ as claimed in the statement of the corollary.
\begin{center}
\begin{figure}\includegraphics[width=0.55\textwidth]{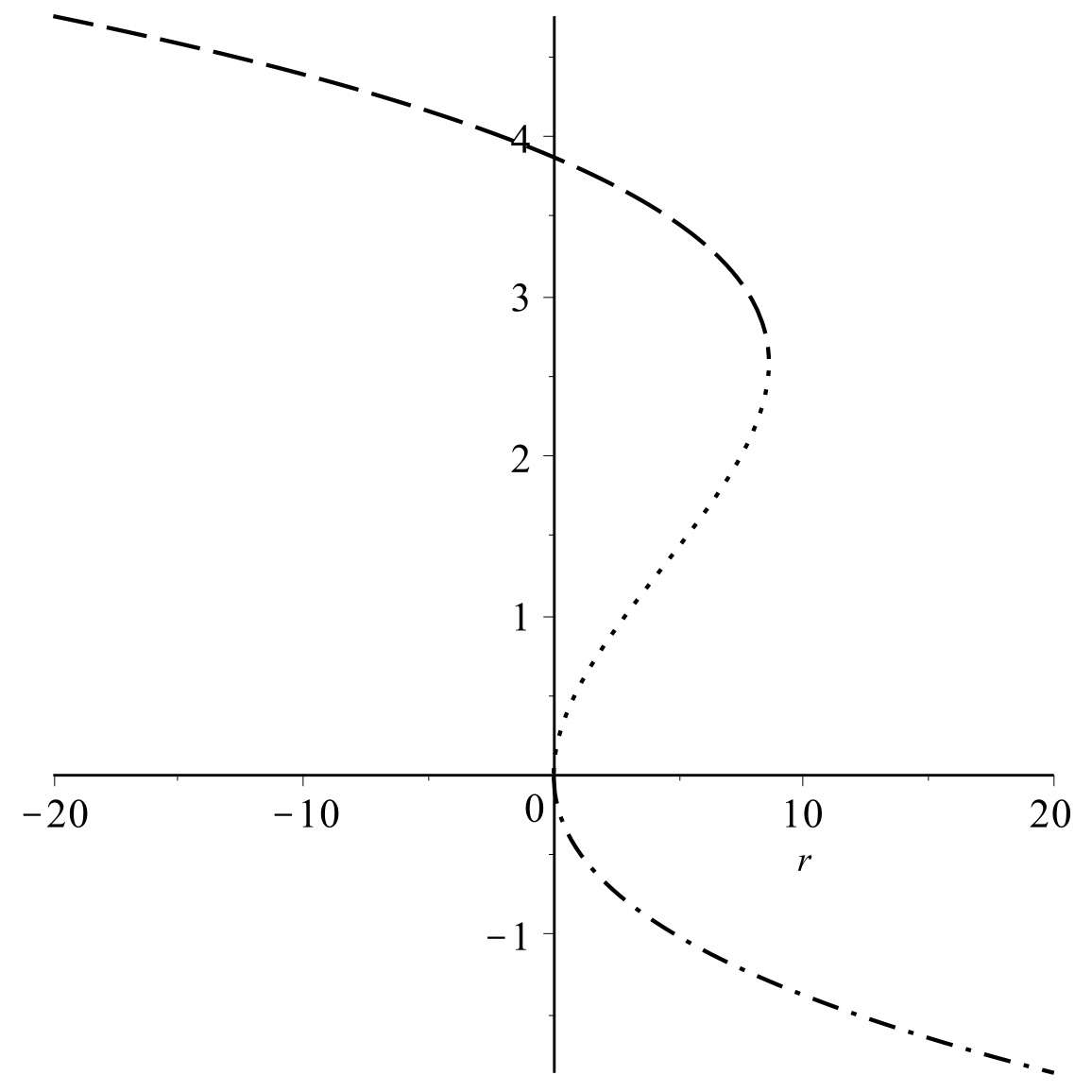}
\caption{Roots of the third order polynomial \eqref{roots}. $\tau_0(r)$ (dash), $\tau_1(r)$ (dots) and $\tau_2(r)$ (dash-dots); here, $h_0=15$ and $r_0=8.61 $.}
\label{figroot}
\end{figure}
\end{center}
\end{proof}
\subsection{The numerical scheme for the nonlinear shallow water equations}\label{sectnumSW}

We first present in \S \ref{sectschemefluid} the numerical scheme we shall use for the flow model (here, the nonlinear shallow water equations \eqref{eqnum1}); the associated discretization of the terms describing the wave-structure interaction is then presented in \S \ref{sectadapt} when the motion of the solid is assumed to be prescribed. The coupling with the motion of the solid itself in the case where it is freely floating (in the vertical direction) is then described in \S \ref{sectfreenum}.

\subsubsection{The numerical scheme for the hydrodynamical model}\label{sectschemefluid}

For the sake of simplicity, we shall show how to discretize the source terms describing the wave-structure interaction in the case where the numerical scheme for the flow method is the simple Lax-Friedrichs scheme. More precise, higher order, schemes would complicate the analysis, and shall be considered in future works. Let us introduce first some
notations.
\begin{notation}
- We denote by $[0,L]$ the computational domain; for some $N\in {\mathbb N}^*$, we let $\delta_x=L/N$ and define
$(x_j)_{0\leq j\leq N}$ and $(x_{j+1/2})_{0\leq j\leq
  N-1}$ by $x_j=j\delta_x$ and $x_{j+1/2}=(j+1/2)\delta_x$. We define $N+1$ finite volumes by $K_j=[x_{j-1/2},x_{j+1/2}]$ for $1\leq j\leq N-1$ and $K_0=[0,x_{1/2}]$, $K_N=[x_{N-1/2},L]$.\\
- For the time discretization, we denote by $\delta_t$ the time step and
$t^n=n\delta_t$, and we write $U^n_j=(h^n_j,q^n_j)$ the approximation
of $U(t^n,\cdot)$ on $K_j$; for second order time derivatives, we write
$$
\ddot{U}^n_j:=\frac{1}{\delta_t^2}\big( U^{n+1}_j-2U^n_j+U^{n-1}_j).
$$
- We also denote by ${\mathcal F}^n_{j+1/2}$ an
approximation at time $t^n$ of the flux at the interface
$x_{j+1/2}$. \\
- The ratio of
the time and space steps is denoted $\alpha=\delta_t/\delta_x$.\\
- Finally, we denote by $j_-$ the index  of the nearest cell
outside the interior region ${\mathcal I}=(x_-,x_+)$ on the left-hand-side of the solid
($j_-$ is the largest integer smaller than $x_-/\delta_x$),
and similarly $j_+$ for the first cell outside the wetted region on
the right-hand-side.\\
\end{notation}

A general finite volume discretization of \eqref{eqnum1} can be
written under the form
\begin{equation}\label{FV}
U^{n+1}_j=U^n_j-\alpha\big({\mathcal F}^n_{j+1/2}-{\mathcal
  F}^n_{j-1/2}\big)+\delta_t S^n_j.
\end{equation}
Our aim is to choose a discretization of the flux and of the source
term that ensures that at machine precision, one has
$\zeta^n_j=\zeta_{{\rm w},j}^n$ for all $j_-<j<j_+$. Let us base our analysis on the most simple stable
scheme for \eqref{eqnum1} when there is no immersed solid,
namely, the Lax-Friedrichs scheme for which the discretization of the
flux is
\begin{equation}\label{LF}
{\mathcal F}_{j+1/2}^n=\frac{1}{2}\big({\mathcal
  F}(U^n_{j+1})+{\mathcal F}(U^n_j)\big)-\frac{1}{2\alpha}\big(U^n_{j+1}-U^n_j\big).
\end{equation}
We show in the next section how to adapt this scheme in the presence of a floating structure.

\subsubsection{Adaptation in the presence of a  structure with a prescribed vertical motion}\label{sectadapt}

In order to take into account the presence of a floating structure, one has to adapt the Lax-Friedrichs scheme \eqref{LF} in the interior region, and to provide a discretization of the source term $S$ in \eqref{eqnum1} ensuring that  one has
$\zeta^n_j=\zeta_{{\rm w},j}^n$ for all $j_-<j<j_+$. \\
The second term in the right-hand-side of \eqref{LF} is a diffusive term that
ensures stability. However, the expression for the wetted pression for
the continuous model relied on the relation $\dt \dx {\mathcal
  F}_1=\dx\dt q$. At the discrete level, the presence of the diffusive
term in the equation for the surface elevation does not seem to be
compatible with a discrete version of this fundamental relation. We
are therefore led to the following adaptation of the Lax-Friedrichs
flux,
\begin{equation}\label{LFadapt1}
{\mathcal F}^n_{1,j+1/2}=\begin{cases}
\frac{1}{2}\big(q^n_{j+1}+q^n_j\big)-\frac{1}{2\alpha}\big(h^n_{j+1}-h^n_j\big)&\mbox{if }
j < j_- \mbox{ or }j\geq j_+,\\
\frac{1}{2}\big(q^n_{j+1}+q^n_j\big)&\mbox{if } j_-\leq j<j_+
\end{cases}
\end{equation}
for the equation on the surface elevation, while the flux for the
momentum equation is the same as for  the standard Lax-Friedrichs
scheme
\begin{equation}\label{LFadapt2}
{\mathcal F}_{2,j+1/2}^n=\frac{1}{2}\big({\mathcal
  F}_2(U^n_{j+1})+{\mathcal F}_2(U^n_j)\big)-\frac{1}{2\alpha}\big(q^n_{j+1}-q^n_j\big).
\end{equation}
It is well known that the Lax-Friedrichs scheme is unstable if the
diffusive term is removed. The reason why stability is (numerically) preserved in
our case is because we are able to choose a discretization of the
source term that ensures that on the wetted region (where the
diffusive terms are removed), the surface elevation $\zeta$ is exactly
equal to the parametrization $\zeta_{\rm w}$ of the solid structure,
so that no instability occurs. This is done in the following
proposition where we use the following notations which corresponds to a transposition of Notation \ref{notav} at the discrete level.
\begin{notation}\label{notavdisc}
- We use the notation   $\sum_{\sharp k}^l$, $\sum_k^{\sharp l}$ and $\sum_{\sharp k}^{\sharp l}$ for the following modified summations
$$
\sum_{j=\sharp k}^l a_j=\frac{1}{2}a_{k}+\sum_{j=k+1}^{l}a_j, \qquad \sum_{j=k}^{\sharp l} a_j=\frac{1}{2}a_{l}+\sum_{k}^{l-1}a_j,\quad\mbox{etc.}
$$
- For any vector $(f)_{j_-\leq j \leq j_+}$ we define $\av{f}$ and $f^*=(f^*_j)_{j_-\leq j \leq j_+}$ by
$$
\av{f}^n=\frac{\sum_{\sharp j_-}^{\sharp j_+} f^n_j/h^n_j }{\sum_{\sharp j_-}^{\sharp j_+}1/h^n_j}
\quad \mbox{\and }\quad
f^{*,n}_j=f^n_j-\av{f}^n.
$$
- We also define $D_0{\mathcal F}_2$ as
$$
(D_0{\mathcal F}^n_2)_j=\frac{{\mathcal F}^n_{2,j+1/2}-{\mathcal
  F}^n_{2,j-1/2}}{\delta_x}.
$$
\end{notation}
We also recall that the {\it interior cells} corresponding to the discretization of the interior domain $\cI$ correspond to the indexes $j_-<j<j_+$.
\begin{proposition}\label{propdiscSW}
Let us consider a floating body in purely vertical motion and denote by $z_G$ the vertical coordinate of its center of mass at time $t^n$. Let the discretization of the nonlinear shallow water equations  \eqref{eqnum1}-\eqref{eqnum3} be furnished by \eqref{FV} with fluxes \eqref{LFadapt1} and \eqref{LFadapt2}. Defining $S_j^n$ ($j_-\leq j\leq j_+$) as
\begin{equation}\label{Snj}
S^n_j=(D_0{\mathcal F}^n_2)^*_j-\ddot{z}_G^{n+1} (x-x_G)^{*,n}_j-g\frac{(h_{j_+}^n-h_{j_+-1}^n) -  (h_{j_-}^n-h_{j_-+1}^n) }{\sum_{\sharp j_-}^{\sharp j_+}1/h^n_j},
\end{equation}
and provided that the initial condition $(h^0,q^0)$ satisfies, for all $ j_-<j<j_+$, 
\begin{equation}\label{discreteCI}
 h^0_j=h^0_{{\rm w},j}\quad\mbox{ and }\quad h^1_{{\rm w},j}=h_{{\rm w},j}^0-\alpha({\mathcal F}^0_{1,j+1/2}-{\mathcal F}^0_{1,j-1/2}),
\end{equation}
then, for all $n\in \N$ and $j_-<j<j_+$, one has $h^n_j=h^n_{{\rm w}, j}$.
\end{proposition}
\begin{proof}
Instead of seeking a discretization of $S_{\rm i}=-\frac{1}{\rho}\dx \uP_{\rm i}$ based on its continuous expression, we shall rather mimic the approach used in the continuous case and look for a discretization of the interior pressure $\uP_{\rm i}$ as a discrete Lagrange multiplier associated to the constraint $\zeta_{\rm i}=\zeta_{\rm w}$.\\
According to \eqref{FV} and \eqref{LFadapt1}, one has for all $j_- <j<j_+$,
\begin{align*}
(h_j^{n+2}-h_j^{n+1})-(h_j^{n+1}-h^{n}_j)=&-\frac{\alpha}{2}\big(q_{j+1}^{n+1}-q_{j-1}^{n+1}\big)
+\frac{\alpha}{2}\big(q_{j+1}^n-q_{j-1}^n\big)\\
=&-\frac{\alpha}{2}\big(q_{j+1}^{n+1}-q_{j+1}^n\big)+\frac{\alpha}{2}\big(q_{j-1}^{n+1}-q_{j-1}^n\big).
\end{align*}
Using the discretization of the momentum equation and the fact that by definition
$h_j^n=h_{{\rm w},j}^n$ for
$j_-<j<j_+$, we therefore want the source terms $S_j^n$ to be such that
$$
\delta_t^2\ddot{\zeta}_{{\rm w},j}^{n+1}=-\frac{\alpha}{2}\delta_t \big(S_{j+1}^n-S_{j-1}^n\big)+\frac{\alpha^2}{2}\big[\big({\mathcal F}^n_{2,j+3/2}-{\mathcal
  F}^n_{2,j+1/2}\big)- \big({\mathcal F}^n_{2,j-1/2}-{\mathcal
  F}^n_{2,j-3/2}\big)\big]
$$
or equivalently
\begin{equation}\label{eqDS}
 \frac{S_{j+1}^n-S_{j-1}^n}{2\delta_x}=-\ddot{\zeta}_{{\rm w},j}^{n+1}+\frac{1}{2\delta_x}\big[(D_0{\mathcal F}^n)_{j+1}- (D_0{\mathcal F}^n)_{j-1}\big].
\end{equation}
On the other hand, $S_j^n$ should be an approximation on the cell $j$ and at time $t^n$ of $S=-h\frac{\dx \uP_{\rm i}}{\rho}$ and we thus look for it under the form
\begin{equation}\label{eqSP1}
\forall j_-<  j < j_+,\qquad S_j^n=-\frac{h_j^n}{\rho}\frac{P^n_{j+1/2}-P^n_{j-1/2}}{\delta_x},
\end{equation}
and, at the boundary points $j=j_\pm^n$, 
\begin{equation}\label{eqSP2}
S_{j_-}^n=-\frac{h_{j_-}^n}{\rho}\frac{P^n_{j_-+1/2}-\uP_{\rm
    i}^-}{\delta_x/2}
\quad\mbox{ and }\quad
S_{j_+}^n=-\frac{h_{j_+}^n}{\rho}\frac{\uP_{\rm i}^+- P^n_{j_+-1/2}}{\delta_x/2},
\end{equation}
where $\uP_{\rm i}^\pm$ denote the interior pressure at $x=x_\pm$.\\
We shall now use the following lemma, in which the superscript $n$ for the time dependance is omitted for the sake of clarity.
\begin{lemma}
A solution $(P_{j+1/2})_{j_-\leq j <j_+}$ to the equation
$$
\forall j_-<j<j_+,\qquad  \frac{S_{j+1}-S_{j-1}}{2\delta_x}=\frac{g_{j+1}-g_{j-1}}{2\delta_x}
$$
with $(S_j)_{j_-\leq j\leq j_+}$ as in \eqref{eqSP1}-\eqref{eqSP2} is given by
$$
\forall j_-\leq j<j_+,\qquad -\frac{P_{j+1/2}}{\rho}=-\frac{\uP_{\rm i}^-}{\rho}+\delta_x\big[\sum_{k=\sharp j_-}^j \frac{g_k}{h_k} +c\sum_{k=\sharp j_-}^j \frac{1}{h_k} \big].
$$
where the constant $c$ is given by
$$
c=\big[-\frac{\jump{\uP_{\rm i}^\pm}}{\rho}-\delta_x\sum_{j=\sharp j_-}^{\sharp j_+}\frac{g_j}{h_j}\big]\times \frac{1}{\delta_x\sum_{\sharp j_-}^{\sharp j_+}\frac{1}{h_j}}.
$$
In particular, one has
$$
S_j=g^*_j-\frac{\jump{\uP_{\rm i}^\pm}}{\rho}\frac{1}{\delta_x\sum_{\sharp j_-}^{\sharp j_+}\frac{1}{h_j}}.
$$
\end{lemma}
\begin{proof}[Proof of the lemma]
It suffices to prove that there exists a constant $c$ such \eqref{eqSP1}-\eqref{eqSP2} are satisfied with $S_j=g_j+c$. Rewriting \eqref{eqSP1} under the form
$$
-\frac{P_{j+1/2}-P_{j-1/2}}{\rho}=\delta_x \frac{S_j}{h_j}=\delta_x \frac{g_j+c}{h_j}
$$
and using the first equation of \eqref{eqSP2} one easily obtains that
$$
\forall j_-\leq j<j_+,\qquad -\frac{P_{j+1/2}}{\rho}=-\frac{\uP_{\rm i}^-}{\rho}+\delta_x\big[\sum_{j=\sharp j_-}^j \frac{g_j}{h_j} +c\sum_{j=\sharp j_-}^j \frac{1}{h_j} \big].
$$
The value of this expression at $j=j_+-1$ should match the one provided by the second equation of \eqref{eqSP2}; this is possible only with a particular choice of $c$, which is the one given in the statement of the lemma.
\end{proof}
In the configuration considered here, the solid is only allowed to move in vertical translation, so that $\ddot\zeta^{n+1}_{\rm w,j}=\ddot{z}_G^{n+1}$, where $z_G$ is the vertical coordinate of the center of mass (or any other point of the solid). In particular, this quantity does not depend on $j$, and one can write
$$
\ddot\zeta^{n+1}_{\rm w,j}=\ddot{z}_G^{n+1} \frac{1}{2\delta x}\big( x_{j+1}- x_{j-1}\big).
$$
The expression for $S_j$ given in the statement of the proposition follows therefore from \eqref{eqDS},  the above lemma, and the fact that $\uP_{\rm i}^\pm=\rho g(\zeta_{\rm e}^\pm-\zeta_{\rm i}^\pm)$, which at the discrete level, reads
$$
\uP_{\rm i}^{-,n}=\rho g\big(h_{j_-}^n-h^n_{j_-+1}\big)\quad\mbox{ and }\quad \uP_{\rm i}^{+,n}=\rho g\big(h_{j_+}^n-h^n_{j_+-1}\big).
$$
Conversely, if the source term is given by the expression stated in the proposition, the above computations show that
$$
\forall j_-<j<j_+, \qquad \ddot{h}_j^{n+1}=\ddot{h}_{\rm w,j}^{n+1}.
$$
If the assumption made in the statement of the proposition on the initial data $(h^0,q^0)$ holds, then one has ${h}_j^{0}={h}_{\rm w,j}^{0}$ and ${h}_j^{1}={h}_{\rm w,j}^{1}$ and a straightforward induction shows that ${h}_j^{n}={h}_{\rm w,j}^{n}$ for all $n\in \N$.
\end{proof}

\subsubsection{The case of a freely floating object}\label{sectfreenum}

The difficulty in the discretization of the wave-structure interaction is that according to Proposition \ref{propdiscSW} the discretization of the source term $S^n$ in \eqref{FV} requires the knowledge of $\ddot{z_G}^{n+1}$ and therefore the position of $z_G$ (or equivalently the distance $\delta_G$ to its equilibrium position) at time $t^{n+2}$.
The time discretization for the ODE \eqref{ODE1} is the following
\begin{equation}\label{ODE1disc}
\big({\mathfrak m}+{m}_{\rm a}(\delta_G^{n})\big)\ddot \delta_G^{n+1}=-{\mathfrak c}\delta_G^n+\rho g \big(\zeta_{\rm e,+}^n x_+^{*,n}-\zeta^n_{\rm e,-}x_-^{*,n}\big)+F_{\rm NL}(\delta_G^n,q_{\rm i}^n),\end{equation}
where  $F_{\rm NL}$ and $m_{\rm a}$ are as defined\footnote{The nonlinear term $F_{\rm NL}(\delta_G,\dot\delta_G,\av{q})$ that appears in Proposition \ref{propODE} depends on $\dot{\delta}_G$ and $\av{q}$ through $q_{\rm i}=\av{q}-x^*\dot\delta_G$ only. Since we choose to compute $q_{\rm i}$ through the second equation of \eqref{eqnum1} (which is of course equivalent to the equation on $\av{q}$ in Proposition \ref{propODE}), we rather write $F_{\rm NL}(\delta_G,q_{\rm i})$ here.}
 in Proposition \ref{propODE}.\\
 The iterative scheme we use to compute $U^{n+1}$ and $\delta_G^{n+2}$ in terms of $U^n$ and $(\delta_G^{n+1},\delta_G^n)$ is given in Algorithm \ref{algolin}.
\begin{algorithm}\label{algolin}
\caption{The wave-structure coupling algorithm}
\begin{algorithmic}
\REQUIRE The quantities $U^n=(h^n,q^n)$ and $(\delta_G^{n+1},\delta_G^n)$ are known
\STATE Compute $\rho g \big(\zeta_{\rm e,+}^n x_+^{*,n}-\zeta^n_{\rm e,-}x_-^{*,n}\big)$, $F_{\rm NL}(\delta_G^n,q_{\rm i}^n)$ and ${m}_{\rm a}(\delta_G^n)$
\STATE Compute $\ddot{\delta}_G^{n+1}$ through \eqref{ODE1disc}
\STATE Compute the source term $S^n$ with the formula of Proposition \ref{propdiscSW}
\STATE Compute $U^{n+1}$ with \eqref{FV}-\eqref{LFadapt2}.
\end{algorithmic}
\end{algorithm}
\begin{remark}
In \eqref{ODE1disc}, the contribution of $\rho g \big(\zeta_{\rm e,+} x_+^{*}-\zeta_{\rm e,-}x_-^{*}\big)$ and $F_{\rm NL}$ has been treated explicitly. The general fact that $F^{\rm II}_{\rm fluid}$ can be put under the form of an added mass term (see \eqref{pkey}) allows us to treat it here in implicit form, which is of crucial importance for the stability of the scheme. The numerical importance of the added mass effects  has been evidenced in other fluid-structure interactions \cite{FGG}, and discussed in particular in \cite{CausinGerbeauNobile,FWR}.
\end{remark}

\subsection{Using the Boussinesq system as hydrodynamic model}\label{sect_modBouss}

We are considering here the same situation as in \S \ref{sect_modSW} with the only difference that the nonlinear shallow water equations are replaced by the one dimensional Boussinesq equations. 

\subsubsection{The continuous equations}
We recall that Proposition \ref{propprescBouss} describes the Boussinesq equations in the presence of a floating structure; in the particular case of vertical motion considered here, the source term $S^{\rm III}_{\rm B}$ vanishes; moreover, the source term  $S^{\rm II}_{\rm B}$ is the sameas for the shallow water model, $S^{\rm II}_{\rm B}=S^{\rm II}_{\rm SW}$, since the second order derivative in the source term defining $P^{\rm II}_{\rm B,i}$ vanishes in the case of a purely vertical motion. Therefore, in conservative form, the equations take the form
\begin{equation}\label{eqnum1B}
{\mathcal D}\dt U+\dx \big({\mathcal F}_{\rm B}(U))=(0,S_{\rm B})^T
\end{equation}
where the term $q^2/h$ is replaced by $q^2/h_0$ in the flux, 
$$
U=\left( \begin{array}{c}\zeta \\ q \end{array}\right),\qquad
{\mathcal F}_{\rm B}(U)=\left( \begin{array}{c}{\mathcal F}_1(U) \\ {\mathcal F}_{\rm B,2}(U) \end{array}\right)=\left( \begin{array}{c} q \\ \frac{1}{h_0}q^2+\frac{1}{2}g h^2 \end{array}\right),
$$
and where the source term $S_{\rm B}$ is correspondingly given by
\begin{equation}\label{eqnum2B}
S_{\rm B,e}=0,\qquad S_{\rm B,i}=\big[\dx\big(\frac{1}{h_0}q^2+\frac{1}{2}gh^2\big) \big]^*-\ddot z_G x^*-g\jump{\zeta_{\rm e}-\zeta_{\rm i}}\frac{1}{\int_{x_-}^{x_+}1/h};
\end{equation}
finally, the additional dispersive term is
$$
{\mathcal D}=(1,1-\frac{h_0^2}{3}\dx^2)^T.
$$
Of course the same continuity condition \eqref{eqnum2} and compatibility condition on the initial data are made, so that the constraint $\zeta_{\rm i}=\zeta_{\rm w}$ is automatically satisfied at all times.

\begin{remark}\label{remsolB}
In the case where the solid structure is freely floating in the vertical direction, we deduce from Proposition \ref{propprescBouss} and Remark \ref{remBoussvert} that the vertical coordinate $z_G(t)$ of the center of mass (or equivalently its distance $\delta_G$ to its equilibrium position)  is found by solving the second order nonlinear ODE deduced from \eqref{ODE1} by replacing
$q_{\rm i}^2/h_{\rm w}$  by $q_{\rm i}^2/h_{0}$ in the nonlinear terms  $F_{\rm NL}$ and $H_{\rm NL}$.
\end{remark}
\begin{remark}
The dispersive component of the Boussinesq system does not contribute to the ODE governing the motion of the solid; indeed, the third diagonal coefficient of the dispersive correction ${\mathcal M}_{\rm B}$ of the added mass matrix vanishes in Proposition \ref{propprescBouss}. They have however an incidence on the motion of the object since they modify the wave field and in particular the terms $\zeta_{\rm e,\pm}$ that appear in the damping/excitation force $\rho g (\zeta_{\rm e,+}x^*_+-\zeta_{\rm e,-}x^*_-)$.
\end{remark}
\begin{remark}\label{remvanish}
More generally, the dispersive term does not play any role in the interior region. Indeed, from the first equation of \eqref{eqnum1B}, one knows that $q$ is linear in $x$ in the interior region $\cI$. It follows that the dispersive term $-\frac{h_0^2}{3}\dx^2\dt q$ identically vanishes on $\cI$.
\end{remark}
\subsubsection{The discrete equations}

The approach is the same as for the shallow water equations in \S \ref{sectnumSW}; therefore, we only sketch the adaptations one has to perform.\\
The Boussinesq equations \eqref{eqnum1B} differ from the nonlinear shallow water equations \eqref{eqnum1} by the presence of the dispersive term ${\mathcal D}$ and a slight modification in the second component of the flux, namely, ${\mathcal F}$ must be replaced by ${\mathcal F}_B$ as in \eqref{eqnum1B}. We shall therefore use a numerical scheme based on  the  finite volume discretization \eqref{FV}, namely, 
\begin{equation}\label{FV_B}
({\mathcal D}U)^{n+1}_j=({\mathcal D}U)^n_j-\alpha\big({\mathcal F}^n_{j+1/2}-{\mathcal
  F}^n_{j-1/2}\big)+\delta_t S^n_j,
\end{equation}
where the flux numerical flux ${\mathcal F}_{j+1/2}^n$ is still given by \eqref{LFadapt1}-\eqref{LFadapt2} ---  {\it for the sake of clarity, we shall denote ${\mathcal F}$ instead of ${\mathcal F}_B$ throughout this section.}

 The only thing that remains to be specified is the discretization we shall use for $({\mathcal D}U)^{n}_j$. We use 
$$
({\mathcal D}U)^n_j=\big(\zeta^n_j , q^n_j-\frac{h_0^2}{3}{\mathfrak d}^2_jq^n \big),
$$
where ${\mathfrak d}^2_jq^n$ is the classical centered second order approximation of $\dx^2$ except in the interior region where it is equal to zero,\begin{equation}\label{d2sing}
{\mathfrak d}^2_j Q=\frac{1}{\dx^2}\begin{cases}
Q_{j-1}-2 Q_j +Q_{j+1}& \mbox{if } j<j_--2\mbox{ or }j>j_++2\\
0 & \mbox{if } j_--2\leq j\leq j_++2.
\end{cases}
\end{equation}
This discretization is motivated by Remark \ref{remvanish} that shows that the dispersive term identically vanishes in the interior region. We also discarded the dispersive term in the first two cells of the exterior domain in order to avoid the singularity at the contact line. This means that in these cells,  the hydrodynamical model considered locally is the nonlinear shallow water system, which is still physically relevant but less precise than the Boussinesq model\footnote{In the nonlinear shallow water model, terms of order $O(\mu)$ are discarded, while only $O(\mu^2)$ terms are neglected in the Boussinesq model.}. This strategy consisting in switching locally in the vicinity of a singularity to a less precise but more robust model is often used to handle wave-breaking for instance \cite{BCLMT}. The following proposition generalizes to the Boussinesq system the result proved in Proposition \ref{propdiscBouss} for the nonlinear shallow water equations. We omit the proof.
\begin{proposition}\label{propdiscBouss}
Let us consider a floating body in purely vertical motion and denote by $z_G$ the vertical coordinate of its center of mass at time $t^n$. Let the discretization of the Boussinesq equations be furnished by \eqref{FV_B}-\eqref{d2sing} with flux \eqref{LFadapt1}-\eqref{LFadapt2} and with the source term provided by \eqref{Snj}.\\
If the initial condition $(h^0,Q^0)$ satisfies \eqref{discreteCI}, then, for all $n\in \N$ and $j_-<j<j_+$, one has $h^n_j=h^n_{{\rm w}, j}$. 
\end{proposition}
\begin{remark}
In the case of a freely floating object, we follow the same lines as in \S \ref{sectfreenum} and Remark \ref{remsolB} to find the motion of the solid.
\end{remark}

\section{Numerical simulations}\label{sectcomput}

We present here some numerical computations based on the schemes introduced in Section \ref{sectdiscrete}. We first consider in \S \ref{sectnumSW} the case of the nonlinear shallow water equations, and then in \S \ref{sectnumBouss} the Boussinesq equations.

\subsection{Numerical simulations for the nonlinear shallow water model}\label{sectnumsimSW}

Throughout this section, we shall consider a floating object as in Figure \ref{fig_shape}. It consists of the union of a rectangular box of width $2R$ and height $2R\sin(\pi/3)-R$ and, at the lower bottom, of a portion of disk of radius $2R$ and whose center is located at the vertical of the middle of the top of the solid, denoted by $C$. In all the computations presented below, we take $R=10{\mathtt m}$.\\
For the fluid, we shall always assume that the depth at rest is $h_0=15{\mathtt m}$, that the density of water is $\rho=1000 \mathtt{kg}\cdot {\mathtt m}^{-3}$, we also take $g=9.81 {\mathtt m}\cdot {\mathtt s}^{-2}$ for the acceleration of gravity.

\begin{figure}
\includegraphics[width=0.5\textwidth]{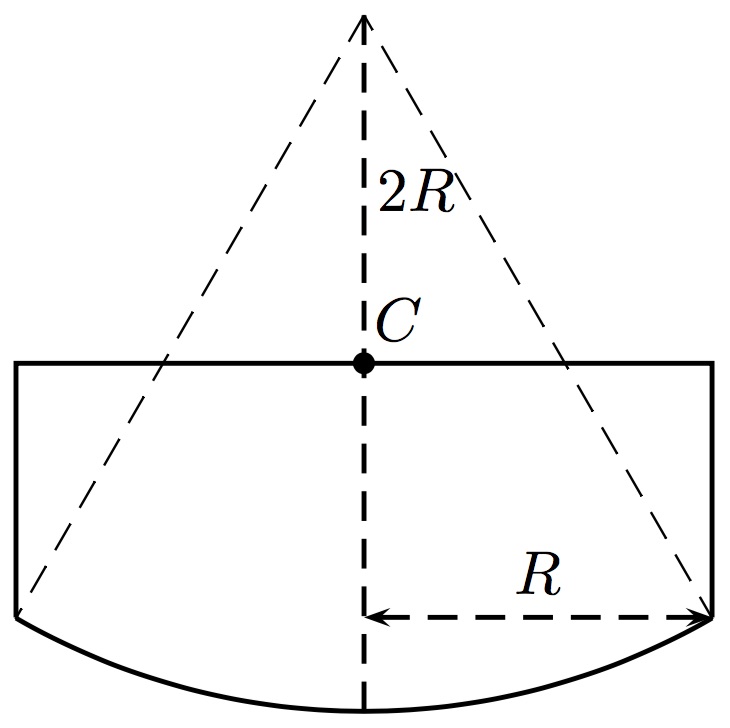}
\caption{Shape of the floating structure}
\label{fig_shape}
\end{figure}

\subsubsection{The case of a fixed object}

We assume here that the floating object is maintained fixed, at the location $C=(150,4.57)$ (this particular height corresponds to the equilibrium state for the configuration considered in \S  \ref{sectnumfree} below). The fluid is initially at rest but forced on the left boundary ($x=0$) with a periodic incoming swell of amplitude $1{\mathtt m}$ and period $T=15{\mathtt s}$. The solution is represented in Figure \ref{fig_fixed} at different times.\\
We represented the solid structure in the first plot, but not in the others, in order to insist on the fact that we solve the equations on the full computational domain and that, with our choice of pressure, the surface of the fluid matches at machine precision the boundary of the solid in the wetted region. There is no need to impose this matching as an extra constraint, consistently with the result stated in Proposition \ref{propWWst} in the continuous case, and Proposition \ref{propdiscSW} in the discrete case under consideration here.\\
As expected from the conservation of mass equation
$$
\dt \zeta+\dx q=0,
$$
the discharge is constant in space (but not in time) in the wetted region since $\dt\zeta_{\rm w}=0$ when the solid is fixed. \\
We can also see that part of the wave is reflected, while the other part, is transmitted to the other side of the solid; this is of course because the flow is allowed to flow underneath the solid.
\begin{center}
\begin{figure}
\includegraphics[width=0.48\textwidth]{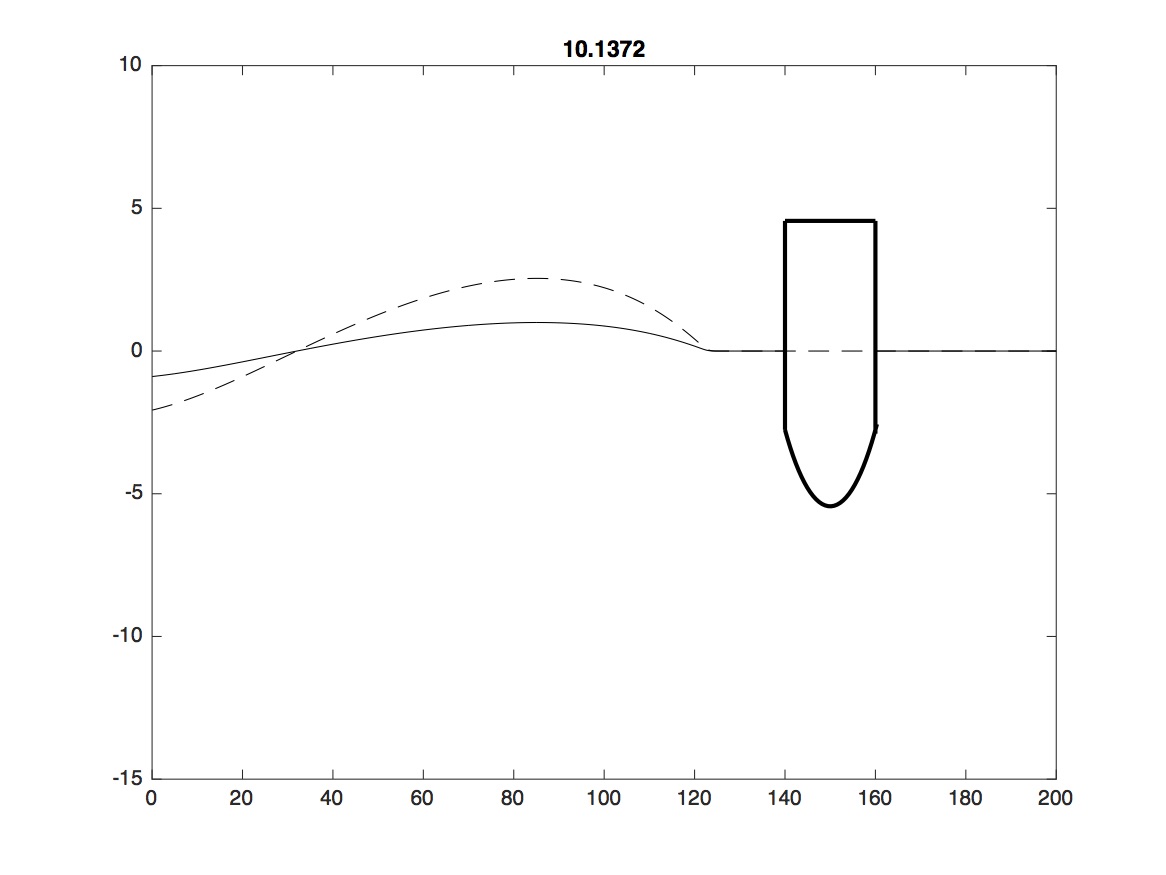}
\includegraphics[width=0.48\textwidth]{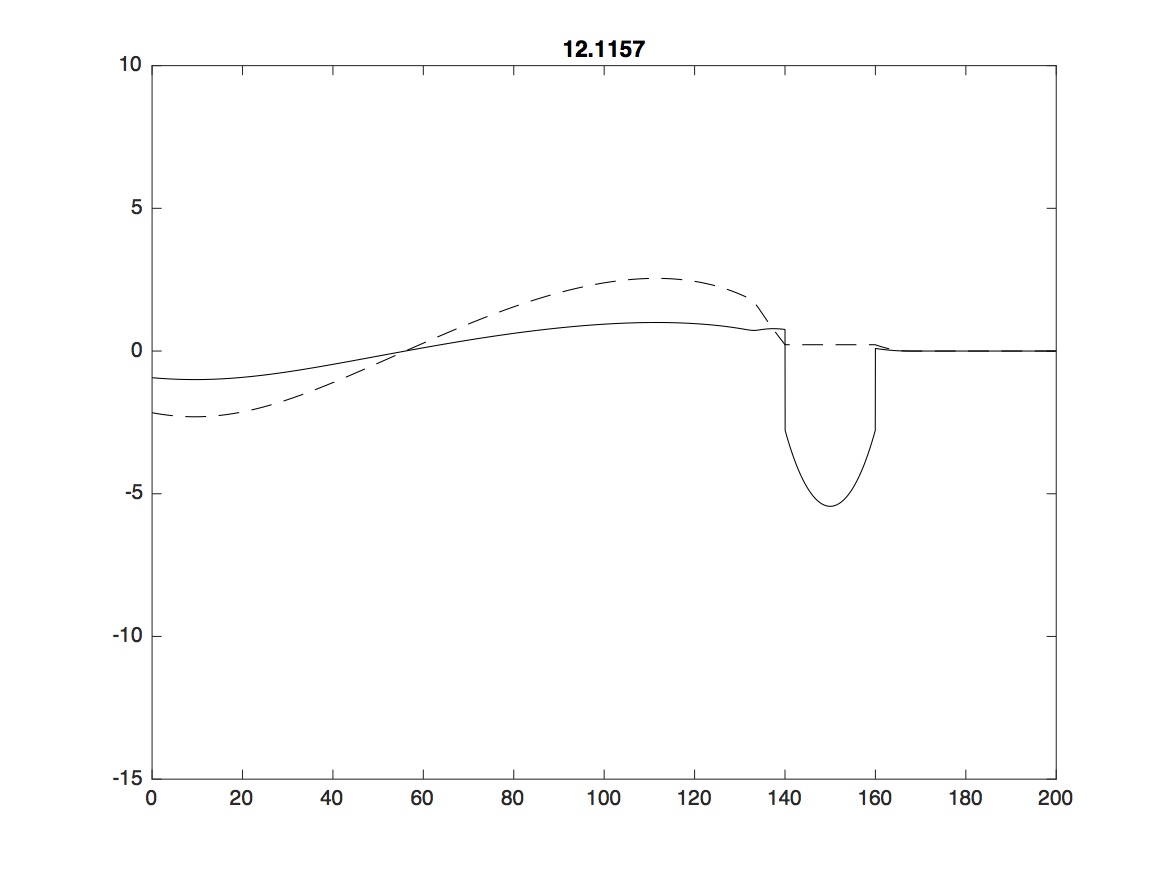}
\includegraphics[width=0.48\textwidth]{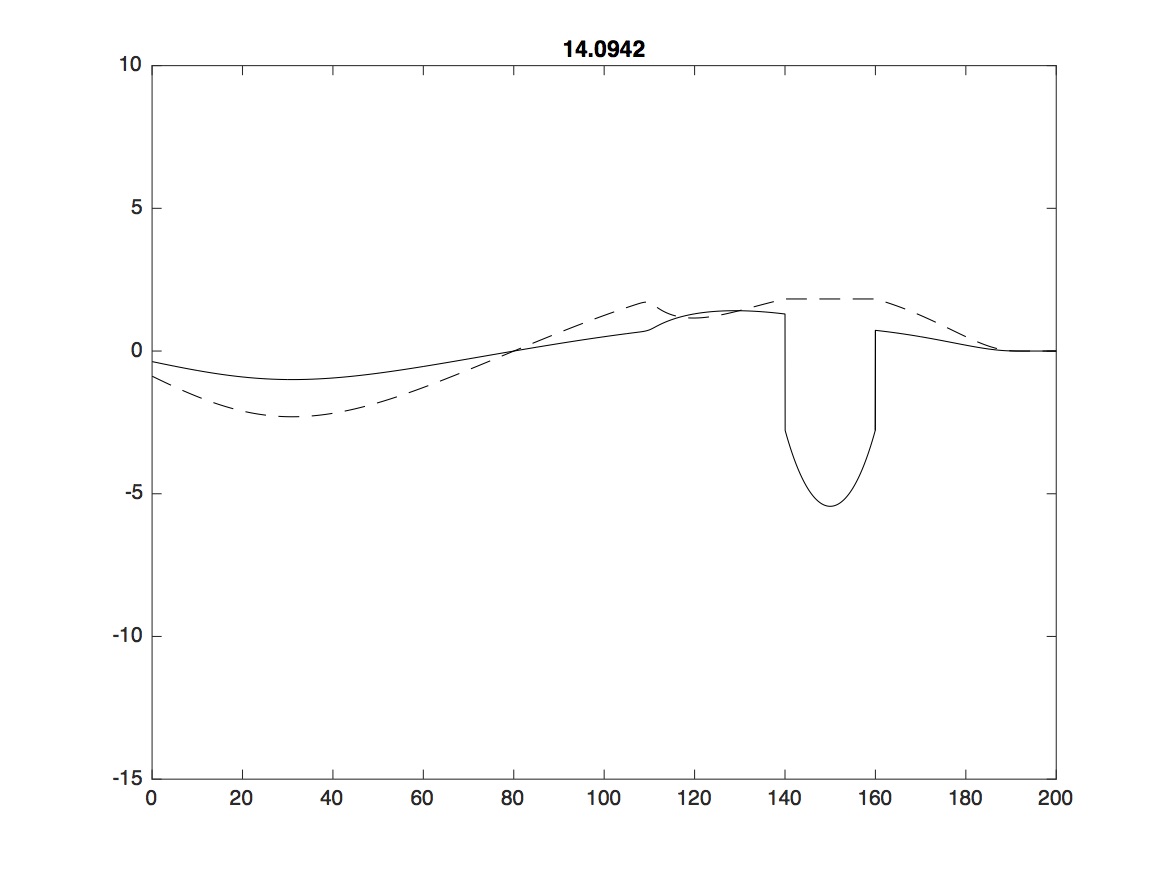}
\includegraphics[width=0.48\textwidth]{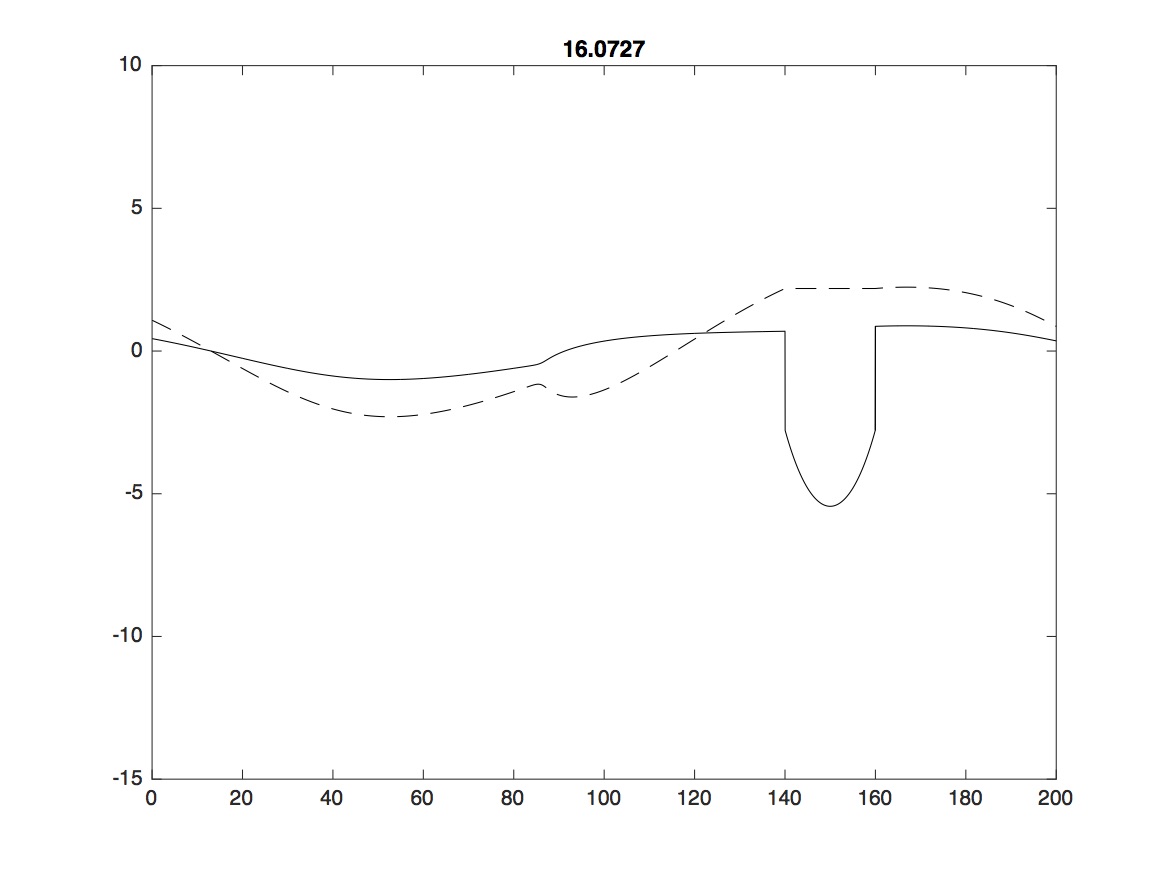}
\caption{A wave arriving on a fixed floating structure. Surface elevation (full) and discharge (divided by a rescaling factor 5, dash).}
\label{fig_fixed}
\end{figure}
\end{center}

\subsubsection{The case of an object in prescribed vertical motion}\label{sectforcednum}

We represent in Figure \ref{fig_forced} the waves created by the floating object when it is in forced vertical motion. We took an initial position corresponding to $z_{\rm C}=4.57{\mathtt m}$, and to an oscillation of $10 {\mathtt s}$ and amplitude $2{\mathtt m}$.\\
The discharge is no longer constant in the wetted region since $\dt{\zeta}_{\rm w}$ is not zero. Since, for a purely vertical motion, it is a function of time only, the discharge is linear in space in the wetted region, as observed in the computations.
\begin{center}
\begin{figure}
\includegraphics[width=0.48\textwidth]{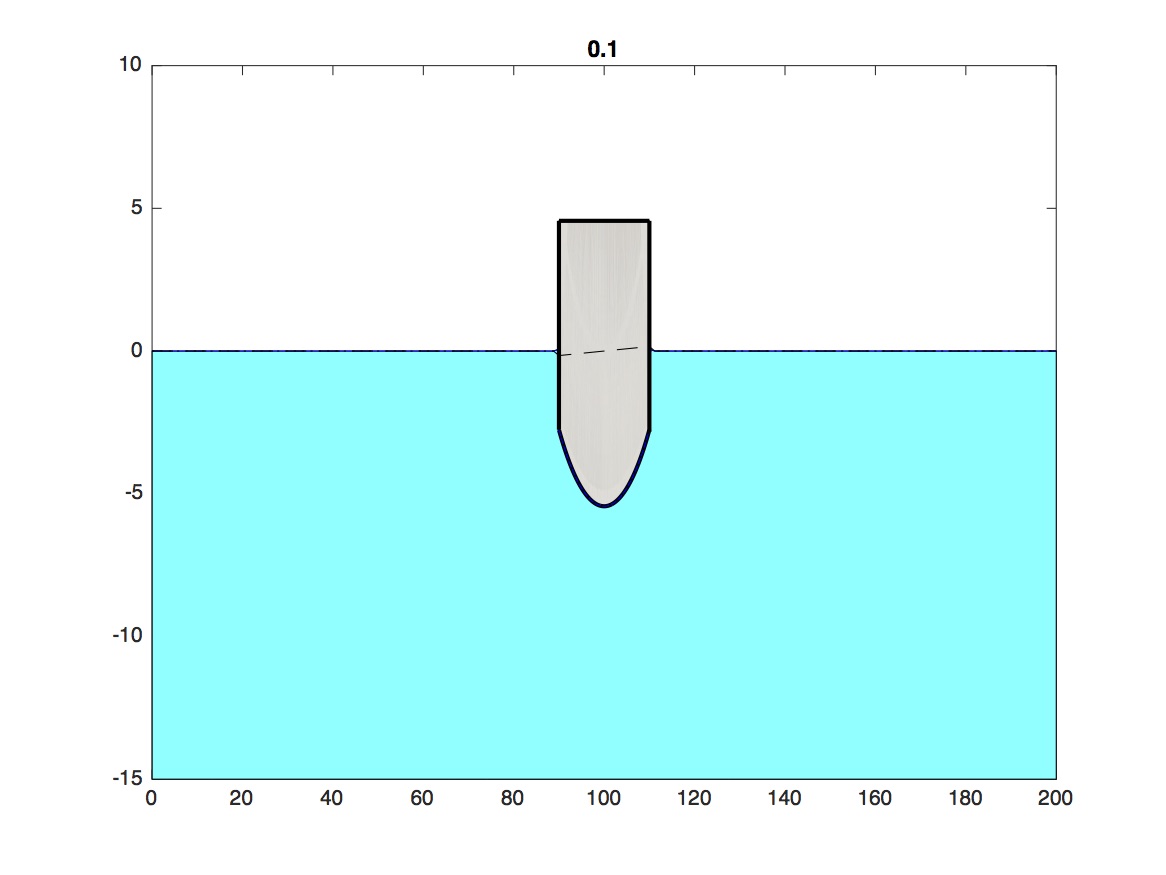}
\includegraphics[width=0.48\textwidth]{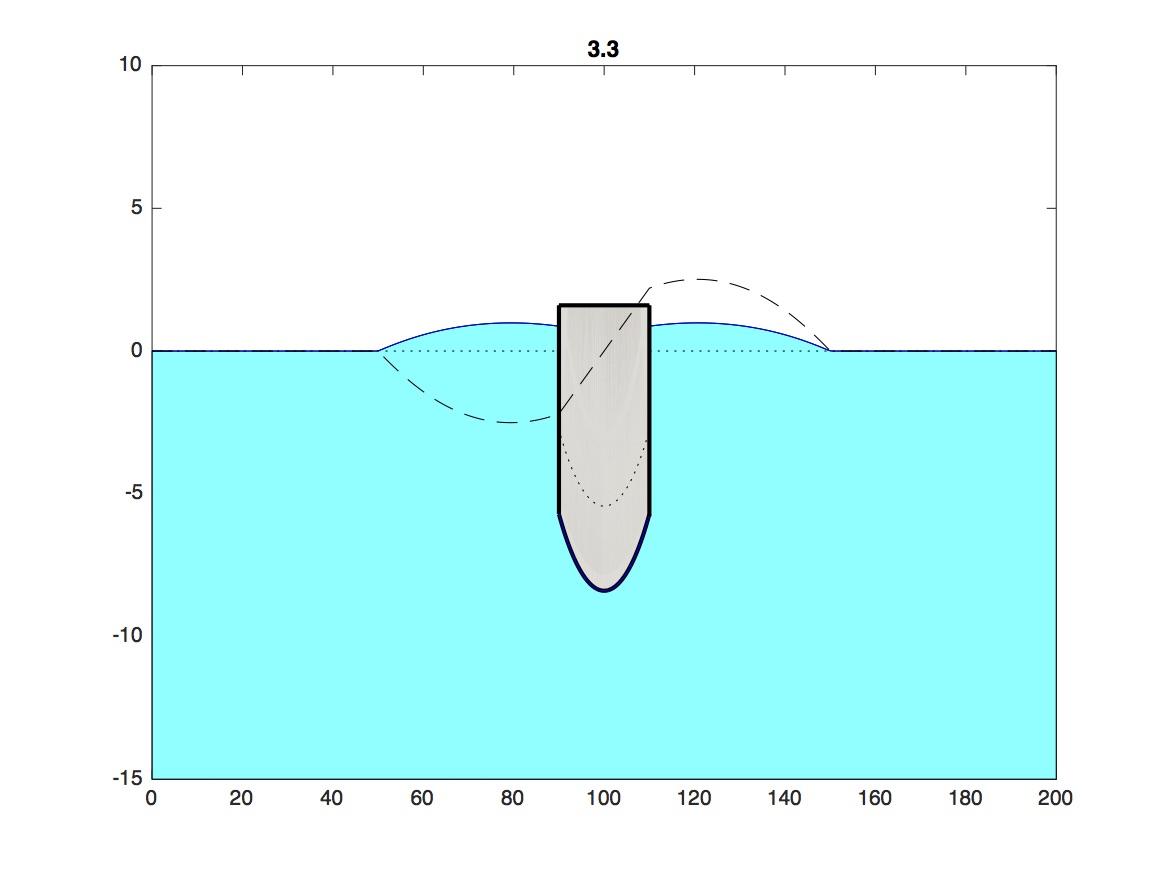}
\includegraphics[width=0.48\textwidth]{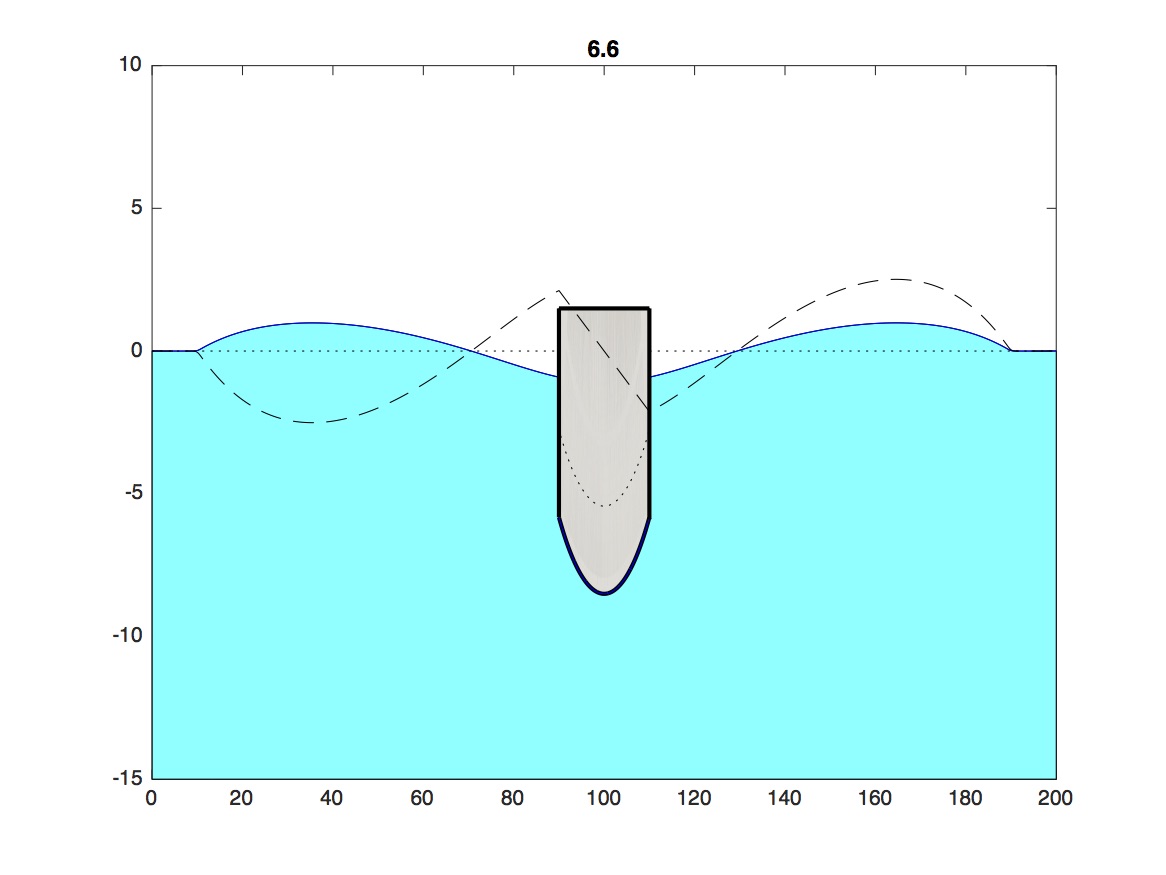}
\includegraphics[width=0.48\textwidth]{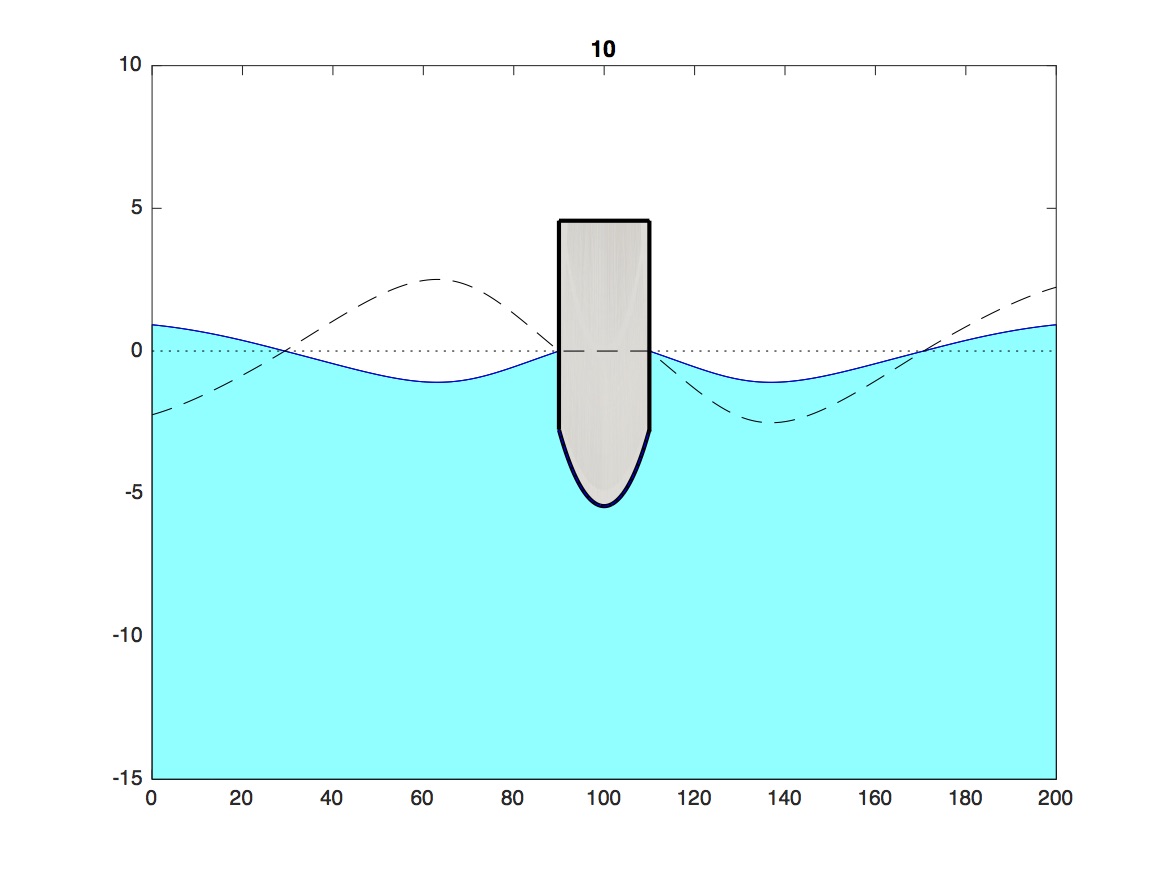}
\caption{A solid in forced vertical motion. Surface elevation (full) and discharge (divided by a rescaling factor 5, dash), initial position (dots).}
\label{fig_forced}
\end{figure}
\end{center}
In order to provide some validation of these computations, we can observe that owing to Remark \ref{remvalforced}, the elevation of the water at the contact points $x_\pm$ is given by
$$
\zeta_{\rm e}(t,x_\pm)=\big(\tau_0(\frac{x_+-x_-}{4\sqrt{g}}\dot{\delta}_G)\big)^2-h_0,
$$
with $\tau_0(\cdot)$ as in \eqref{deftau0}. Since $\dot{\delta}_G$ is in the present case a known function of time, this formula provides an explicit exact solution for the water elevation at the contact points. In Table \ref{tableCV_forced}, we reproduce the error between the solution computed with our numerical scheme and this explicit formula (the configuration considered is the same as in Figure \ref{fig_forced}, and the error computed corresponds to the $L^\infty$-norm of the difference between the exact and computed solutions over one period $T=10{\mathtt s}$). As expected, a first order convergence is observed.
\begin{table}
\begin{tabular}{| l | l | l | l | l |}
\hline
$\delta_x$ & 0.00625 & 0.0125 & 0.0250 & 0.05\\
\hline
 Error & 0.00106 &0.00212& 0.00423 &0.00846\\
 \hline
 \end{tabular}
 \caption{Convergence to the exact solution for a solid in forced motion.}
 \label{tableCV_forced}
\end{table}

\subsubsection{The case of an object freely floating in the vertical direction}\label{sectnumfree}

We now consider the case where the solid is freely floating in the vertical direction, as in \S \ref{sectfreenum}. Since the motion of the solid is found through the resolution of Newton's laws, it is important to specify here the volumic density $\rho_{\rm s}$ of the solid. We take here $\rho_{\rm s}=0.5\rho$. The mass of the solid is then given by ${\mathfrak m}=\rho_{\rm s}{\mathcal V}$ where the volume ${\mathcal V}$ of the solid is then given by
$$
{\mathcal V}=R^2(\sqrt{3}+2\pi/3-2).
$$
One also easily computes the vertical coordinate of the point $C$ at equilibrium,
$$
z_{\rm C,eq}=\frac{R}{2}(1-\frac{\rho_{\rm s}}{\rho})(\sqrt{3}+\frac{2\pi}{3}-2);
$$
in the numerical computations presented in this section, an horizontal line on the floating object marks the contact line at this equilibrium state. In the configuration considered here, one finds $z_{\rm C,eq}=4.57{\mathtt m}$.
    
The first test performed here and represented in Figure \ref{float_eq} represents the return to the equilibrium when the solid is initially below its buoyancy line, with an initial vertical coordinate for the point $C$ given by $z_{\rm C,eq}=2{\mathtt m}$. For this problem, we know from Corollary \ref{coroeq} that the motion of the solid can be found directly by solving the nonlinear ODE \eqref{ODE1expl}. This furnishes a reference solution that we can use to validate our numerical scheme. As shown in Table \ref{tableCV}, there is as expected a first order convergence of the solution computed through the numerical scheme presented in \S \ref{sectnumSW} for the nonlinear shallow water equations in the presence of a floating structure (the error computed in the table corresponds to the $L^\infty$-norm of the difference between the exact and computed positions of the center of mass over the time interval $[0,10{\mathtt s}]$).
\begin{center}
\begin{figure}
\includegraphics[width=0.32\textwidth]{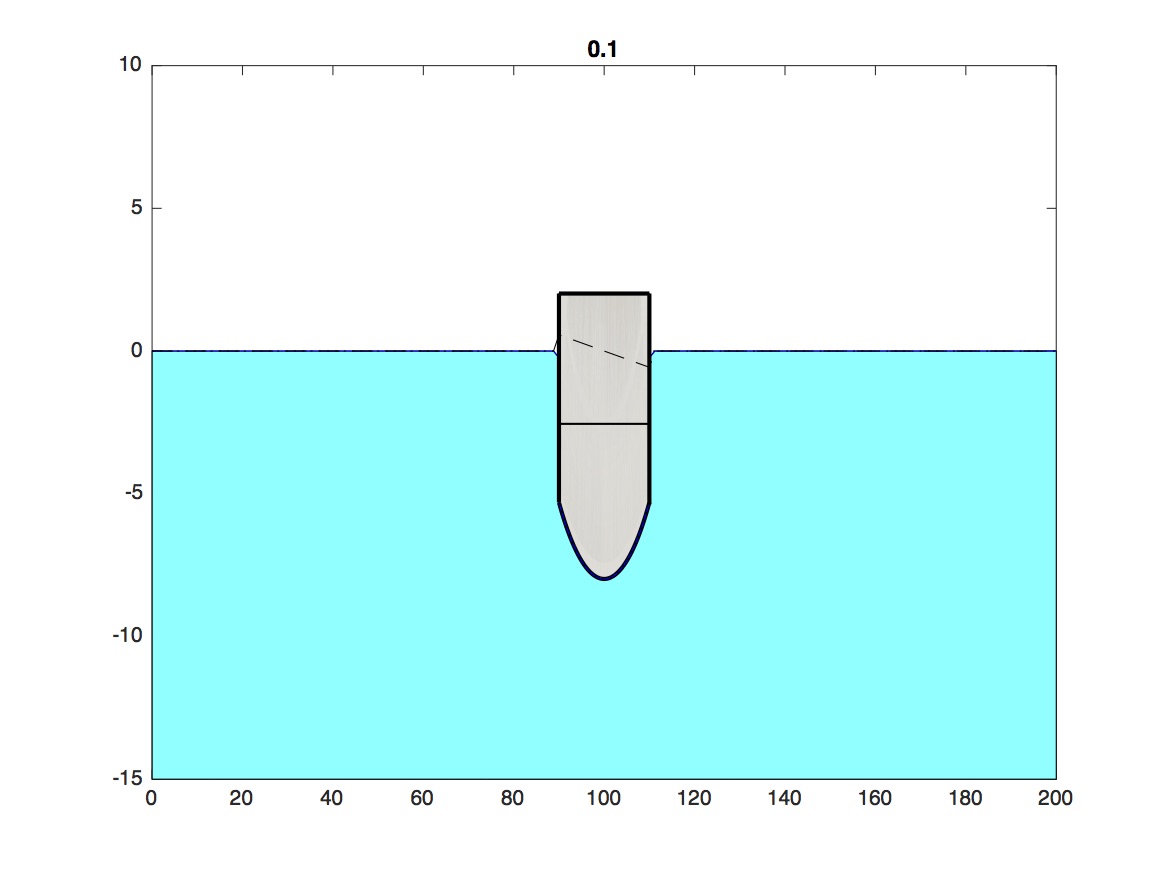}
\includegraphics[width=0.32\textwidth]{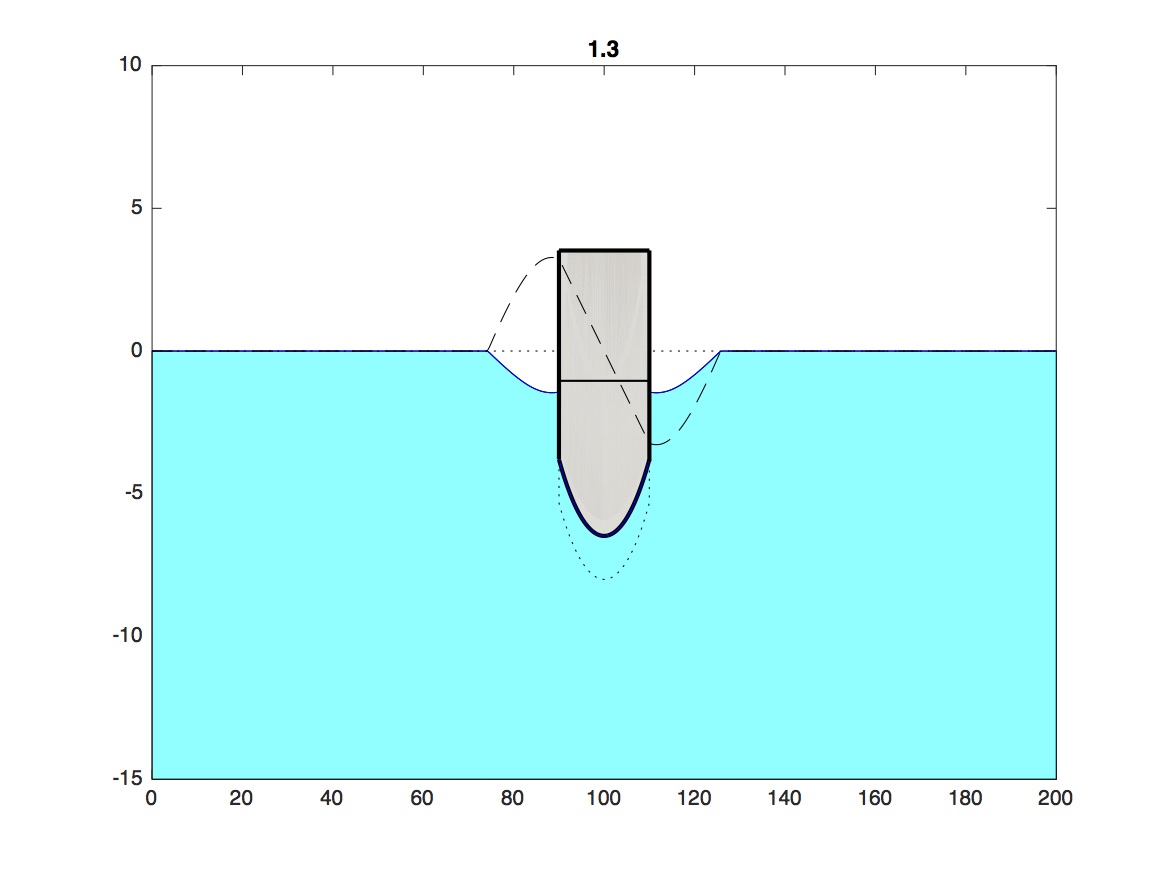}
\includegraphics[width=0.32\textwidth]{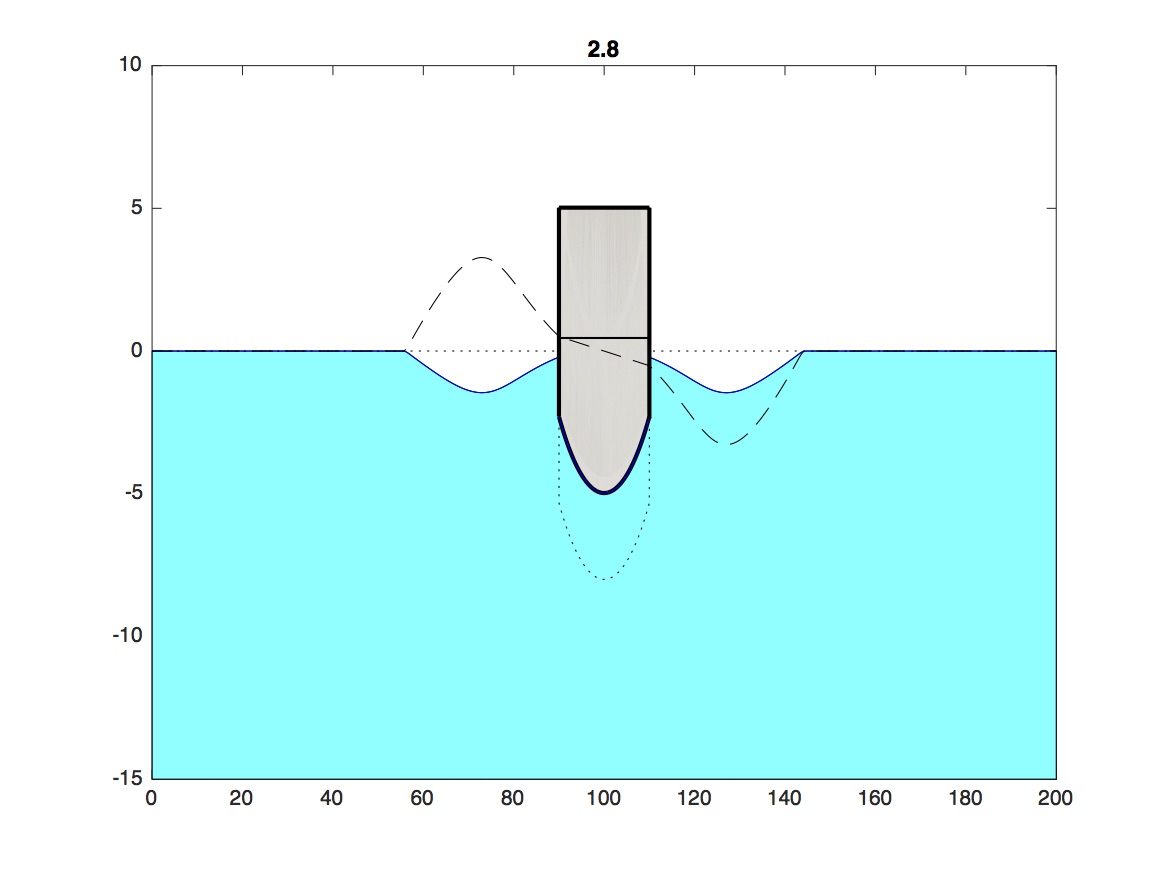}
\includegraphics[width=0.32\textwidth]{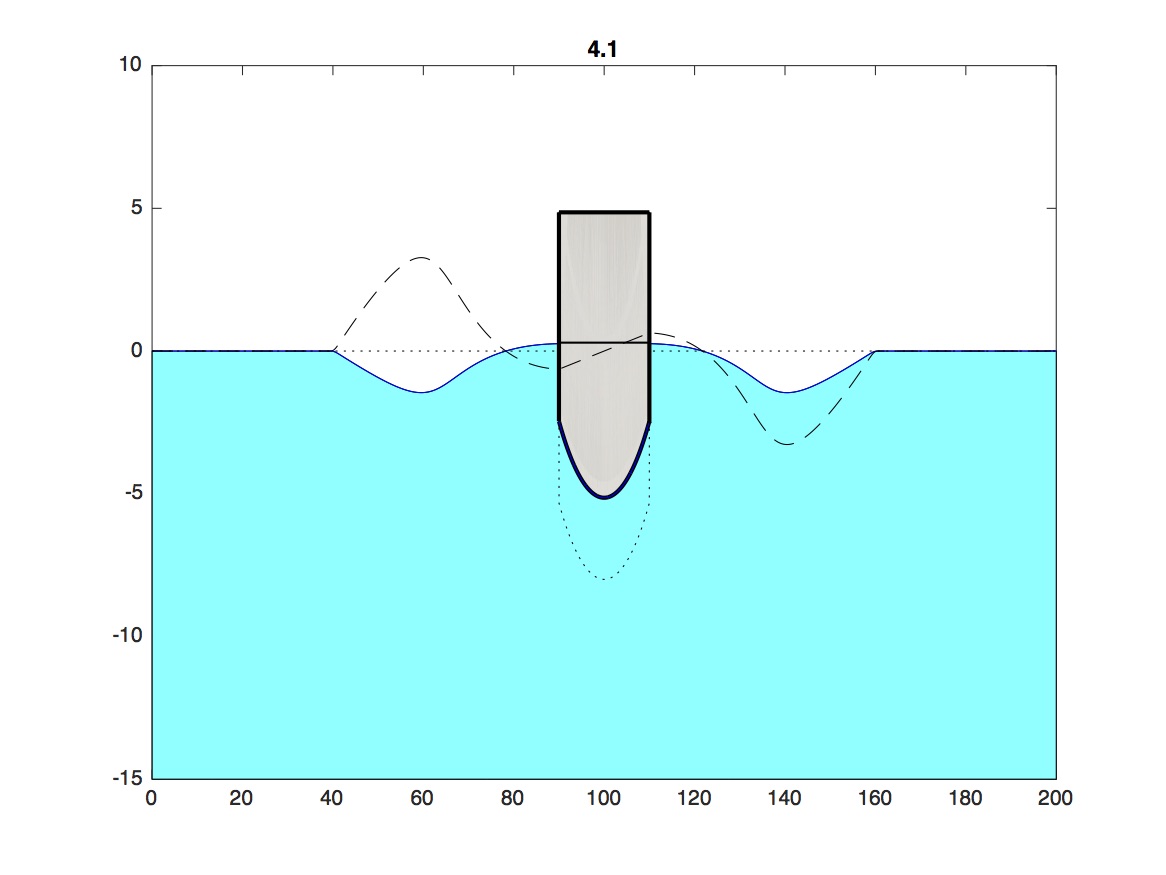}
\includegraphics[width=0.32\textwidth]{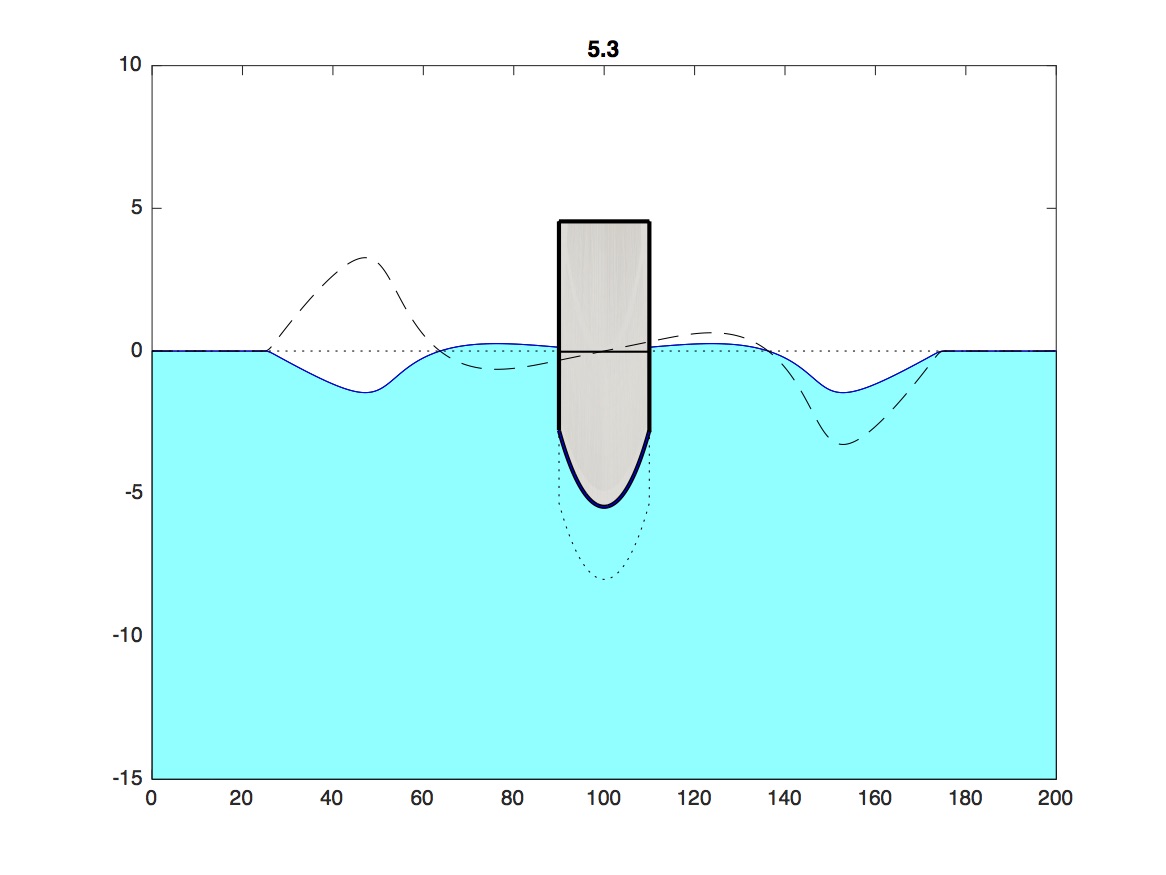}
\includegraphics[width=0.32\textwidth]{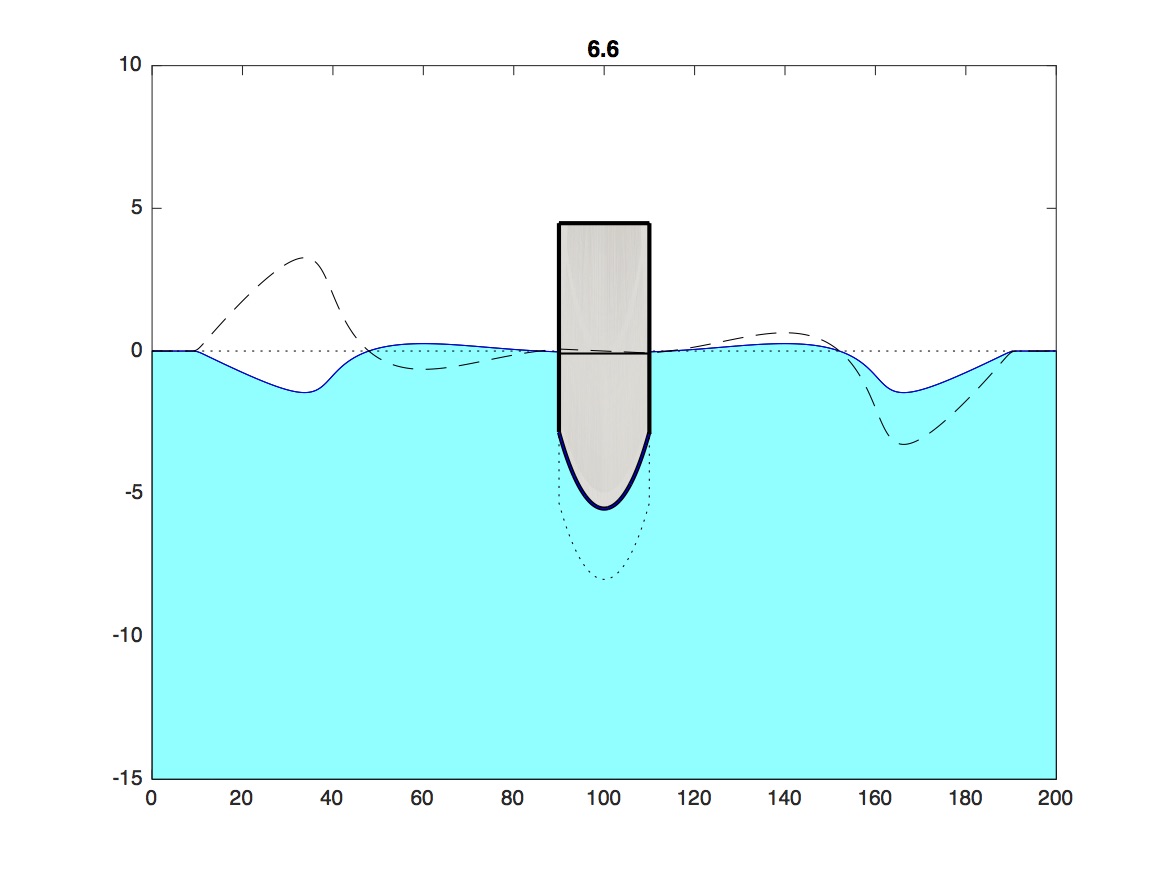}
\caption{Return to equilibrium of a solid initially below its buoyancy line. Surface elevation (full), discharge (divided by a rescaling factor 5, dash), and initial condition (dots). The horizontal line on the structure marks the contact line at the equilibrium.}
\label{float_eq}
\end{figure}
\end{center}
\begin{table}
\begin{tabular}{| l | l | l | l | l |}
\hline
$\delta_x$ & 0.00625 & 0.0125 & 0.0250 & 0.05\\
\hline
 Error & 0.00286 &0.00556& 0.0111 &0.0218\\
 \hline
 \end{tabular}
 \caption{Convergence to the exact solution in the return to equilibrium problem.}
 \label{tableCV}
\end{table}

The second test performed here consists in studying the motion of a solid initially at equilibrium when a waves arrives. The solid is the same as above and the incoming wave is obtained by forcing a sinusoidal wave of amplitude $3.5{\mathtt m}$ and period $20{\mathtt s}$ at the left boundary located at 
$140{\mathtt m}$ from the left boundary of the solid; when it arrives at the floating structure, the wave is near to the point of breaking. The result is represented in Figure \ref{fig_float}.
\begin{center}
\begin{figure}
\includegraphics[width=0.32\textwidth]{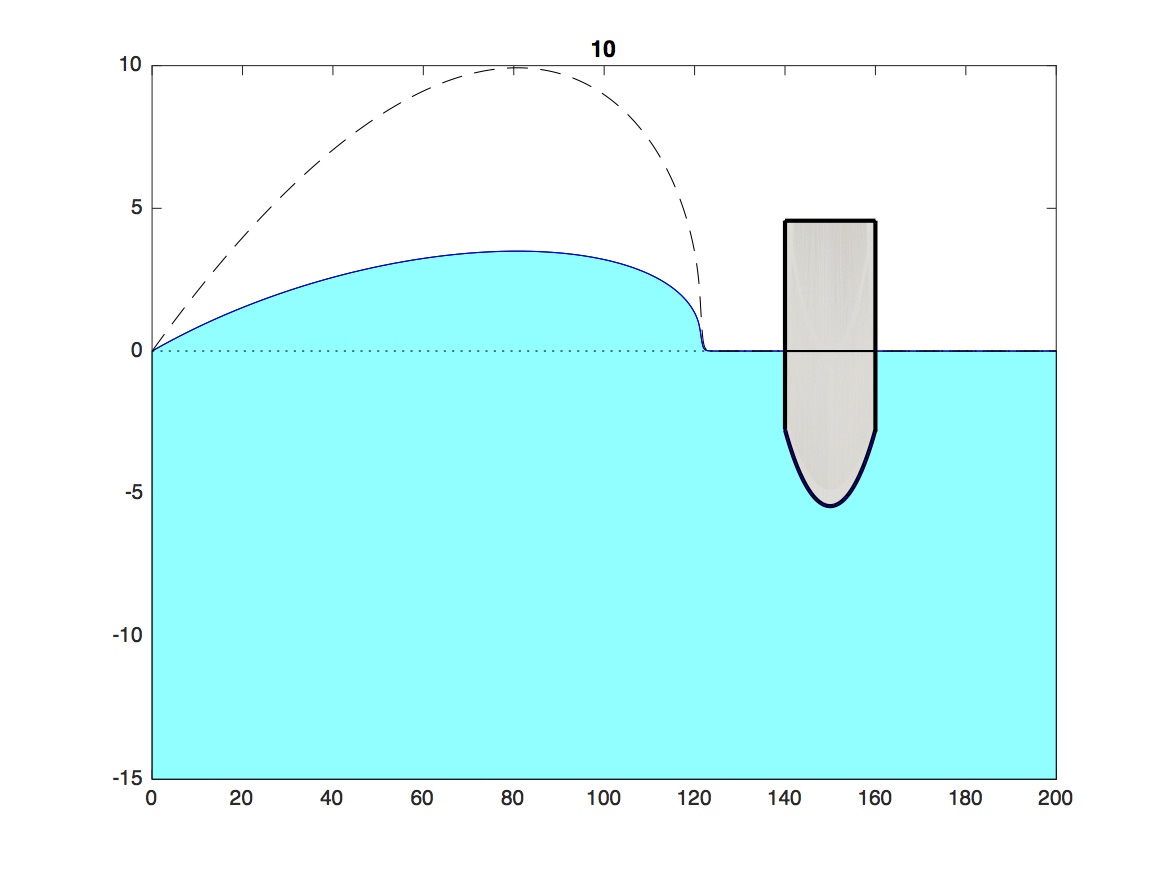}
\includegraphics[width=0.32\textwidth]{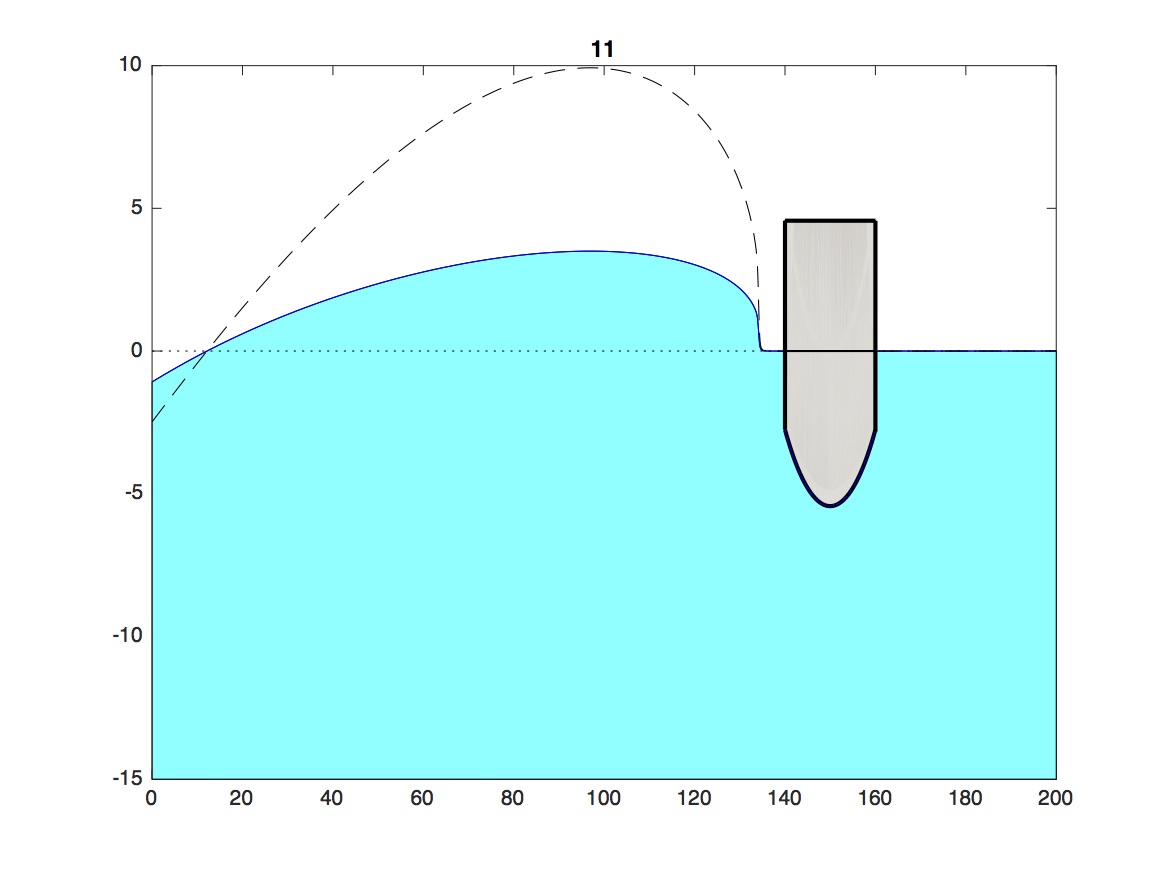}
\includegraphics[width=0.32\textwidth]{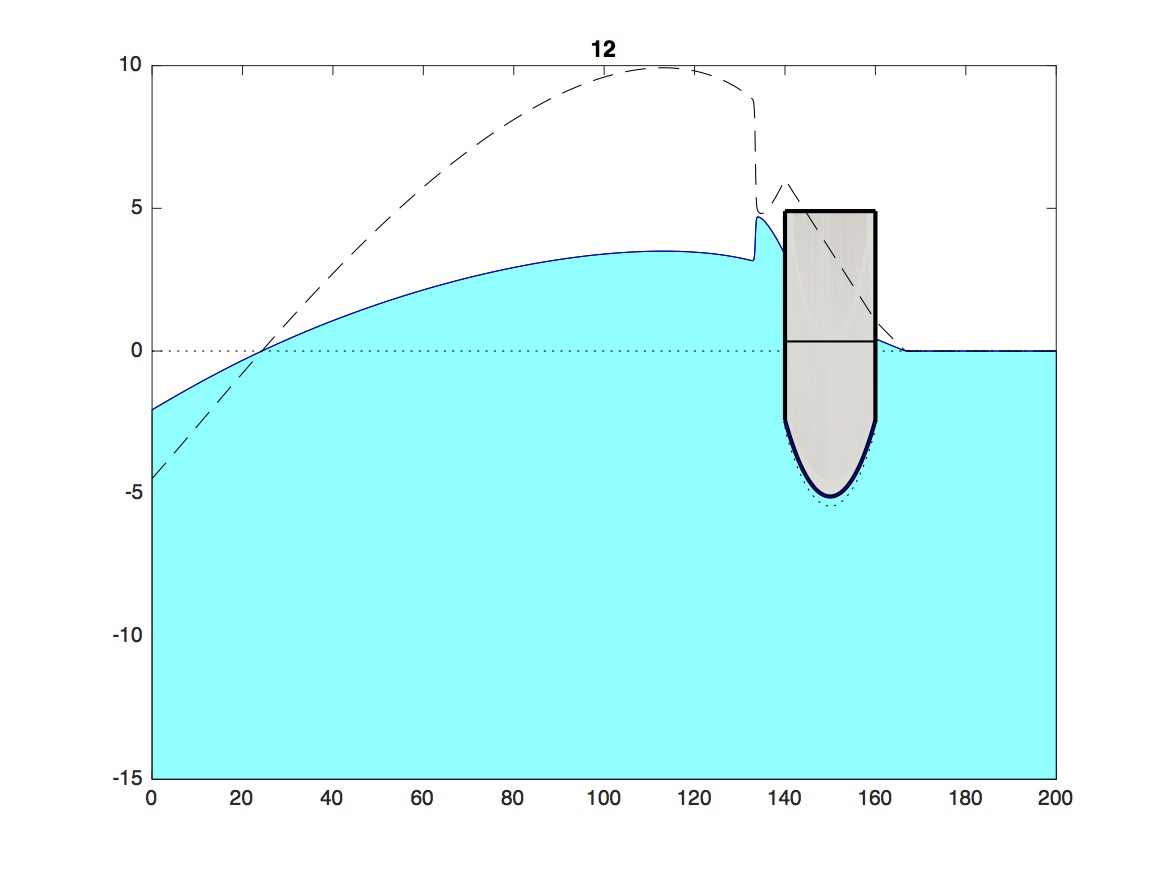}
\includegraphics[width=0.32\textwidth]{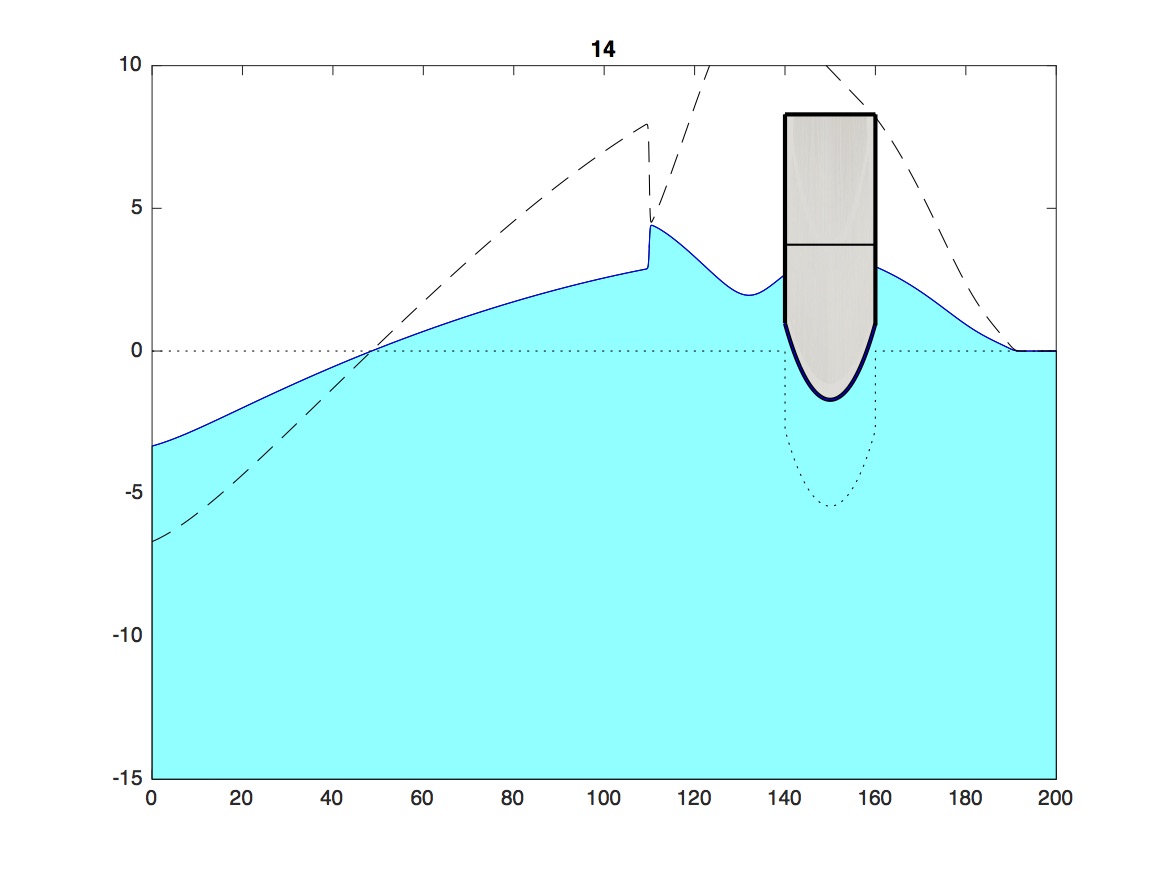}
\includegraphics[width=0.32\textwidth]{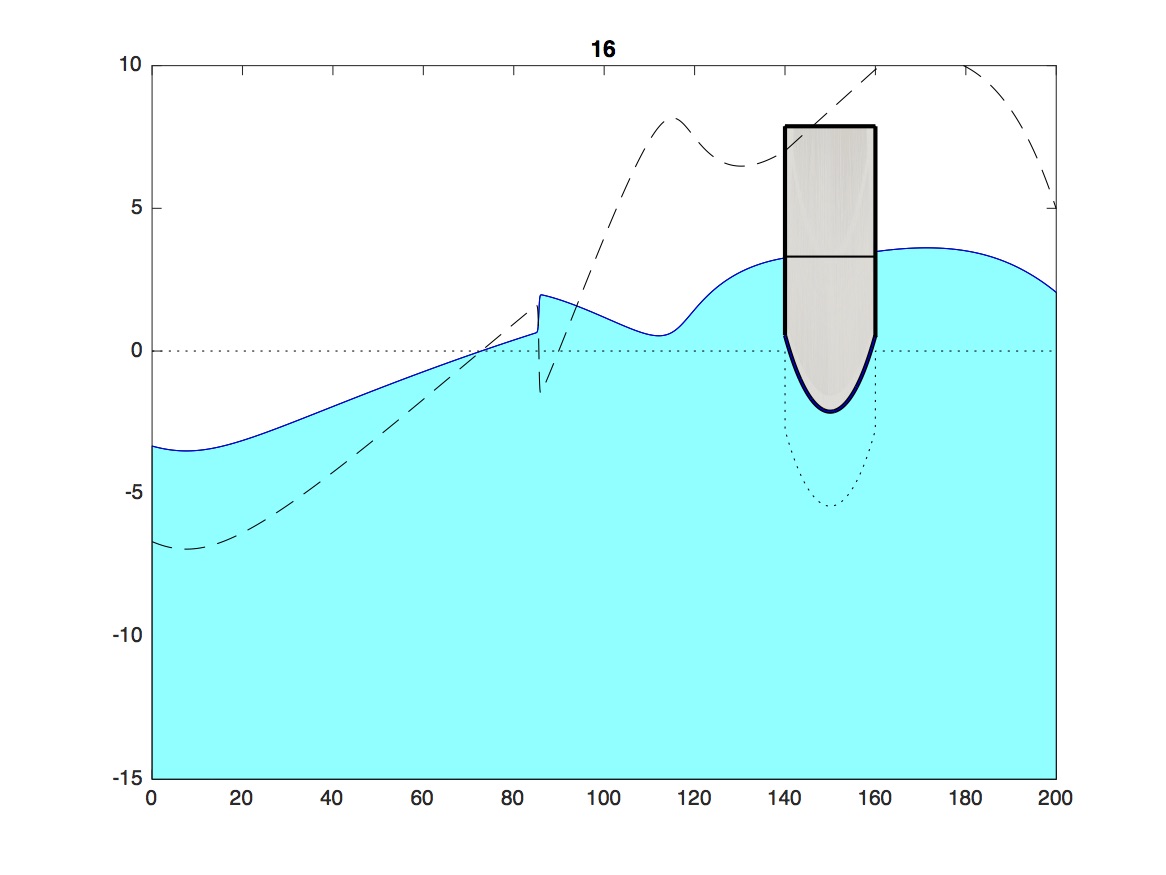}
\includegraphics[width=0.32\textwidth]{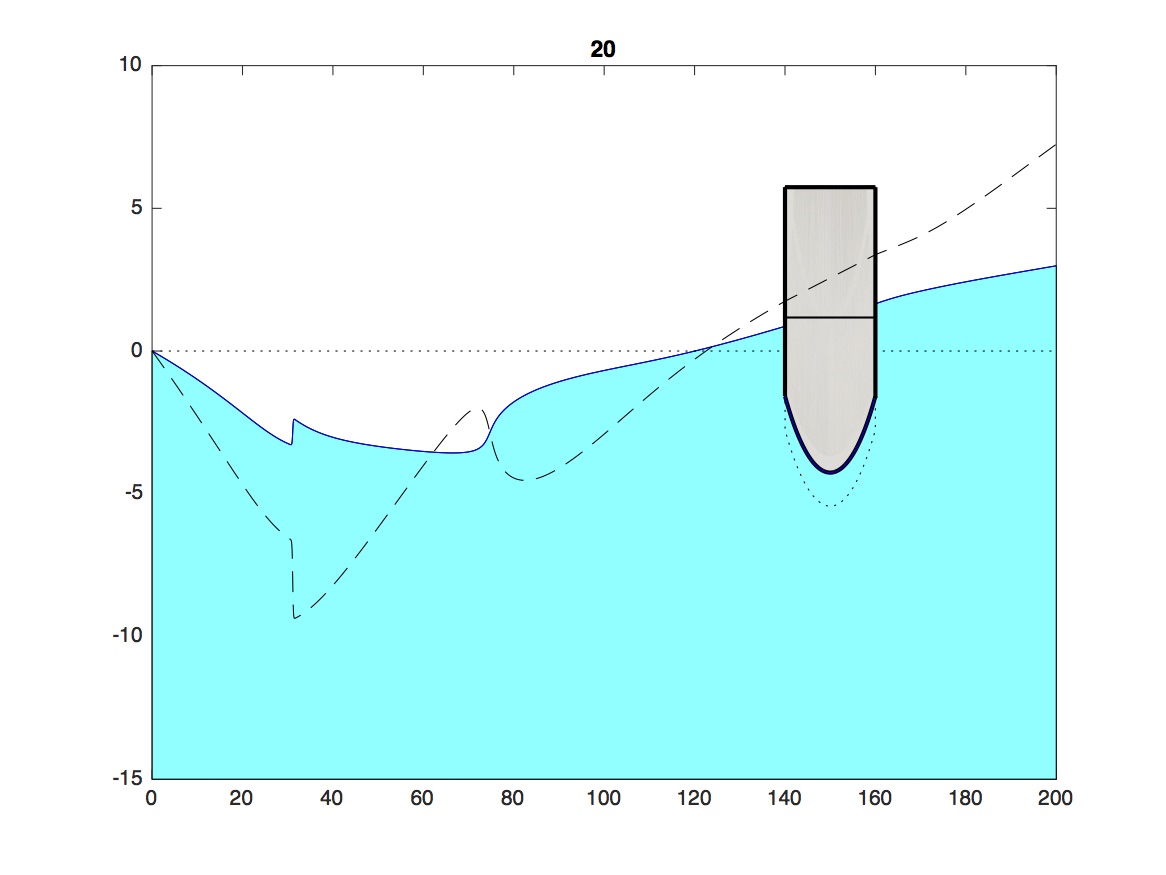}
\caption{A wave arriving on a freely floating structure (nonlinear shallow water model). Surface elevation (full), initial condition (dots) and rescaled discharge (dash). The horizontal line on the structure marks the contact line at the equilibrium.}
\label{fig_float}
\end{figure}
\end{center}

It is also interesting to know the forces exerted on the solid; let us first recall that the sum of the vertical component of these forces can be decomposed into four components
\begin{equation}\label{Fvert}
F_{\rm vert}=\underbrace{-{\mathfrak c}\delta_G}_{F_{\rm restor}}\underbrace{-m_a(\delta_G)\ddot\delta_G}_{:=F_{\rm added}}+\underbrace{\rho g (\zeta_{\rm e,+}w_+^*-\zeta_{\rm e,-}^*)}_{:=F_{\rm D+E}}+F_{\rm NL}(\delta_G,\dot\delta_G,\av{q}),
\end{equation}
where we used the same notations as in \eqref{ODE1};  the component $F_{\rm restor}$ is the resulting restoring force (weight plus Archimedes' force), while $F_{\rm added}$ is the force due to the added mass effect, $F_{\rm D+E}$ stands for the damping and excitation force, and $F_{\rm NL}$ is the nonlinear correction. The forces  exerted on the solid in the configuration considered just above are represented in Figure \ref{fig_forces}.
\begin{center}
\begin{figure}
\includegraphics[width=\textwidth]{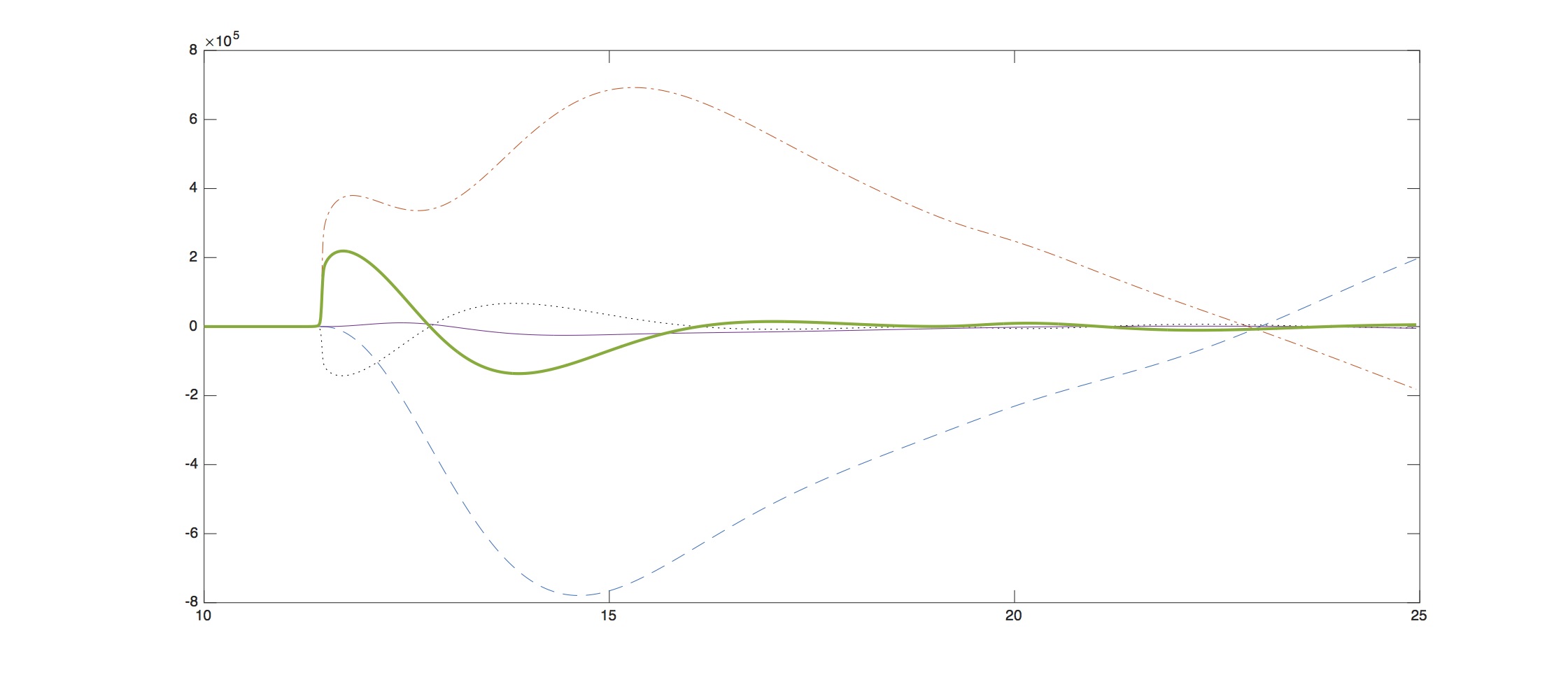}
\caption{Decomposition of the vertical force $F_{\rm vert}$ (thick solid line) exerted on the structure, as in \eqref{Fvert}. Damping-excitation force $F_{\rm D+E}$ (dash-dot), restoring force $F_{\rm restor}$ (dash), force due to added mass $F_{\rm added}$ (dots) and nonlinear corrections $F_{\rm NL}$ (solid).}
\label{fig_forces}
\end{figure}
\end{center}

\begin{remark}
Archimede's principle states that the upward buoyant force that is exerted on a body immersed in a fluid is equal to the weight of the fluid that the body displaces. In the case of a floating object, this quantity is easily computed when the fluid is at rest, but otherwise not intuitive since in order to compute the displaced fluid, one would need to know what the free surface would have been without the solid. In $F_{\rm restor}$ we have only taken into account the standard Archimedes force; the corrections due to the perturbations of the free surface have been included in the damping/excitation force $F_{\rm D+E}$.
\end{remark}

\subsection{Numerical simulations for the Boussinesq model}\label{sectnumBouss}

In absence of any floating structure, the Boussinesq equations \eqref{Boussf} admit in horizontal dimension $d=1$ solitary waves of the form
$$
\zeta=a\big[\mbox{sech}\big(K(x-ct)\big)\big]^2, \qquad q=c \zeta
$$
with 
$$
K=\sqrt{\frac{9a}{12h_0^3+4a h_0^2}} \quad\mbox{ and }\quad c=\sqrt{\frac{gh_0}{1-\frac{4h_0^2K^2}{3}}}.
$$
As a brief illustration of the possibility to implement our approach to nonlinear dispersive wave models, we show in Figure \ref{fig_solit} the numerical simulation of a solitary wave of amplitude $a=3m$ arriving on the floating structure initially at equilibrium. It can be observed that a solitary wave of slightly smaller amplitude is transmitted on the other side of the solid, and that a small part dispersive trail is reflected. The decomposition of the vertical force exerted on the solid during this experiment is reproduced in Figure \ref{fig_forces_sol}.
\begin{center}
\begin{figure}
\includegraphics[width=0.32\textwidth,height=3cm]{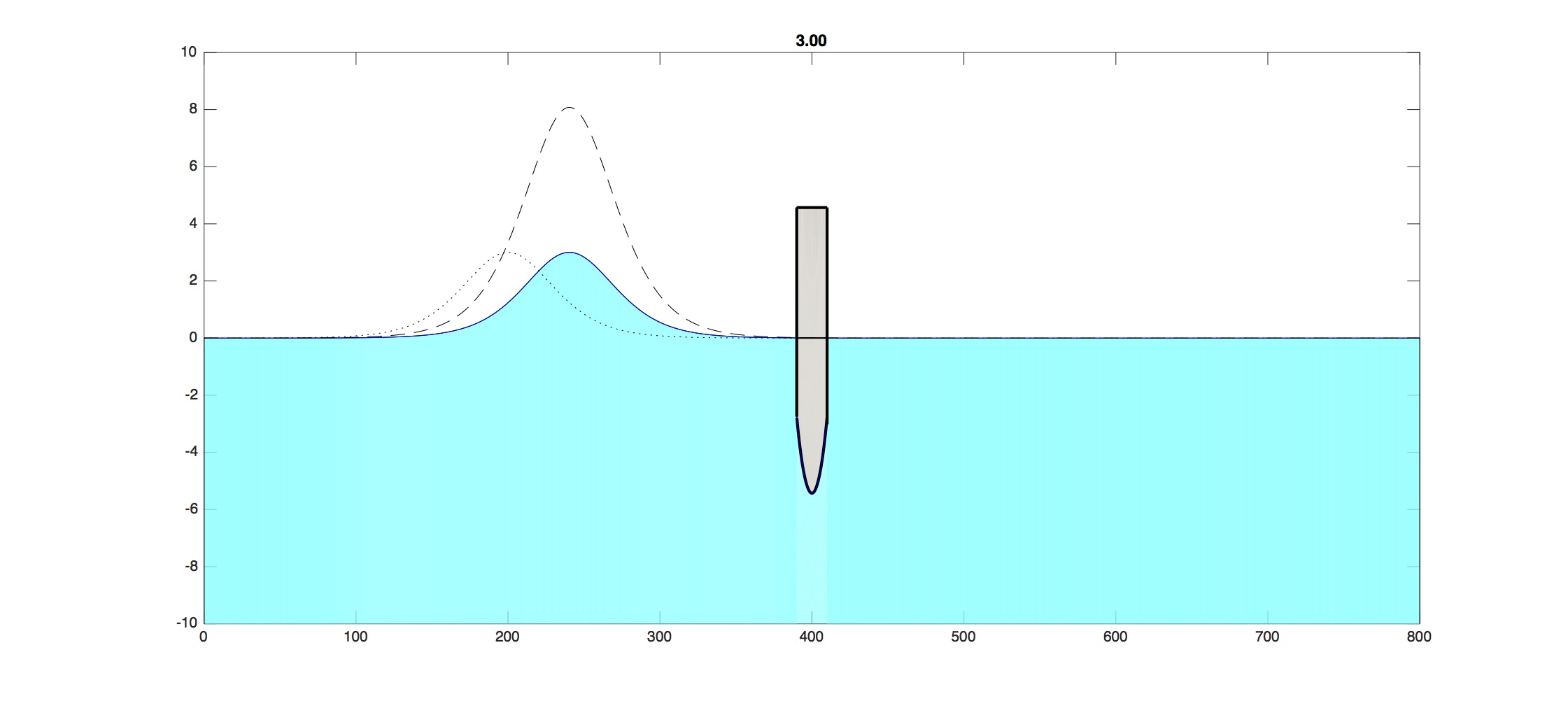}
\includegraphics[width=0.32\textwidth,height=3cm]{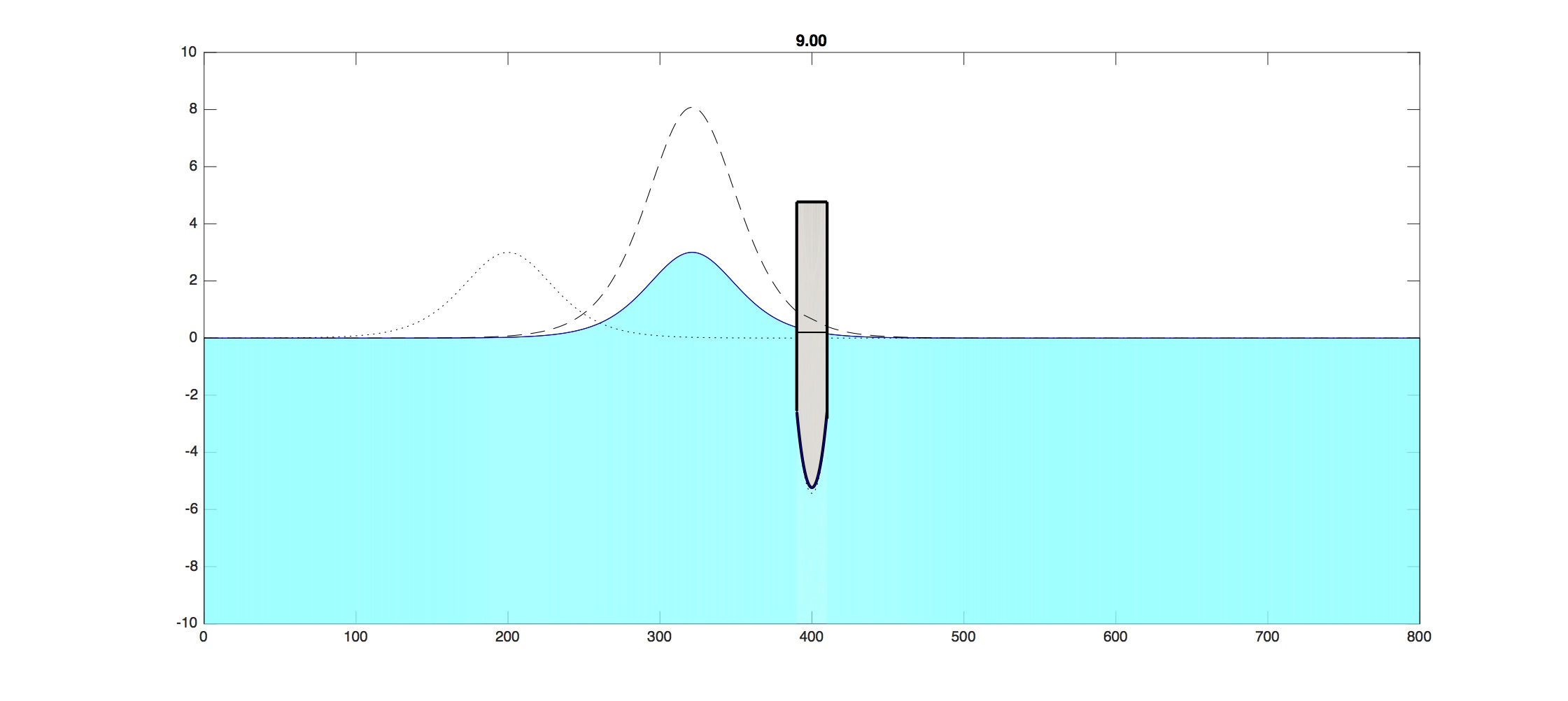}
\includegraphics[width=0.32\textwidth,height=3cm]{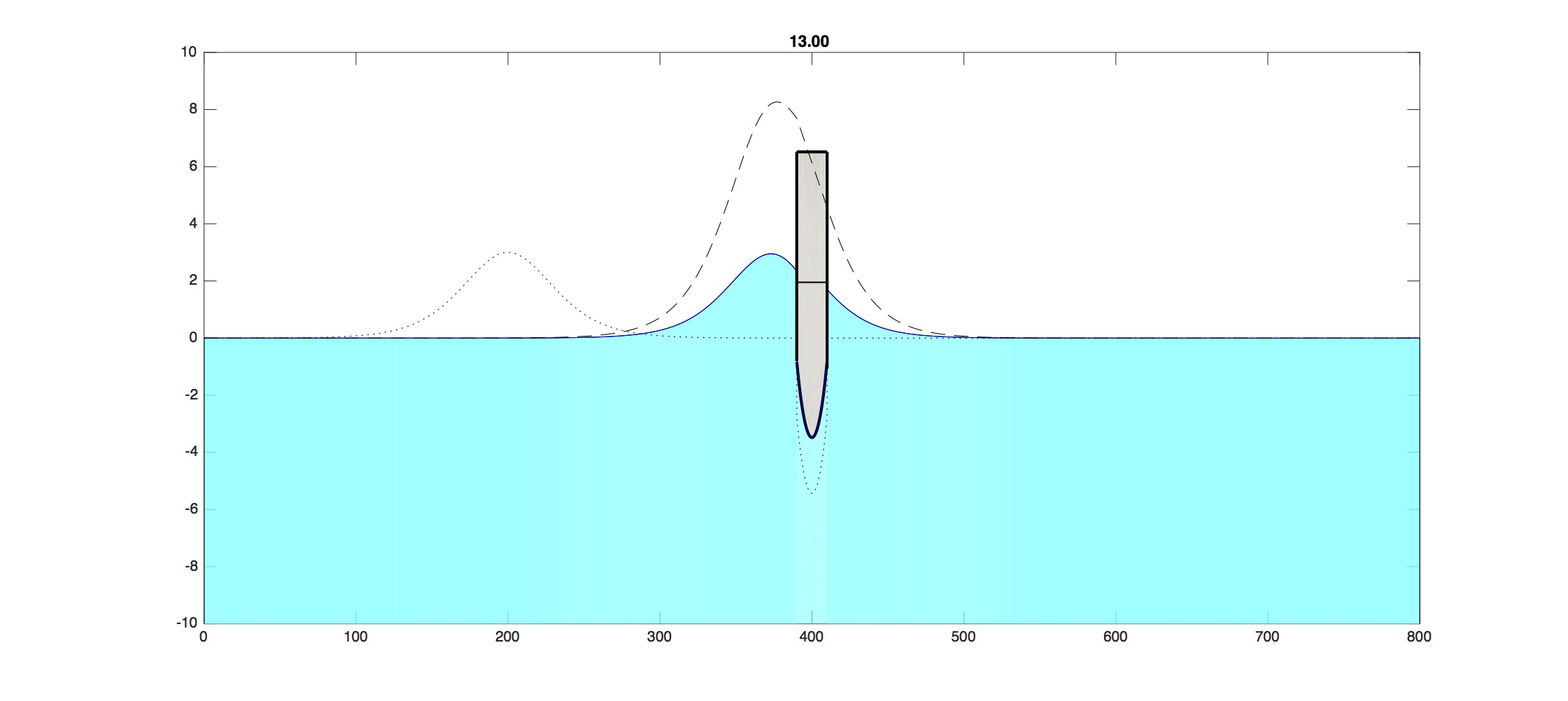}
\includegraphics[width=0.32\textwidth,height=3cm]{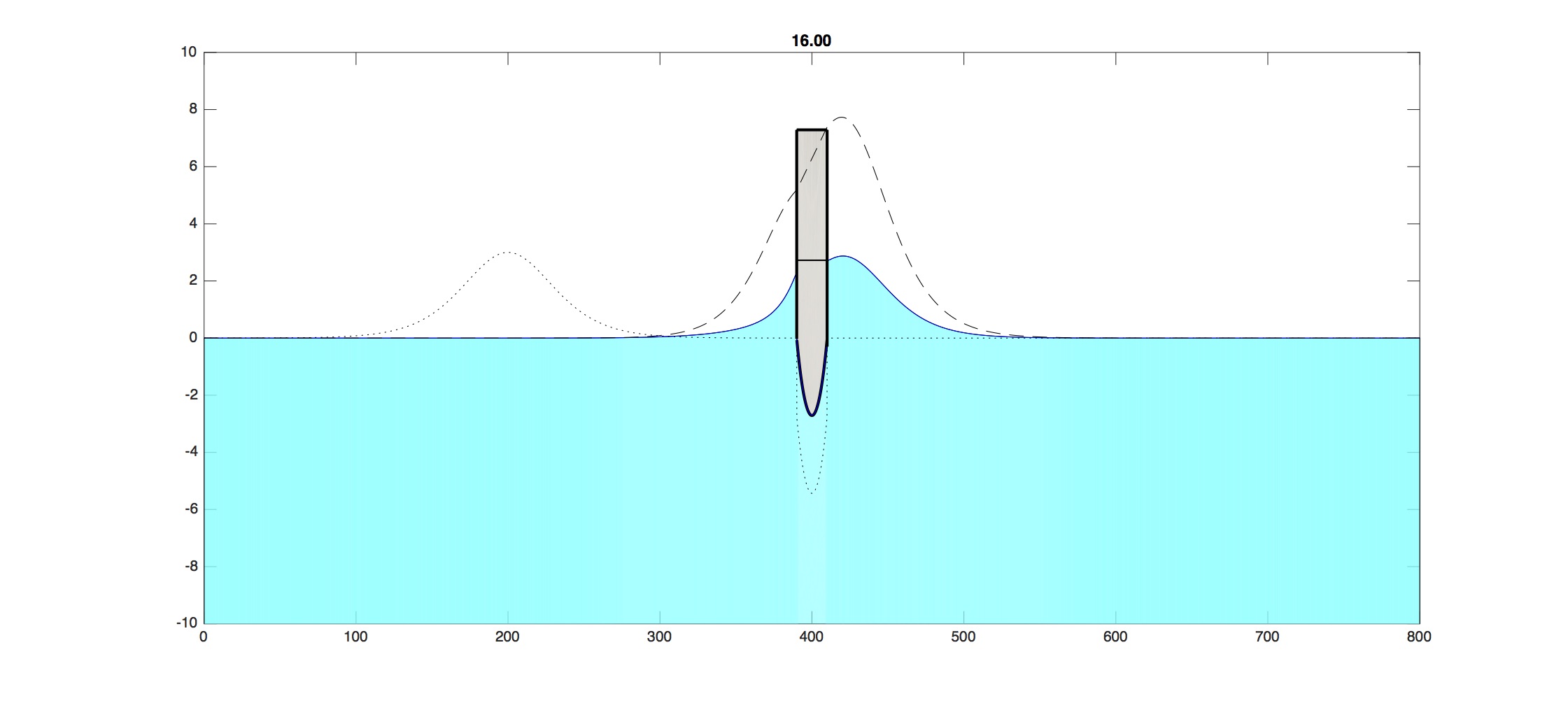}
\includegraphics[width=0.32\textwidth,height=3cm]{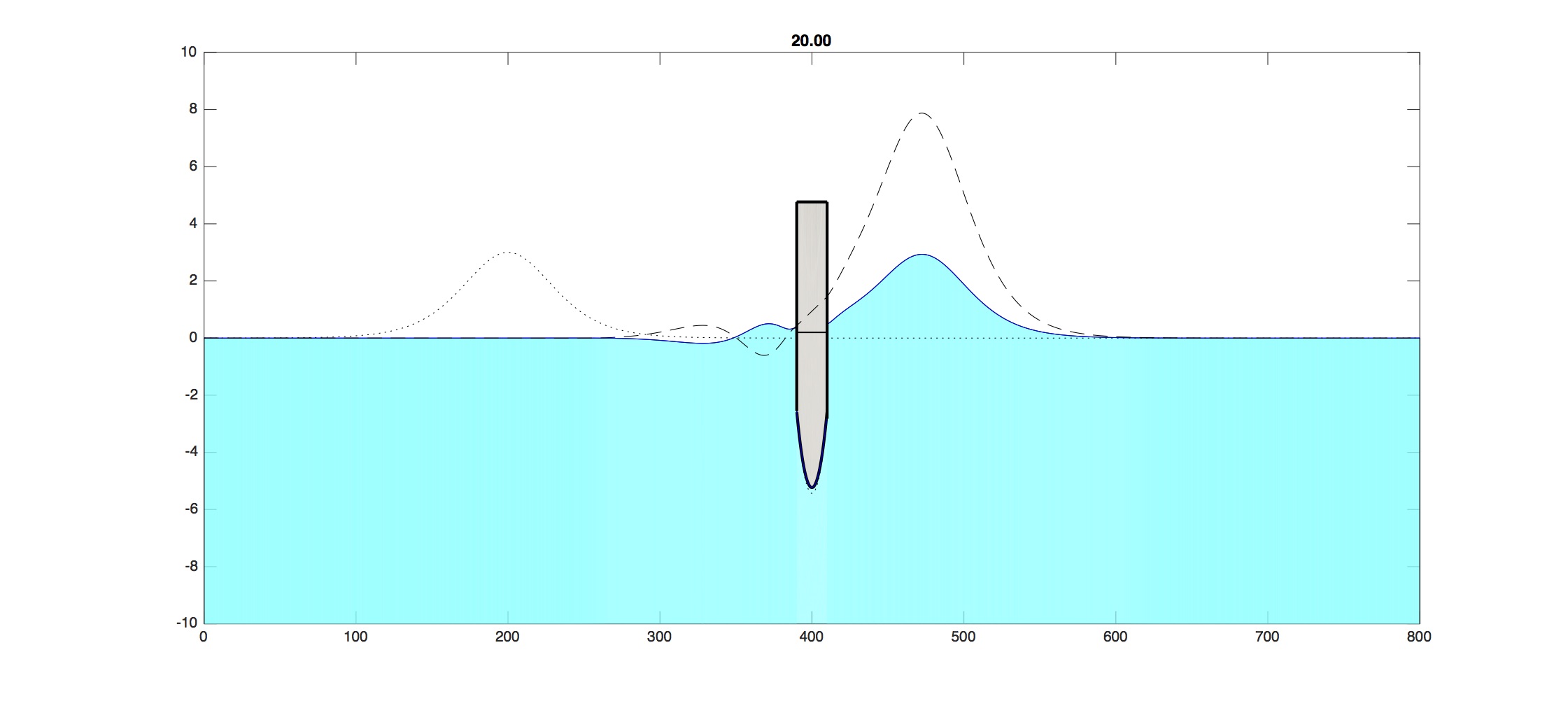}
\includegraphics[width=0.32\textwidth,height=3cm]{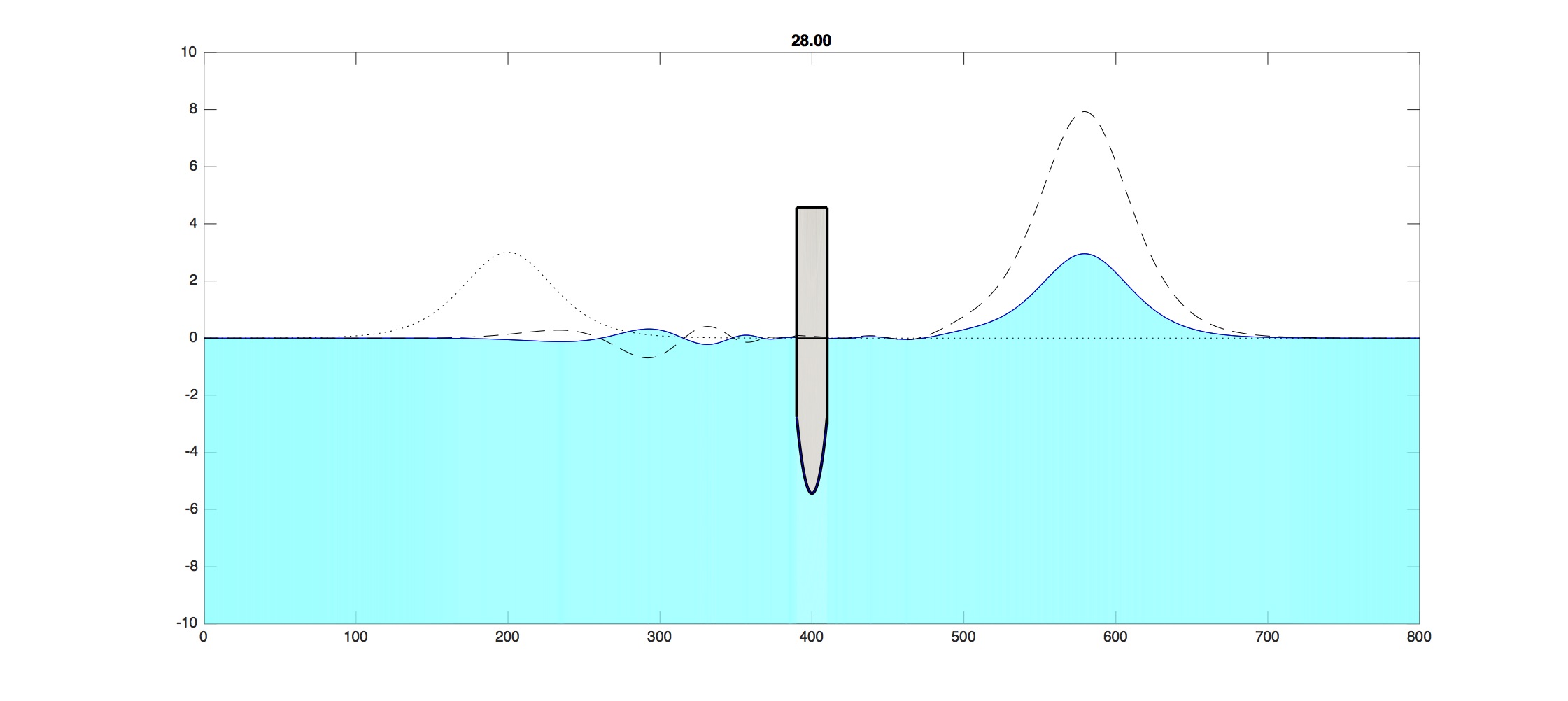}
\caption{A solitary wave arriving on a freely floating structure (Boussinesq model). Surface elevation (full), initial condition (dots) and rescaled discharge (dash). The horizontal line on the structure marks the contact line at the equilibrium.}
\label{fig_solit}
\end{figure}
\end{center}
\begin{center}
\begin{figure}
\includegraphics[width=1\textwidth]{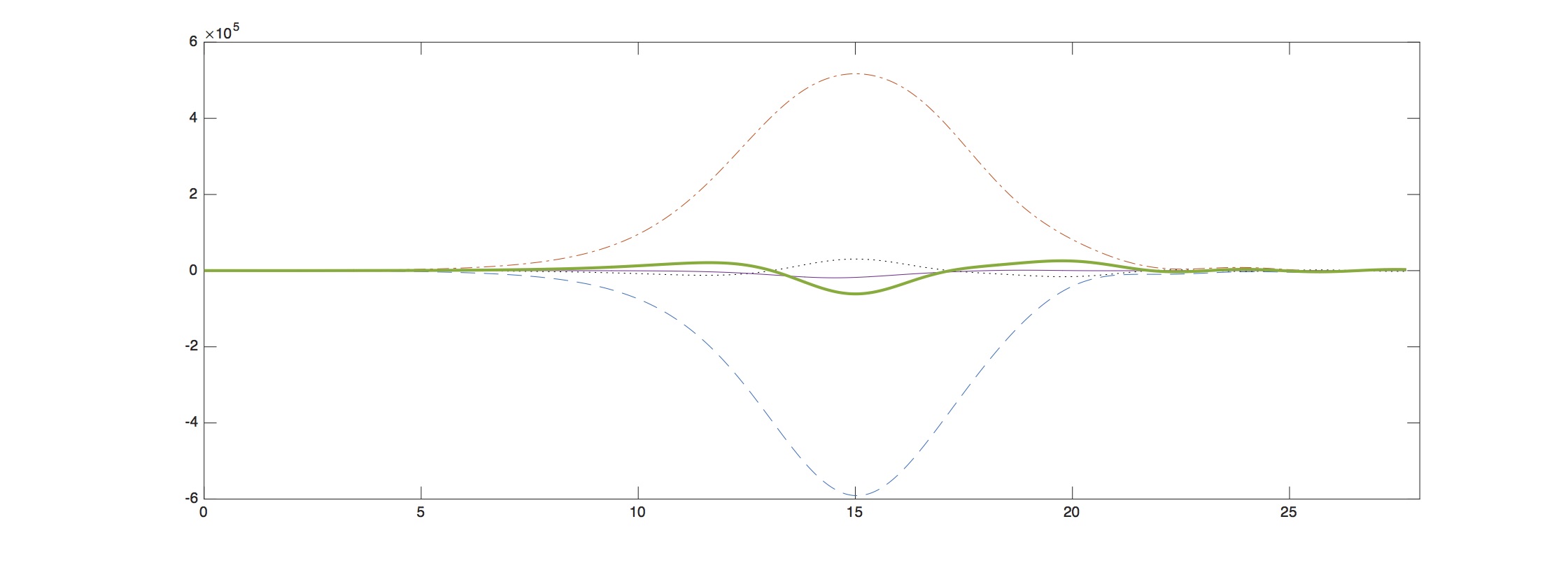}
\caption{Decomposition of the vertical force $F_{\rm vert}$ (thick solid line) exerted on the structure, as in \eqref{Fvert}. Damping-excitation force $F_{\rm D+E}$ (dash-dot), restoring force $F_{\rm restor}$ (dash), force due to added mass $F_{\rm added}$ (dots) and nonlinear corrections $F_{\rm NL}$ (solid).}
\label{fig_forces_sol}
\end{figure}
\end{center}

\begin{appendix}

\section{An alternative equation for the interior pressure}\label{Appinteriorpressure}

\subsection{Derivation of the equation}

We derived in Proposition \ref{propWWst} the elliptic equation \eqref{eqPw} for the interior pressure. If we express the time derivatives of the source term (coming from the non-hydrostatic acceleration in  ${\bf a}_{\rm FS}$) in terms of the pressure field using Euler's equation, one obtains another equation for the interior pressure. In this alternative equation, the time derivatives of the velocity field have been removed from the source term. We recall that $\uU=(\uV,\uw)$ denotes the trace at the surface of the velocity field, and that the Dirichlet-Neumann operator is as in Definition \ref{defDN}.
\begin{proposition}\label{propWWst_full}
The surface pressure $\uP$ in \eqref{Eulerav2} satisfies the following equations
$$
\begin{cases}
\dsp \uP_{\rm e}= P_{\rm atm},\\
\dsp \frac{1}{\rho}G[\zeta](\uP- P_{\rm atm})=-\dt^2\zeta_{\rm w} +\nabla\cdot \big(\nabla\cdot(h\ovV) \uV\big)-G[\zeta]\big(g\zeta+\frac{1}{2}\abs{\underline{\bU}^2}\big)\quad\mbox{ on }\quad \cI(t)\\
\uP_{\rm i}= P_{\rm atm}\quad \mbox{ on }\quad \Gamma(t).
\end{cases}
$$
\end{proposition}
\begin{proof}
Using the same notations as in the proof of Proposition \ref{propclosed}, one can write $\bU=\nabla_{X,z}\Phi$, and $\Phi$ satisfies the Bernoulli equation which, when evaluated at the surface, can be written
$$
\dt \psi -\dt \zeta \dz \Phi_{\vert_{z=\zeta}}+g\zeta+\frac{1}{2}\abs{\underline{\bU}^2}=-\frac{1}{\rho}(\uP-P_{\rm atm}).
$$
Applying the Dirichlet-Neumann operator to this equation, one gets
$$
\dt G[\zeta]\psi+\big([G[\zeta],\dt]\psi-\dt \zeta  \dz \Phi_{\vert_{z=\zeta}}\big)+G[\zeta]\big(g\zeta+\frac{1}{2}\abs{\underline{\bU}^2}\big)=-\frac{1}{\rho}G[\zeta](\uP-P_{\rm atm}).
$$
We can now use the shape derivative formula of \cite{Lannes_JAMS} to get
$$
[G[\zeta],\dt]\psi=G[\zeta]\big(\dt \zeta (\dz \Phi_{\vert_{z=\zeta}})\big)+\nabla\cdot (\dt \zeta \uV)
$$
so that
$$
\dt G[\zeta]\psi+\nabla\cdot (\dt \zeta \uV)+G[\zeta]\big(g\zeta+\frac{1}{2}\abs{\underline{\bU}^2}\big)=-\frac{1}{\rho}G[\zeta](\uP- P_{\rm atm}).
$$
Taking the restriction of this identity on the interior region $\cI(t)$ and using the constraint \eqref{constraint} then yields the result.
\end{proof}

\subsection{On the solvability of the interior pressure equation}

In the previous section, we derived an equation for the surface pressure $\uP$, namely,
\begin{equation}\label{eqPw_full2}
\frac{1}{\rho}G[\zeta](\uP- P_{\rm atm})=-\dt^2\zeta_{\rm w} +\nabla\cdot \big(\nabla\cdot(h\ovV)  \uV\big)-G[\zeta]\big(g\zeta+\frac{1}{2}\abs{\underline{\bU}^2}\big)\quad\mbox{ on }\quad \cI(t).
\end{equation}
We must show that there exists a unique solution $\uP-\uP_{atm}$ to this equations that vanishes on the exterior region ${\mathcal E}(t)$ and such that its trace on $\Gamma(t)$ also vanishes. If this is true, then the interior pressure $\uP_{\rm i}$ will simply be given by
$$
\uP_{\rm i}=P_{\rm atm}+\uP \quad \mbox{ in }\quad \cI(t).
$$
Before proving such a result, we need to introduce some functional spaces. Let us first define the spaces $H^{1/2}(\cI)$ and $\widetilde{H}^{1/2}(\cI)$ as follows.
\begin{definition}\label{defSob}
Let $\cI\subset \R^d$ be a bounded domain with Lipschitz boundary.\\
{\bf i.} We denote by $H^{1/2}(\cI)$ the Banach space consisting of the restriction to $\cI$ of all the elements of $H^{1/2}(\R^d)$, and we endow it with its canonical norm.\\
{\bf ii.} We denote by $\widetilde{H}^{1/2}(\cI)$ the set of all $f\in L^2(\cI)$ such that $\tilde f\in H^{1/2}(\R^d)$, where $\tilde f\in L^2(\R^d)$ stands for the extension of $f$ by zero outside $\cI$, and endowed with the norm
$$
\abs{f}_{\widetilde{H}^{1/2}(\cI)}=\abs{\widetilde{f}}_{H^{1/2}(\R^d)}.
$$
\end{definition}
We can now state the following proposition for nonlocal elliptic equation of the kind \eqref{eqPw_full2}. We recall that the Beppo-Levi space $\dot{H}^{1/2}(\R^d)$ is defined in \eqref{BL1}.
\begin{proposition}\label{propeqGN}
Let $\zeta,b\in W^{1,\infty}(\R^d)$ be such that $inf_{\R^d}(h_0+\zeta-b)\geq h_{\rm min}>0$. Let also $\cI\subset \R^d$ be a bounded domain with Lipschitz boundary. Then, for all $F\in H^{1/2}(\cI)^d$ and all $g\in \dot{H}^{1/2}(\R^d)$, there exists a unique  $f\in \widetilde{H}^{1/2}(\cI)$ such that
$$
G[\zeta]\widetilde{f}=\nabla\cdot {F}+G[\zeta]{g} \quad \mbox{ on }\quad \cI.
$$
Moreover, one has
$$
\abs{\widetilde{f}}_{{\dot{H}}^{1/2}(\R^d)}\leq C\big(\frac{1}{h_{\rm min}},\abs{\zeta,b}_{W^{1,\infty}}\big)\big(\abs{F}_{{H}^{1/2}(\cI)}+\abs{{g}}_{\dot{H}^{1/2}(\R^d)}\big).
$$
\end{proposition}
\begin{remark}
With the notations of Section  \ref{sectsolid}, one can write
 $$
 \dt \zeta_{{\rm w}}=(\bug+\bom\times {\bf r}_G)\cdot N_{\rm w},
$$ 
which can be put in divergence form as follows
 $$
 \dt \zeta_{\rm w}=\nabla\cdot \big(-(\zeta_{\rm w}-z_G)V_G+\frac{1}{d}w_g (X-X_G)  +\frac{1}{2}\abs{{\bf r}_G}^2\bom_h^\perp -\omega_v (\zeta_{\rm w}-z_G)(X-X_G)^\perp\big).
 $$
 Time differentiating this expression, the term $\dt^2 \zeta_{\rm w}$ is also in divergence form, and the right-hand-side of \eqref{eqPw_full2} can be put under the form $\nabla\cdot F+G[\zeta]g$ for some $F$ and $g$. If $F$ and $g$ have the required regularity, then  the proposition implies that there exists a unique solution $\uP- P_{\rm atm}\in \tilde H^{1/2}(\R^d)$ to \eqref{eqPw_full2}.
 \end{remark}
\begin{remark}
A solution $\uP- P_{\rm atm}\in \tilde H^{1/2}(\R^d)$ to \eqref{eqPw_full2} clearly solves the first two equation of the system derived in Proposition \ref{propWWst_full}. Using a standard characterization of the $\tilde H^{1/2}(\cI)$ spaces (see Lemma 1.3.2.6 in \cite{Grisvard}), the third one, namely, the continuity condition $\uP_{\rm i}- P_{\rm atm}=0$ on $\Gamma$, is satisfied in the following sense
$$
\int_{\cI}\abs{\uP_{\rm i}(X)- P_{\rm atm}}^2\frac{1}{d(X;\Gamma)}dX<\infty,
$$
where $d(X;\Gamma)$ denotes the distance between $X$ and the boundary $\Gamma$ of $\cI$.
\end{remark}

\begin{proof}
Denoting as usual by $\Omega\subset{\R^{d+1}}$ the domain delimited from above by $\{z=\zeta\}$ and from below by $\{z=-h_0+b\}$, let us first define $H^1_{0,{\mathcal I}}(\Omega)$ as the completion for the canonical norm of $H^1(\Omega)$ of the set of all $C^\infty(\overline{\Omega})$ functions with support in $\Omega\cup \{z=-h_0+b\}\cup \{z=\zeta(X),X\in{\mathcal I}\}$.
The main step is to show that there exists a unique $\Phi\in H^1_{0,{\mathcal I}}(\Omega)$ such that 
 the following variational identity,
$$
\forall \varphi \in H^1_{0,{\mathcal I}}(\Omega),\qquad
\int_\Omega \nabla_{X,z}\Phi\cdot \nabla_{X,z}\varphi=\int_{\cI} \nabla\cdot F\varphi_{\vert_{z=\zeta}}+\int_{\cI}(G[\zeta]g )\varphi_{\vert_{z=\zeta}}.
$$
By Poincar\'e inequality, the left-hand side defines a continuous and coercive bilinear form on $H^1_{0,{\mathcal I}}(\Omega)$; moreover, one has $\varphi_{\vert_{z=\zeta}}\in \widetilde{H}^{1/2}(\cI)$, and we know that $\partial_j$ ($1\leq j\leq d$) maps $H^{1/2}(\cI)$ into the dual of $\widetilde{H}^{1/2}(\cI)$ (see Remark 1.4.4.7 in \cite{Grisvard}); similarly, we know that $G[\zeta]$ also maps  $H^{1/2}(\cI)$ into the dual of $\widetilde{H}^{1/2}(\cI)$ (Proposition 3.3 in \cite{L_book}). The right-hand-side of the above variational identity therefore defines a continuous linear form on $\widetilde{H}^{1/2}(\cI)$, and consequently (by the trace theorem), on $\varphi \in H^1_{0,{\mathcal I}}(\Omega)$. The result follows therefore from Lax-Milgram's theorem.\\
 It then follows from the definition of the Dirichlet-Neumann operator that $f:=\Phi_{\vert_{z=\zeta}}$ furnishes a solution to the equation $G[\zeta]f=\nabla\cdot F$; the uniqueness of the solution easily follows from the coercivity property.\\
 Finally, the estimate is obtained upon multiplying the equation by $\widetilde{f}$ and using Cauchy-Schwarz and the following inequalities (\cite{L_book}, Proposition 3.12), for all $\psi \in \dot{H}^{1/2}(\R^d)$,
 \begin{align*}
 (\psi,G[\zeta]\psi)&\leq  C\big(\frac{1}{h_{\rm min}},\abs{\zeta,b}_{W^{1,\infty}}\big)  \abs{\psi}_{\dot{H}^{1/2}}^2\\
  \abs{\psi}_{\dot{H}^{1/2}}^2& \leq  C\big(\frac{1}{h_{\rm min}},\abs{\zeta,b}_{W^{1,\infty}}\big)  (\psi,G[\zeta]\psi).
 \end{align*}
 \end{proof}
 
\subsection{An alternative formulation for the motion of the solid structure}\label{appadded}

Proceeding as for Proposition \ref{proppresc} -- and with the same notations -- but using the formulation of Proposition \ref{propWWst_full} for the interior pressure, we can decompose $\uP$ as
$$
\uP={\mathtt P}^{\rm I}+{\mathtt P}^{\rm II}+{\mathtt P}^{\rm III}
$$
where ${\mathtt P}^{\rm I}- P_{\rm atm}$, ${\mathtt P}^{\rm II}$ and ${\mathtt P}^{\rm III}$ vanish on the exterior domain ${\mathcal E}(t)$ and  satisfy the following equations in the interior region $\cI(t)$,
\begin{align*}
\frac{1}{\rho}G[\zeta]({\mathtt P}^{\rm I}- P_{\rm atm})&=\nabla\cdot \big(\nabla\cdot(h\ovV)  \uV\big)-G[\zeta]\big(g\zeta+\frac{1}{2}\abs{\underline{\bU}^2}\big)\\
\frac{1}{\rho}G[\zeta]{\mathtt P}^{\rm II}-&=-(\dbug+\dot\bom\times {\bf r}_G)\cdot N_{\rm w} \\
\frac{1}{\rho}G[\zeta]{\mathtt P}^{\rm III}-&={\mathcal Q}[{\bf r}_G](V_H,\bom),
\end{align*}
together with the continuity conditions ${\mathtt P}^{\rm I}_{\rm i}- P_{\rm atm}={\mathtt P}^{\rm II}_{\rm i}={\mathtt P}^{\rm III}_{\rm i}=0$ on $\Gamma(t)$. The force-torque corresponding to ${\mathtt P}^{\rm I}$ and ${\mathtt P}^{\rm III}$ is then given by
\begin{equation}\label{defFtilde}
\left(\begin{array}{c}{\mathtt F}\\ {\mathtt T}\end{array}\right)=\left(\begin{array}{c} \dsp \int_{\cI(t)} ({\mathtt P}^{\rm I}+{\mathtt P}^{\rm III}- P_{\rm atm}) N_{\rm w} \\
\dsp \int_{\cI(t)}  ({\mathtt P}^{\rm I}+{\mathtt P}^{\rm III}- P_{\rm atm}){\bf r}_G \times N_{\rm w}
\end{array}\right).
\end{equation}
 In order to exhibit the added mass effect associated to ${\mathtt P}^{\rm II}$, we need to work with different elementary potentials than those introduced in Definition \ref{defelempot} -- note that the existence of these elementary potentials in $\tilde{H}^{1/2}(\cI)$ is provided by Proposition \ref{propeqGN}.
\begin{definition}\label{defelempot_full}
Under the same assumptions and with the same notations as in Definition \ref{defelempot}, we  define the elementary potentials ${\Psi}_{\mathcal I}^{(j)}$ ($j=1,\dots,6$) as the unique solutions in $\tilde{H}^{1/2}(\cI)$ of the boundary value problems, for $j=1,2,3$,
$$
\begin{cases}
G[\zeta] {\Psi}^{(j)}_{{\mathcal I}}=(N_{\rm w})_j \quad \mbox{on}\quad {\mathcal I}\\
{\Psi}^{(j)}_{{\mathcal I}}\,_{\vert_{\Gamma}}=0
\end{cases}
\mbox{and }\quad
\begin{cases}
G[\zeta] {\Psi}^{(j+3)}_{{\mathcal I}}=({\bf r}_G\times N_{\rm w})_j \quad \mbox{on}\quad {\mathcal I}\\
{\Psi}^{(j+3)}_{{\mathcal I}}\,_{\vert_{\Gamma}}=0.
\end{cases}
$$
\end{definition}
We can then give the following alternative equation to the one given in Proposition \ref{propfloat} for the motion of a freely floating body. The proof is a straightforward adaptation and is omitted.
\begin{proposition}\label{propfloat_full}
Under the same assumptions as in Proposition \ref{propfloat} and with the same notations,  the velocity $\bug$ of the center of mass and the angular velocity $\bom$ satisfy the ODE 
$$\big({\mathcal M}+{\mathtt M}_{\rm a}[h,\boldsymbol{\Psi}_{{\mathcal I}}]\big)\left(\begin{array}{c} {\dbug} \\ \dot{\bom} \end{array}\right)=\left(\begin{array}{c} -{\mathfrak m}g {\bf e}_z \\
{\mathfrak I}\bom\times \bom 
\end{array}\right)
+{\mathtt F},
$$
with ${\mathtt F}$ as in \eqref{defFtilde} and with the added mass-inertia matrix given by
$$
{\mathtt M}_{\rm a}[h,\boldsymbol{\Psi}_{{\mathcal I}}]:=\rho \big( \int_{\R^d} \widetilde{\Psi}_\cI^{(j)}G[\zeta]\widetilde{\Psi}_\cI^{(k)}\big)_{1\leq j,k\leq 6},
$$
and where we recall that $\widetilde{\Psi}_\cI^{(j)}$ denotes the extension by $0$ outside $\cI$.
\end{proposition}
\begin{remark}
The most important thing to insist on is that the added mass-inertia matrix ${\mathtt M}_{\rm a}[h,\boldsymbol{\Psi}_{{\mathcal I}}]$ differs from the added mass inertia matrix ${\mathcal M}_{\rm a}[h,\boldsymbol{\Phi}_{{\mathcal I}}]$ exhibited in Proposition \ref{propfloat}. This is because some
of the components of the resulting force ${\mathcal F}[h,\boldsymbol{\Phi}_\cI](S_{\rm i}^{\rm I}+S_{\rm i}^{\rm III})$ in Proposition \ref{propfloat} can be put as a (possibly negative) added mass-inertia term. 
\end{remark}
\begin{remark}
One can see the added mass-inertia matrix ${\mathcal M}_{\rm a}[h,\boldsymbol{\Phi}_{{\mathcal I}}]$ of Proposition \ref{propfloat} as a shallow water approximation of ${\mathtt M}_{\rm a}[h,\boldsymbol{\Psi}_{{\mathcal I}}]$. Indeed, it is known that in shallow water, one has
at leading order (see Proposition 3.8 in \cite{AL1} or \S 3.6 in \cite{L_book}),
$$
G[\zeta]\psi\sim -\nabla\cdot (h\nabla \psi)
$$
Replacing $G[\zeta]$ by this expression in Definition \ref{defelempot_full}, one recovers the same elementary potentials introduced in Definition \ref{defelempot_full}, and doing the same substitution in the definition of ${\mathtt M}_{\rm a}[h,\boldsymbol{\Psi}_{{\mathcal I}}]$, one recovers ${\mathcal M}_{\rm a}[h,\boldsymbol{\Phi}_{{\mathcal I}}]$.
\end{remark}
\section{Floating structure equations in a body frame}\label{appbodyframe}

In Proposition \ref{propfloat} we gave a formulation of the water waves equations with a freely floating body where the equations for the motion of the floating structure are given in the Eulerian frame ${\mathfrak E}=(Oxyz)$. Another natural possibility is to use a system of coordinates moving with the rigid body, whose axis are the principal axes of inertia of the body, and whose origin is the center of mass. We denote by ${\mathfrak B}(t)=(G(t)x'y'z')$ this body frame. \\
Since both ${\mathfrak E}$ and the body frame ${\mathfrak B}(0)$ at $t=0$ are orthogonal, there exists a rotation matrix $\Theta_0\in SO(3)$ sending the unit directional vectors  $({\bf e}_x,{\bf e}_y,{\bf e}_z)$ of ${\mathfrak E}$ to their counterparts $({\bf e}^0_{x'},{\bf e}^0_{y'},{\bf e}^0_{z'})$ in ${\mathfrak B}(0)$.  If $A$ and $A'$ represent the coordinates of some vector in ${\mathfrak E}$ and ${\mathfrak B}(0)$ respectively, one has therefore
$$
A=\Theta_0 A'\quad\mbox{ and }\quad
{\mathfrak I}(0)=\Theta_0{\mathfrak I}_0\Theta_0^T
\quad\mbox{ with }\quad
{\mathfrak I}_0=\mbox{diag}({\mathfrak i}_1,{\mathfrak i}_2,{\mathfrak i}_3),
$$
the scalars ${\mathfrak i}_j$ ($j=1,2,3$) denoting the principal moment of inertia of the solid. Combining this with \eqref{eqI} and \eqref{eqTheta}, the above relations for the change of frame become at time $t$ follow
\begin{equation}\label{bodyframe}
A=\widetilde\Theta(t) A'\quad\mbox{ and }\quad
{\mathfrak I}(t)=\widetilde\Theta(t){\mathfrak I}_0\widetilde\Theta(t)^T
\quad\mbox{ with }\quad
\widetilde\Theta(t)=\Theta(t)\Theta_0
\end{equation}
(we recall that $\Theta(t)$ stands for the rotation matrix defined by \eqref{eqTheta}). \\
The main advantage in working in the body frame is that the mass-inertia matrix becomes independent of time; we denote it $\underline{{\mathcal M}}_0$, with
$$
\underline{{\mathcal M}}_0=\mbox{{\rm diag}}({\mathfrak m},{\mathfrak m},{\mathfrak m},{\mathfrak i}_1,{\mathfrak i}_2,{\mathfrak i}_3).
$$
For the added mass matrix and the source terms, we also need to replace the elementary potential $\boldsymbol{\Phi}_\cI$ by $\boldsymbol{\Phi}'_\cI$, where for all $1\leq j\leq 3$,
$$
{\Phi_\cI'}^{(j)}=\Phi_\cI^{(j)}
\quad \mbox{and }\quad
\begin{cases}
-\nabla\cdot h\nabla {\Phi'}^{(j+3)}_{{\mathcal I}}=({\bf r}'_G\times N'_{\rm w})_j \quad \mbox{on}\quad {\mathcal I}\\
{\Phi'}^{(j+3)}_{{\mathcal I}}\,_{\vert_{\Gamma}}=0
\end{cases}
$$
(i.e. ${\Phi_\cI'}^{(j+3)}$ is defined using the $j$-th coordinates of ${\bf r}_G\times N_{\rm w}$ in the body frame rather than in the inertial frame). The equations of motion for the floating structure can then be written as follows.
\begin{proposition}\label{propfloatbis}
The equations for  the velocity $\bug$ of the center of mass and the angular velocity $\bom$ given in Proposition \ref{propfloat} can be replaced by
$$\big(\underline{\mathcal M}_0+{\mathcal M}_{\rm a}[h,\boldsymbol{\Phi}'_{{\mathcal I}}]\big)\left(\begin{array}{c} {\dbug} \\ \dot{\bom} \end{array}\right)=\left(\begin{array}{c} -{\mathfrak m}g {\bf e}_z \\
{\mathfrak I}_0\bom\times \bom 
\end{array}\right)
+{\mathcal F}[h,\boldsymbol{\Phi}'_{\mathcal I}](S_{\rm i}^{\rm I}+S_{\rm i}^{\rm III}\big).
$$
\end{proposition}
\begin{proof}
We show how to rewrite the equation \eqref{Newton2} for the angular momentum in the body frame. The adaptations for the equation \eqref{Newton1} for the linear momentum are similar and therefore omitted. 
Owing to \eqref{bodyframe}, one has
\begin{align*}
\frac{d}{dt}({\mathfrak I}\bom)=&\frac{d}{dt}(\widetilde \Theta {\mathfrak I}_0 \bom')\\
=&\dot\Theta \Theta_0{\mathfrak I}_0 \bom'+\widetilde\Theta{\mathfrak I}_0 \dot{\bom}';
\end{align*}
with \eqref{eqTheta}, this gives
$$
\frac{d}{dt}({\mathfrak I}\bom)=\bom\times(\widetilde\Theta {\mathfrak I}_0\bom')+\widetilde\Theta{\mathfrak I}_0 \dot{\bom}'.
$$
Multiplying \eqref{Newton2} on the left by $\widetilde\Theta^T$, we obtain therefore
$$
{\mathfrak I}_0\dot{\bom}'+\bom'\times {\mathfrak I}_0\bom'=T'_{\rm fluid}
$$
with 
\begin{align*}
T'_{\rm fluid}=&\int_{\cI}(\uP_{\rm i}-P_{\rm atm})\widetilde\Theta^T ({\bf r}_G\times N_{\rm w})\\
=&\int_{\cI}(\uP_{\rm i}-P_{\rm atm}) {\bf r}'_G\times N'_{\rm w}.
\end{align*}
Using the definition of $\boldsymbol{\Phi}'_\cI$, and integrating by parts, we therefore have
\begin{align*}
T'_{\rm fluid}\cdot {\bf e}'_j&=-\int_\cI \nabla \cdot h \nabla \uP_{\rm i}^{\rm II}{\Phi_\cI'}^{(j+3)}-\rho\int_{\cI}(S_{\rm i}^{\rm I}+S_{\rm i}^{\rm III})\cdot \nabla {\Phi_{\cI}'}^{(j+3)}\\
&=-\rho \int_\cI (\dbug+\dot\bom\times {\bf r}_G)\cdot N_{\rm w} {\Phi_\cI'}^{(j+3)}+\big({\mathcal F}[h,\boldsymbol{\Phi}'_{\mathcal I}](S_{\rm i}^{\rm I}+S_{\rm i}^{\rm III})\big)_{j+3},
\end{align*}
with $S_{\rm i}^{\rm I}$, $S_{\rm i}^{\rm II}$ and $S_{\rm i}^{\rm III}$ as in Proposition \ref{proppresc}. Remarking that 
\begin{align*}
(\dot\omega\times {\bf r}_G)\cdot N_{\rm w}=&({\bf r}'_G\times N_{\rm w}')\cdot \dot\bom'\\
=-&\sum_{k=1}^3 \nabla\cdot (h \nabla {\Phi'_{\cI}}^{(k+3)}){\bf e}'_k\cdot \dot\bom',
\end{align*}
one can conclude the proof as for Proposition \ref{propfloat}.
\end{proof}

\end{appendix}

\nomenclature[ga ]{$d$}{Horizontal dimension}
\nomenclature[gb ]{$X=(x,y)$}{Horizontal variable if $d=2$; $X=x$ if $d=1$}
\nomenclature[gb ]{$z$}{Vertical variable $z$}
\nomenclature[gc ]{$\nabla$}{Gradient operator with respect to the horizontal variables}
\nomenclature[gc ]{$\nabla^\perp=(-\dy,\dx)^T$}{Orthogonal horizontal gradient operator}
\nomenclature[gc ]{$\nabla_{X,z}$}{Full gradient operator}
\nomenclature[gd ]{$d$}{Horizontal dimension, $d=1,2$}
\nomenclature[gd ]{$f_{\rm e}$}{Restriction of the function $f$ to the exterior region ${\mathcal E}(t)$}
\nomenclature[gd ]{$f_{\rm i}$}{Restriction of the function $f$ to the interior region ${\mathcal I}(t)$}
\nomenclature[gd ]{$\av{f}$}{The average part of $f$, see Notation \ref{notav}}
\nomenclature[gd ]{$f^*$}{The oscillating part of $f$, see Notation \ref{notav}}
\nomenclature[ge ]{$\mbox{\rm Var}(f)$}{The variance of $f$, see Notation \ref{notavar}}
\nomenclature[ge ]{${\bf A}=({\bf A}_{\rm h},A_{\rm v})$}{${\bf A}_h$ and $A_v$  are the horizontal and vertical components of ${\bf A}$}
\nomenclature[gf ]{${\bf A}_h^\perp$}{Given by $(-A_2,A_1)$ if ${\bf A}_{\rm h}=(A_1,A_2)$}.
\nomenclature[gg ]{$\underline{A}$}{The trace of ${\bf A}$ on the surface of the fluid}
\nomenclature[gh ]{}{}

\nomenclature[La ]{$\alpha$}{Ratio $\delta_t/\delta_x$ of the time step and cell size}
\nomenclature[La ]{$\Gamma(t)$}{Horizontal projection of the contact line}
\nomenclature[Lb ]{$\Phi_{\cI}^{(j)}$}{Elementary potential, see Definition \ref{defelempot}}
\nomenclature[Lc ]{$\boldsymbol{\Phi}_{\cI}$}{Vectors whose components are the elementary potentials $\Phi_{\cI}^{(j)}$}
\nomenclature[Lc ]{$\zeta$}{Surface elevation}
\nomenclature[Lc ]{$\zeta_{\rm w}$}{Parametrization of the bottom of the solid}
\nomenclature[Ld ]{$\Theta$}{Rotation matrix, see \eqref{eqTheta}}
\nomenclature[Lda ]{$\nu$}{Unit normal vector to the $d$-dimensional curve $\Gamma(t)$}
\nomenclature[Le ]{$\bom=(\bom_{\rm h},\omega_{\rm v})$}{Angular velocity}
\nomenclature[Le ]{$\omega$}{Angular velocity in dimension $d=1$}
\nomenclature[Le ]{$\Omega(t)$}{Domain occupied by the fluid at time $t$}
\nomenclature[Le ]{$\Omega(t)$}{Domain occupied by the fluid at time $t$}

\nomenclature[Lf ]{${\bf a}_{\rm FS}$}{Acceleration of the free surface, see \eqref{defba}}
\nomenclature[Lf ]{${\bf a}_{\rm NH}$}{Non-hydrostatic accelaration, see \eqref{defFNH}}
\nomenclature[Lg ]{${\mathfrak A}[\zeta ]$}{Average mapping, see Prop. \ref{propclosed}}
\nomenclature[Lh ]{$b$}{$z=-h_0+b(X)$ is the bottom parametrization}
\nomenclature[Li ]{${\mathcal C}(t)$}{Domain occupied by the solid at time $t$}
\nomenclature[Li ]{${\mathcal E}(t)$}{Exterior region at time $t$}
\nomenclature[Lj ]{$F_{\rm fluid}$}{Force exerted by the fluid on the solid, see \eqref{forcetorque}}
\nomenclature[Lk ]{$ {\mathcal F}[h,{\boldsymbol{\Phi}_\cI}]S_{\rm i}$}{Force-torque generated by $S_{\rm i}$, see Definition \ref{defmassiner}}
\nomenclature[Ll ]{$ \widetilde{\mathcal F}[h,{\bf r}_G]S_{\rm i}$}{Force-torque generated by $S_{\rm i}$ if $d=1$, see \S \ref{sect1dfloat}}
\nomenclature[Lm ]{$G=(X_G,z_G)$}{Position of the center of mass}
\nomenclature[Lm ]{$G[\zeta]$}{Dirichlet-Neumann operator, see Definition \ref{defDN}}
\nomenclature[Ln ]{$h(t,X)$}{Water depth $h=h_0+\zeta(t,X)-b(X)$}
\nomenclature[Ln ]{$h_0$}{Reference water depth at rest}
\nomenclature[Lo ]{$H^{1/2}(\cI), \widetilde{H}^{1/2}(\cI)$}{Sobolev spaces on the bounded domain $\cI$, see Definition \ref{defSob}}
\nomenclature[Lo ]{$\dot{H}^s(\R^d), \dot{H}^1(\Omega)$}{Beppo-Levi spaces, see \eqref{BL1}}
\nomenclature[Lp ]{${\mathcal I}(t)$}{Interior region at time $t$}
\nomenclature[Lq ]{$ {\mathfrak i}_0$}{Inertia coefficient if $d=1$}
\nomenclature[Lq ]{$ {\mathfrak I}(t)$}{Inertia matrix of the solid, see \eqref{eqI}}
\nomenclature[Lq ]{$ {\mathfrak m}$}{Mass of the solid}
\nomenclature[Lr ]{$ {\mathcal M}(t)$}{Mass-inertia matrix of the solid, see Definition \ref{defmassiner}}
\nomenclature[Lr ]{$ {\mathcal M}_0$}{Mass-inertia matrix of the solid when $d=1$}
\nomenclature[Ls ]{$ {\mathcal M}_{\rm a}[h,{\boldsymbol{\Phi}_\cI}]$}{Added mass-inertia matrix of the solid, see Definition \ref{defmassiner}}
\nomenclature[Ls ]{$ \widetilde{\mathcal M}_{\rm a}[h,{\bf r}_G]$}{Added mass-inertia matrix if $d=1$, see \S \ref{sect1dfloat}}
\nomenclature[Lt ]{$N$}{Non unit upward normal vector at the surface $N=(-\nabla\zeta,1)$}
\nomenclature[Lt ]{$N_{\rm w}$}{Non unit upward normal vector on the wetted surface}
\nomenclature[Lt ]{$N_{\rm b}$}{Non unit upward normal vector at the bottom $N_{\rm b}=(-\nabla b,1)$}
\nomenclature[Lt ]{$P$}{Pressure field}
\nomenclature[Lu ]{$\uP_{\rm i}^{\rm I},\uP_{\rm i}^{\rm II},\uP_{\rm i}^{\rm III}$}{Decomposition of the interior pressure, see Prop. \ref{proppresc}}
\nomenclature[Lv ]{$\uP_{\rm i}^{\rm IV}$}{Additional pressure for vertical walls, see Prop.  \ref{proppresc_vert}}
\nomenclature[Lw ]{$P_{\rm atm}$}{Atmospheric pressure (constant)}
\nomenclature[Lw ]{$P_{\rm {NH}}$}{Non-hydrostatic pressure, see \eqref{PNH}}
\nomenclature[Lx ]{$q$}{Horizontal discharge when $d=1$}
\nomenclature[Lx ]{$Q$}{Horizontal discharge}
\nomenclature[Lx ]{${\mathcal Q}[{\bf r}_G]$}{Quadratic source term for the interior pressure, see \S \ref{sectpresc}}
\nomenclature[Ly ]{${\bf r}_G$}{Position with respect to the center of mass, see \eqref{defUc}}
\nomenclature[Ly ]{${\mathbf R}[h,Q]$}{Reynolds tensor, see \eqref{defRey}}
\nomenclature[Lya ]{${\mathfrak R}[\zeta ]$}{Reconstruction mapping, see Prop. \ref{propclosed}}
\nomenclature[Lz ]{$\bU$}{Velocity field for the fluid}
\nomenclature[Lz ]{${\mathcal U}_{\mathcal C}=(V_{\mathcal C},w_{\mathcal C})$}{Velocity field for the solid}
\nomenclature[Lza ]{$\bU_G=(V_G,w_G)$}{Velocity of the center of mass}
\nomenclature[Lza ]{$\underline{U}_{{\rm w}}$}{Velocity of the solid on the wetted surface}
\nomenclature[Lzb ]{$\underline{U}_{{\rm w,\tau}}$}{See \S \ref{sectpresc}}
\nomenclature[Lzc ]{$V$}{Horizontal component of the velocity field}
\nomenclature[Lzz ]{$\ovV$}{Vertically averaged horizontal velocity, see \eqref{defavV}}
\nomenclature[Lzz ]{$w$}{Vertical component of the velocity field}
\nomenclature[Lzzz ]{$x_\pm(t)$}{Boundaries of the interior region in dimension $d=1$}

\newpage
\printnomenclature[3cm]

\end{document}